\documentclass[11pt]{article}

  \usepackage{latexsym}
  \usepackage{comment}
  \usepackage{hyperref}
  \usepackage{amsmath,amsthm,amsfonts,amssymb,graphicx,epsfig,latexsym,color,amssymb}
  \usepackage{amscd,latexsym,esint,upref,stmaryrd,enumerate,verbatim,yfonts,bm,mathtools,comment}
  \usepackage{musicography}

\usepackage{indentfirst} 
  \usepackage{epsf,enumerate}
  \usepackage{tikz}
  \usepackage{mathtools}
  \usepackage{outlines}
  \usepackage{float}
  \usepackage [english]{babel}
  \usepackage[utf8]{inputenc}
\usepackage [autostyle, english = american]{csquotes}
\MakeOuterQuote{"}

\usepackage{enumitem} 

\DeclareMathAlphabet{\mathbbold}{U}{bbold}{m}{n}
\usepackage[
backend=biber,
style=alphabetic,
sorting=nty
]{biblatex}

\newtheorem{manualtheoreminner}{Theorem}
\newenvironment{manualtheorem}[1]{%
  \renewcommand\themanualtheoreminner{#1}%
  \manualtheoreminner
}{\endmanualtheoreminner}

  \newtheorem{innercustomgeneric}{\customgenericname}
\providecommand{\customgenericname}{}
\newcommand{\newcustomtheorem}[2]{%
  \newenvironment{#1}[1]
  {%
   \renewcommand\customgenericname{#2}%
   \renewcommand\theinnercustomgeneric{##1}%
   \innercustomgeneric
  }
  {\endinnercustomgeneric}
}

\newcustomtheorem{customthm}{Theorem}
\newcustomtheorem{customlemma}{Lemma}
\newcustomtheorem{customprop}{Proposition}

 \usepackage{theoremref}

  \def\R{{\mathbb R}}

  \def\Z{{\mathbb Z}}
  \def\N{{\mathbb N}}
    
  \def\G{{\Gamma}}

  \renewcommand{\>}{\rangle}

\newtheorem{theorem}{Theorem}
\numberwithin{theorem}{section}
\newtheorem{proposition}[theorem]{Proposition}
\newtheorem{lemma}[theorem]{Lemma}
\newtheorem{corollary}[theorem]{Corollary}
\newtheorem{conjecture}[theorem]{Conjecture}
\newtheorem{question}[theorem]{Question}

\newtheorem{claim}[theorem]{Claim}
\theoremstyle{definition}
\newtheorem{definition}[theorem]{Definition}
\newtheorem{remark}[theorem]{Remark}

\theoremstyle{remark}

\newtheorem{fact}[theorem]{Fact}

\title{The sphere complex of a locally finite graph}
\author{Brian Udall}
\date{October 2024}

\begin{document}

\maketitle

\begin{abstract}
    For a locally finite graph $\Gamma$, we consider its mapping class group Map$(\Gamma)$ as defined by Algom-Kfir--Bestvina. For these groups, we prove a generalization of the results of Laudenbach and Brendle--Broaddus--Putman, producing a $3$-manifold $M_{\Gamma}$ whose mapping class group surjects onto Map$(\Gamma)$ with kernel a compact abelian group of sphere twists so that the corresponding short exact sequence splits. Along the way we obtain an induced faithful action of $\text{Map}(\Gamma)$ on the sphere complex $\mathcal{S}(M_{\Gamma})$ of $M_{\Gamma}$, which is the simplicial complex whose simplices are isotopy classes of finite collections of spheres in $M_{\Gamma}$ which are pairwise disjoint. When $\Gamma$ has finite rank, we further show that the action of $\text{Map}(\Gamma)$ on a certain natural subcomplex has elements with positive translation length, and also consider a candidate for an Outer space of such a graph. As another application, we prove that for many $\Gamma$, $\text{Map}(\Gamma)$ is quasi-isometric to a particular subgraph of $\mathcal{S}(M_{\Gamma})$, following Schaffer-Cohen. We also deduce analogs of the results of Domat--Hoganson--Kwak. 
\end{abstract}

\section{Introduction}

Prompted by the recent surge of interest in the mapping class groups of infinite-type surfaces, called big mapping class groups, Algom-Kfir--Bestvina proposed a version of "big Out$(F_n)$" \cite{algom-kfir_groups_2021}. Given a locally finite graph $\Gamma$, its mapping class group, denoted Map$(\Gamma)$, is the group of proper homotopy equivalences of $\Gamma$ up to proper homotopy. 

\par 
Essentially all of the work so far in studying mapping class groups of infinite type graphs has been in understanding the basic algebraic and topological properties of these groups, as well as their coarse geometry \cite{algom-kfir_groups_2021} \cite{domat_coarse_2023} \cite{domat_generating_2023}.
\par 
Given a locally finite connected graph $\Gamma$, we consider the associated \textit{doubled handlebody with punctures} $M_{\Gamma}$. This is a $3$-manifold obtained by gluing two copies of a regular neighborhood of $\Gamma$ in $\R^3$ along their common surface boundary by the identity map (see Definition \ref{def:DoubledHandlebody} for more details). Then the \textit{sphere complex} $\mathcal{S}(M_{\Gamma})$ is the simplicial complex whose $k$-simplices are given by isotopy classes of $k+1$ essential non-peripheral smoothly embedded spheres in $M_{\Gamma}$ which can be realized disjointly. The mapping class group Map$(M_{\Gamma})$ of $M_{\Gamma}$, which is the group of orientation preserving diffeomorphisms of $M_{\Gamma}$ up to isotopy, acts naturally on this complex.
\par 
This complex is well understood in the case when $\Gamma$ is compact. Hatcher initially defined it \cite{hatcher_homological_1995}, building on results of Whitehead and Laudenbach \cite{whitehead_certain_1936, whitehead_equivalent_1936}\cite{laudenbach_sur_1973, laudenbach_topologie_1974}. Both its contractibility and the high connectivity of many of its subcomplexes has been used to great effect to produce many results on the homological properties of Aut$(F_n)$ and Out$(F_n)$, see \cite{hatcher_homological_1995}\cite{hatcher_cerf_1998}\cite{hatcher_homology_2004}\cite{hatcher_stabilization_2010} for examples of this. Their geometry is also fairly well understood, though much is still not known \cite{handel_free_2013}\cite{iezzi_sphere_2016}\cite{hilion_hyperbolicity_2017}\cite{clay_uniform_2017}\cite{handel_free_2019}\cite{hamenstadt_spotted_2023}\cite{hamenstadt_submanifold_2024}. See the beginning of Subsection \ref{SphereGraph} as well as the discussion around Conjecture \ref{conj:NSHyperb}.
\par
The main result of this paper implies that the action of $\text{Map}(M_{\Gamma})$ on $\mathcal{S}(M_{\Gamma})$ descends to a faithful action of Map$(\Gamma)$ on $\mathcal{S}(M_{\Gamma})$. Indeed, we have the following. 
\begin{theorem}\label{SESMainTheorem}
Let $\Gamma$ be a locally finite connected graph, and let $M_{\Gamma}$ denote its associated doubled handlebody with punctures.
\begin{enumerate}
    \item  There is a short exact sequence of continuous maps
    \begin{equation*}
        1 \to \mathrm{Twists}(M_{\Gamma}) \to \mathrm{Map}(M_{\Gamma}) \to \mathrm{Map}(\Gamma) \to 1
    \end{equation*}
    where $\mathrm{Twists}(M_{\Gamma})$ denotes the subgroup of $\mathrm{Map}(M_{\Gamma})$ generated by sphere twists on sphere systems of $M_{\Gamma}$. Further, as long as $\Gamma$ is not a sporadic graph, $\mathrm{Twists}(M_{\Gamma})$ is precisely the kernel of the action of $\mathrm{Map}(M_{\Gamma})$ on $\mathcal{S}(M_{\Gamma})$, so there is an induced faithful action of $\mathrm{Map}(\Gamma)$ on $\mathcal{S}(M_{\Gamma})$. 
    \item The short exact sequence splits topologically, giving an isomorphism
    \begin{equation*}
        \mathrm{Map}(M_{\Gamma}) \cong \mathrm{Twists}(M_{\Gamma})\rtimes \mathrm{Map}(\Gamma)
    \end{equation*}
    as topological groups. 
    
    \end{enumerate}
\end{theorem}
Here, a \textit{sphere system} is a smooth proper embedding of (possibly infinitely many) disjoint essential non-peripheral $2$-spheres in $M_{\Gamma}$. Unless otherwise stated, we assume that no pair of spheres are isotopic. A sphere in a $3$-manifold $M$ is \textit{essential} if it does not bound a $3$-ball, and \textit{non-peripheral} if it does not bound a once punctured $3$-ball and it is not isotopic to a boundary component. A sphere system is \textit{reduced} if its complement is simply connected. A locally finite graph $\Gamma$ is \textit{sporadic} if it is rank $0$ with at most $4$ ends, a rank $1$ graph with at most $1$ end, or a compact rank $2$ graph. See the discussion before Conjecture \ref{conj:LaudenbachExtension}.
\par 
\par 
Versions of this theorem for graphs with finitely many ends are well known. Namely, when $\Gamma$ is a finite rank graph with $0$ or $1$ end, the first part of Theorem \ref{SESMainTheorem} is due to Laudenbach \cite{laudenbach_topologie_1974}. For finite rank graphs with finitely many ends, the first part of Theorem \ref{SESMainTheorem} is due to Hatcher--Vogtmann \cite{hatcher_homology_2004} (c.f. Proposition \ref{HatcherVSES}). When $\Gamma$ is compact, the second part of Theorem \ref{SESMainTheorem} is due to Brendle--Broaddus--Putman \cite{brendle_mapping_2023}.
\par 
Even more, we can precisely determine the group Twists$(M_{\Gamma})$. Let rk$(\Gamma)$ denote the rank of the fundamental group of $\Gamma$, which is a free group of possibly infinite rank.
\begin{theorem}\label{TwistGroupStructure}
    The group $\mathrm{Twists}(M_{\Gamma})$ is topologically isomorphic to $\Pi_{i=1}^{\text{rk}(\Gamma)}\Z/2$, with generating set given by sphere twists on a reduced sphere system (allowing for infinite products). In particular, $\mathrm{Twists}(M_{\Gamma})$ is compact and abelian.
\end{theorem}

We remark that Theorems \ref{SESMainTheorem} and \ref{TwistGroupStructure} also restrict to subgroups of $\mathrm{Map}(M_{\Gamma})$. For example, one obtains a splitting 
$$\mathrm{PMap}(M_{\Gamma})\cong \mathrm{Twists}(M_{\Gamma})\rtimes \mathrm{PMap}(\Gamma)$$
of the pure mapping class group $\mathrm{PMap}(M_{\Gamma})$. Theorem \ref{TwistGroupStructure} is also a generalization of previous results \cite{laudenbach_topologie_1974}\cite{hatcher_homology_2004}.
\par 
We now note some of the applications one can obtain from the above results. First, we restrict to the case when $\Gamma$ has finite nonzero rank. In this case, we look at the action of $\text{Map}(M_{\Gamma})$ on the full subcomplex of $\mathcal{S}(M_{\Gamma})$ whose vertices consist of nonseparating spheres. We denote this complex by $\mathcal{S}_{ns}(M_{\Gamma})$. We have the following. 

\begin{manualtheorem}{7.12}
    For $\Gamma$ an infinite graph with $1\leq \mathrm{rk}(\Gamma)<\infty$, where in the rank $1$ case we assume $\Gamma$ has more than one end, the complex $\mathcal{S}_{ns}(M_{\Gamma})$ has infinite diameter. Further, the action of $\mathrm{Map}(\Gamma)$ on $\mathcal{S}_{ns}(M_{\Gamma})$ admits elements with positive translation length.
\end{manualtheorem}

The core of this proof lies in developing projection maps for nonseparating spheres (see Definition \ref{def:SphereProjection}) and mimicking the proof of the corresponding result for the nonseparating curve graph of a finite genus infinite type surface, which can be found in Section 8 of \cite{durham_graphs_2018}.
\par 
Letting $\Gamma$ continue to be a finite rank graph, we also define the \textit{Outer space} $\mathbb{O}(\Gamma)$ to be the space of finite weighted sphere systems in $M_{\Gamma}$ whose complementary components are simply connected.  This can be thought of as a subspace of $\mathcal{S}(M_{\Gamma})$, see Subsection \ref{subsec:OuterSpace}. We have the following, which is an extension of the corresponding result for $\text{Out}(F_n)$.
\begin{customprop}{7.14}
Let $\Gamma$ be a locally finite graph with finite positive rank. Then $\mathbb{O}(\Gamma)$ is contractible. The point stabilizers of the action of $\mathrm{Map}(\Gamma)$ on $\mathbb{O}(\Gamma)$ are subgroups isomorphic to semidirect products of finite groups and homeomorphism groups of clopen subsets of $E(\Gamma)$. In particular, the point stabilizers of the action of $\mathrm{PMap}(\Gamma)$ on $\mathbb{O}(\Gamma)$ are finite.
\end{customprop}

In Subsection \ref{subsec:CoarseGeoPure} we translate the results of Domat--Hoganson--Kwak about pure mapping class groups of graphs to pure mapping class groups of doubled handlebodies. Instead of repeating all the results here, we refer the reader to Subsection \ref{subsec:CoarseGeoPure}.
\par 
Finally, in Subsection \ref{subsec:CoarseGeoTranslatable} we consider a particular collection of graphs known as \textit{translatable graphs}, analogous to the translatable surface of Schaffer-Cohen \cite{schaffer-cohen_graphs_2024}. These can be thought of as graphs which can be formed by taking a collection of copies of a fixed graph arranged like $\Z$ with neighboring copies wedge summed along a point. See Definition \ref{def:Transl} for a more precise definition. For a particular subclass of these graphs (those with tame end space, see Definition \ref{def:tame}), we have the following.

\begin{customthm}{8.20}
    Suppose $\Gamma$ is a translatable graph that has a tame end space. The groups $\mathrm{Map}(\Gamma)$ and $\mathrm{Map}(M_{\Gamma})$ are CB generated and, equipped with a CB generating set and the associated word metric, are equivariantly quasi-isometric to $\mathcal{TS}(M_{\Gamma})$.
\end{customthm}

We also note that Section \ref{Connectivity} contains various results about the connectivity of $\mathcal{S}(M_{\Gamma})$ and some of its subcomplexes, extending results of Hatcher in \cite{hatcher_homological_1995}. In Subsection \ref{subsec:Topology} we show that the permutation topology on $\text{Map}(\Gamma)$ induced by its action on $\mathcal{S}(M_{\Gamma})$ agrees with the topology defined on it in Section \ref{sec:MCG}, and also that $\text{Map}(M_{\Gamma})$ is Polish.

\vskip 10pt

\noindent\textbf{Outline}: In Section \ref{sec:Background}, we set up all the background for the rest of the paper. In Section \ref{sec:MCG}, we discuss the mapping class groups of graphs and their doubled handlebodies, and produce the surjective homomorphism which appears in Theorem \ref{SESMainTheorem}. In Section \ref{sec:SES}, we prove Theorem \ref{SESMainTheorem}(1), and begin the proof of Theorem \ref{TwistGroupStructure}. After this we complete the proofs of both of these theorems in Section \ref{sec:SphereTwists}. We consider the connectivity properties of $\mathcal{S}(M_{\Gamma})$ and some of its subcomplexes in Section \ref{Connectivity}. In Section \ref{Applications} we study the permutation topology on $\text{Map}(\Gamma)$, and afterwords discuss nonseparating spheres and prove Theorem \ref{PositiveTranslationLength}, as well as define an Outer space for finite rank graphs and prove Proposition \ref{prop:OuterSpaceContrStab}. Finally in Section \ref{sec:CoarseGeo} we study the coarse geometry of (pure) mapping class groups of graphs and their doubled handlebodies and prove results analogous to those in \cite{domat_coarse_2023} and \cite{domat_generating_2023} for doubled handlebodies, and afterwords prove Theorem \ref{thm:TranslCBGen}.

\vskip 5pt
\noindent\textbf{Acknowledgements}: The author thanks his advisor Christopher J. Leininger for his support and guidance throughout this project. He'd also like to thank George Domat, Hannah Hoganson, and Mladen Bestvina for discussions about their work and for their support. Thanks to Thomas Hill, Michael Kopreski, Rebecca Rechkin, and George Shaji for helpful conversations on the topic while working on a project in parallel to this one. Finally, thanks to Sanghoon Kwak for helpful comments and corrections on an earlier version of the paper. In addition, the author acknowledges partial support from NSF grant DMS-1745670.

\tableofcontents

\section{Background}\label{sec:Background}

\subsection{Space of ends}

We define a \textit{generalized continuum} to be a topological space $X$ which is Hausdorff, connected, locally connected, and $\sigma$-compact. The \textit{space of ends} $E(X)$ of a generalized continuum $X$ is defined to be the inverse limit
\begin{equation*}
    E(X):=\varprojlim_{K\subset X} \pi_0(X \setminus K)
\end{equation*}
where the limit runs over the compact subsets $K$ of $X$. With the standard topology, $E(X)$ is a totally disconnected compact metrizable space, so it is in particular homeomorphic to a closed subset of the Cantor set. This notion is originally due to Freudenthal \cite{freudenthal_uber_1931}, and the book of Baues--Quintero discusses these notions in the generality considered here \cite{baues_infinite_2001}. Every space that we will consider the end space of is either a manifold or a locally finite graph, both of which are generalized continua.
\par 
Points of $E(X)$ may equivalently be described as equivalence classes of infinite sequences $U_1 \supset U_2 \supset \cdots$ of connected nonempty open subsets of $X$ with noncompact closure so that the intersection of the closures of the $U_i$'s is empty. Two sequences $U_1 \supset U_2 \supset \cdots$ and $V_1 \supset V_2 \supset \cdots$ are equivalent if for all $n$, there is an $m$ so that $V_m\subset U_n$, and vice versa. We write $e=[U_1\supset U_2 \supset \cdots]$ for an element of $E(X)$.
\par 
There is a natural compactification $X \sqcup E(X)$ of $X$, with basis given by the union of a basis for $X$, along with the sets of the form $U\cup U^*$, where $U$ is a connected nonempty open subset of $X$ with noncompact closure and
\begin{equation*}
    U^* = \{e=[U_1\supset U_2 \supset \cdots]\in E(X) \ | \ \exists n \text{ such that } U_n\subset U\}.
\end{equation*}
The collection of all such $U^*$ sets also forms a basis for a topology on $E(X)$.
\par 
We say that $f:X\to Y$, a function between two topological spaces, is a \textit{proper} map if it is continuous and for any compact set $K\subset Y$, $f^{-1}(K)$ is compact. A proper map $f:X\to Y$ between two generalized continua induces a continuous map between their end spaces. In fact, one can think of $E$ as a functor from the category of generalized continua with arrows given by proper maps to the category of compact Hausdorff spaces with arrows the continuous maps. In particular, a proper homotopy equivalence $f:X\to X$ with proper inverse $g:X\to X$ (so that $fg$ and $gf$ are properly homotopic to the identity map) induces a homeomorphism $E(f):E(X)\to E(X)$ which extends to a continuous map from $X\sqcup E(X)$ to itself. 

\subsection{Graphs and their doubled handlebodies}\label{subsec:graphsandtheirdoubledhandlebodies}
Let $\Gamma$ be a locally finite connected graph (every graph considered in this paper will be locally finite and connected, unless otherwise specified). Typically, $\Gamma$ will be noncompact. We consider the space of ends of $\Gamma$. The \textit{space of ends accumulated by loops} of $\Gamma$, denoted $E_{\ell}(\Gamma)$, is the closed subset of $E(\Gamma)$ consisting of those ends whose neighborhoods in $\Gamma$ are all of infinite rank. Note that the elements of $E_{\ell}$ are called \textit{non}-$\infty$-\textit{stable ends} in \cite{ayala_proper_1990} and \textit{ends accumulated by genus} in \cite{algom-kfir_groups_2021}. 
\par 
We say a graph is of \textit{finite type} if it has finite rank and finitely many ends. Otherwise, we say that it is of \textit{infinite type}.
\par 
We consider pairs of spaces of the form $(E, E_{\ell})$, where $E$ is the end space of some graph $\Gamma$, and $E_{\ell}$ is the corresponding space of ends accumulated by loops. We say that two such pairs $(E, E_{\ell})$ and $(E', E'_{\ell})$ have the \textit{same homeomorphism type} and write $(E, E_{\ell})\cong (E', E'_{\ell})$ if there is a homeomorphism $f:E \to E'$ which restricts to a homeomorphism $f|_{E_{\ell}}: E_{\ell} \to E'_{\ell}$. The \textit{characteristic triple} of $\Gamma$ is the triple $(\text{rk}(\Gamma), E(\Gamma), E_{\ell}(\Gamma))$. Two characteristic pairs $(r, E, E_{\ell})$ and $(r', E', E'_{\ell})$ are \textit{isomorphic} if $r=r'$ and $(E, E_{\ell}) \cong (E', E'_{\ell})$, and an \textit{isomorphism} between them is a choice of homeomorphism from $(E, E_{\ell})$ to $(E', E'_{\ell})$.
\par
It is clear that a proper homotopy equivalence between two locally finite graphs induces an isomorphism between their characteristic pairs. Conversely, we have the following classification theorem due to Ayala--Dominguez-M\'{a}rquez--Quintero \cite{ayala_proper_1990}. 
\begin{theorem}[{\cite[Theorem 2.7]{ayala_proper_1990}}]\label{ADMQGraphClassification}
    An isomorphism of two characteristic triples $(r, E, E_{\ell})$ and $(r', E', E'_{\ell})$ of two locally finite graphs $\Gamma$ and $\Gamma'$ is induced by a proper homotopy equivalence $\Gamma \to \Gamma'$. If $\Gamma$ and $\Gamma'$ are both trees, then this extension is unique up to proper homotopy.
\end{theorem}

In particular, every locally finite graph $\Gamma$ is properly homotopic to a graph in \textit{standard form}, originally defined in \cite{domat_coarse_2023}. This is a graph which consists of a tree with loops attached at vertices. It is often simpler to assume a graph is in this form, and we will do so when needed.

\par 
Let $M_{n,s}$ denote the connect sum of $n$ copies of $S^2\times S^1$ with $s$ open balls removed. We will define an associated $3$-manifold $M_{\Gamma}$ as the union of copies of $M_{n,s}$ for varying $n,s$ which are glued together along their boundaries in a specified way, see Definition \ref{def:DoubledHandlebody}. 
\begin{definition}\label{def:pieceOfAVertex}
    The \textit{piece} of a vertex $v$ of $\Gamma$ is a homeomorphic copy of $M_{n,s}$, where $n$ is the number of edges which have $v$ as both vertices (such an edge is called a \textit{loop}), and $s$ is the number of edges which have $v$ as only one of its vertices. 
\end{definition}

The piece of a vertex $v$ with $n$ loops and $s$ other edges contains a copy of a graph with a vertex with $n$ loops and $s$ extra edges beginning at that vertex. 

\begin{definition}\label{def:DoubledHandlebody}
    The \textit{doubled handlebody} (with punctures) $M_{\Gamma}$ of $\Gamma$ is the union of all the pieces of the vertices of $\Gamma$, with two boundary spheres quotiented together by an orientation reversing map if the corresponding vertices are connected by an edge in $\Gamma$.

\end{definition}

\begin{figure}
    \centering
    \includegraphics[scale=.6]{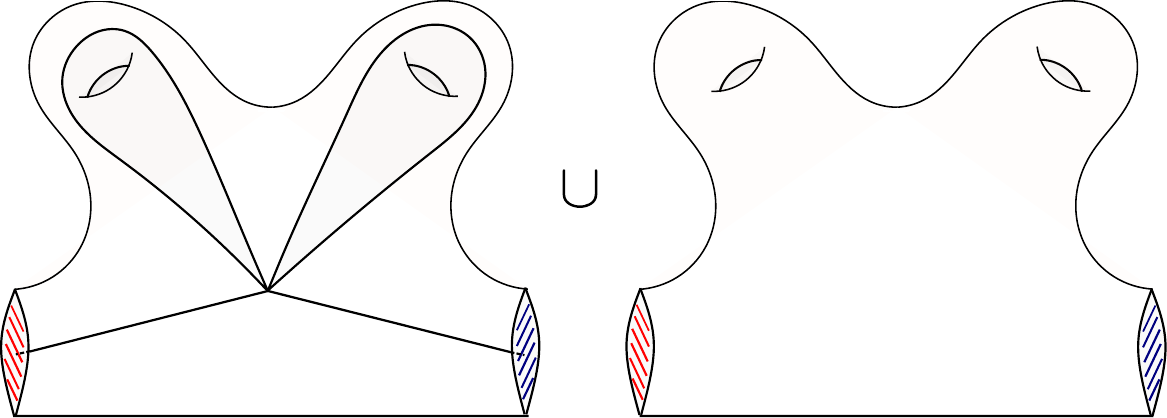}
    \caption{A picture of $M_{2,2}$ with an embedded copy of the graph with a vertex with $2$ loops and $2$ other edges. We have decomposed $M_{2,2}$ as the union of two compact three dimensional handlebodies, glued along their boundaries except for some discs in the boundary. These discs are colored red and blue, corresponding to which discs have their boundaries matched up. Each disc is a hemisphere of one of the boundary components of $M_{2,2}$. We assume that the graph is embedded in such a way that the handlebody that it lies in is isotopic to a regular neighborhood of the graph.}
    \label{fig:compactdoubledHandlebody}
\end{figure}

One can think of $\Gamma$ as being properly embedded into $M_{\Gamma}$ in a natural way. Namely, so that the intersection of a piece with $\Gamma$ looks like Figure \ref{fig:compactdoubledHandlebody}. We will call a regular neighborhood in $M_{\Gamma}$ of this embedding a \textit{handlebody} of $\Gamma$. It is easy to see that the complement of a handlebody of $\Gamma$ in $M_{\Gamma}$ is homeomorphic to the handlebody itself (see Figure \ref{fig:doubledHandlebodyEx}). Thus we can think of $M_{\Gamma}$ as the union of two handlebodies glued along their surface boundary. We will write $M_{\Gamma}=B\cup B'$, where $B$ and $B'$ are two copies of a fixed handlebody of $\Gamma$ which we quotient together along their surface boundaries via the identity map. We denote by $\Sigma$ the corresponding surface, which in the quotient is the intersection of $B$ and $B'$. We think of $\Gamma$ as a subspace of $B\subset M_{\Gamma}$, embedded via the map $i:\Gamma \to M_{\Gamma}$. We sometimes also consider a copy of $\Gamma$ in $B'$ (that is embedded in $B'$ the same way that $\Gamma$ is embedded in $B$), which we denote by $\Gamma'$. It will follow from Proposition \ref{RichardsHandlebody} that the homeomorphism type of $M_{\Gamma}$ only depends on the characteristic triple of $\Gamma$.
\par 
See Figure \ref{fig:doubledHandlebodyEx} for an example of a doubled handlebody. Here $M_{\Gamma}$ is a union of infinitely many copies of $M_{1,2}$.
\begin{figure}
    \centering
    \includegraphics[angle=-270,scale=.5]{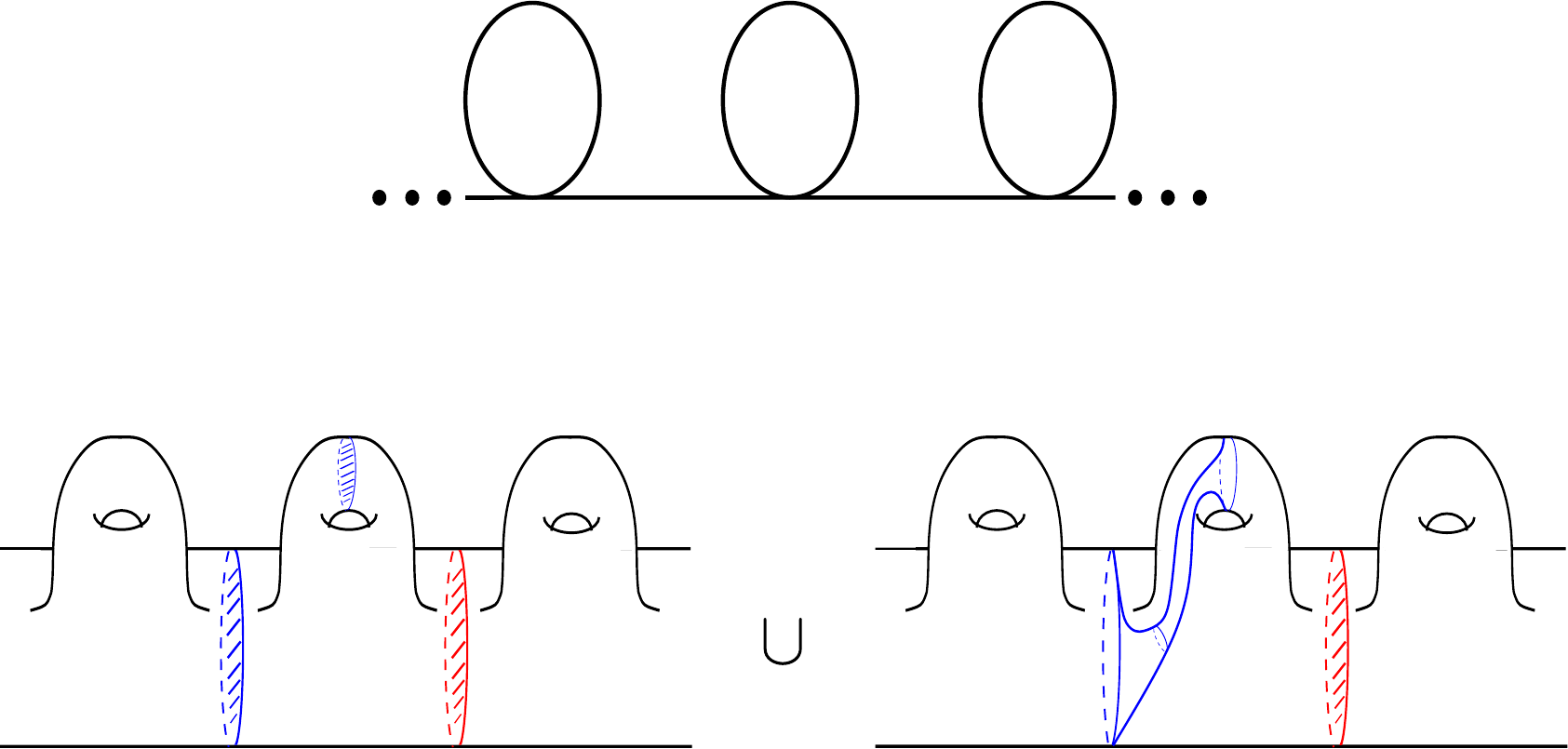}
    \caption{A two ended infinite type graph, with its associated doubled handlebody (drawn as the union of two handlebodies), are pictured here. Also included are two embedded spheres in $M_{\Gamma}$, a red one and a blue one.}
    \label{fig:doubledHandlebodyEx}
\end{figure}

Note that $B$ deformation retracts onto $\Gamma$. Further, we can also obtain a retraction $r:M_{\Gamma}\to \Gamma$ defined as follows. On $B$, we define $r$ as the retraction induced by the deformation retraction given above. On $B'$, $r$ is defined by first reflecting $B'$ onto $B$ by a map preserving $\Sigma$, and then composing by the retraction of $B$ to $\Gamma$ just given. While this retraction is not induced by a deformation retraction, it is easy to see that $r$ induces an isomorphism on fundamental groups. 
\par
We may assume that the retraction $r$ respects the gluing of the pieces of $M_{\Gamma}$. That is, the preimage of a point of $\Gamma$ under $r$ is the expected subset of a given piece. For example, $r^{-1}(p)$ for $p$ a point in the interior of an edge which is not a loop is a sphere containing $p$ which is parallel to a boundary sphere of a piece. Further, we can assume that the boundary components of pieces have $r$ image the midpoint of the corresponding edge. If $K$ is a union of pieces, then $r(K)$ is a graph whose valence $1$ vertices have $r$ preimages the boundary component of $K$. If $\Gamma$ is in standard form, then every piece is homeomorphic to $M_{1,s}$ or $M_{0,s}$ for some $s\geq 0$.
\par 
We note the following simple fact. which will typically be used without reference. 
\begin{fact}\label{CompactExhaustion}
    $M_{\Gamma}$ admits a compact exhaustion whose elements consist of connected unions of pieces. 
\end{fact}
Typically, for consistency, when we need to take a compact exhaustion, we will assume that it is by unions of pieces.
\par
The end spaces of $B$, $\Sigma$, and $M_{\Gamma}$ all have notions of a subspace of ends "accumulated by loops". (For $\Sigma$, this space is typically called the ends accumulated by genus). Thus, all three have notions of characteristic triples $(r, E, E_{\ell})$, defined in the same way, except in the case of $\Sigma$ where $r$ is defined to be the genus of $\Sigma$. It still makes sense to talk about isomorphisms between their characteristic triples, in this case.
\par 
We note the following lemma about the characteristic triples of $\Gamma, B, \Sigma,$ and $M_{\Gamma}$.
\begin{lemma}\label{EndsHomeo}
    The inclusions of $\Gamma$, $B$, and $\Sigma$ into $M_{\Gamma}$ induce isomorphisms between their characteristic triples. 
\end{lemma}
\begin{proof}
    First note that these inclusions induce bijections between pairs $(E, E_{\ell})$, and their ranks/genus are all the same. The latter statement is immediate, and to see the former statement, note that every end of any of the given spaces can be represented by a sequence of particularly simple sets. That is, by Fact \ref{CompactExhaustion}, every end of $M_{\Gamma}$ can be represented by a sequence of open unbounded sets whose closures have spherical boundary so that these sphere components are preimages of points in the interior of edges of $\Gamma$. A similar statement holds for $\Sigma$ and $B$ by intersecting the sets described for $M_{\Gamma}$ with either of the two subspaces.
    \par 
    Similar reasoning shows that each inclusion is continuous. Namely, any open set of $E(M_{\Gamma})$ is the union of sets of the form $U^*$, where $U$ is an unbounded open set of the form described in the previous paragraph. The preimage of each such open set is also open.
    \par 
    The induced maps also preserve the ends accumulated by loops. As each map is a continuous bijection between two compact Hausdorff spaces, they are all homeomorphisms, which restrict to homeomorphisms of the end accumulated by loops.
\end{proof}

We have a similar lemma, whose proof is essentially identical to Lemma \ref{EndsHomeo}.
\begin{lemma}
    The retraction map $r:M_{\Gamma}\to \Gamma$ induces an isomorphism between the characteristic triples of $M_{\Gamma}$ and $\Gamma$.
\end{lemma}

While we will not make use of it, we note a correspondence for proper homotopy classes of lines in $M_{\Gamma}$ and $\Gamma$. Further, we also consider the following special case of lines. Suppose we are given an end $e\in E(\Gamma)\setminus E_{\ell}(\Gamma)$. The fundamental group $\pi_1(\Gamma, e)$ based at $e$ is the set of proper homotopy classes of proper lines $\sigma: \R \to \Gamma$ so that $\lim_{t\to \pm \infty} \sigma(t) = e$. The group operation is given by concatenation, which makes sense as all elements in $\pi_1(\Gamma,e)$ eventually coincide as sets since $e\in E(\Gamma) \setminus E_{\ell}(\Gamma)$. These groups were used in \cite{algom-kfir_groups_2021}.
\par 
We can define the same thing for $M_{\Gamma}$. Namely, for $e\in E(M_{\Gamma}) \setminus E_{\ell}(M_{\Gamma})$, we let $\pi_1(M_{\Gamma}, e)$ denote the set of proper homotopy classes of proper lines in $M_{\Gamma}$ both of whose ends accumulate to $e$. Here the group operation is not obvious as lines accumulating to the same end do not eventually coincide. The following lemma shows that the retraction $r$ induces a natural bijection between these two sets.
\begin{lemma}\label{FundGroupIso}
    The retraction $r:M_{\Gamma} \to \Gamma$ induces a bijection between the proper homotopy classes of lines in $M_{\Gamma}$ and proper homotopy classes of lines in $\Gamma$. In particular, there is an induced bijection from $\pi_1(M_{\Gamma}, e)$ to $\pi_1(\Gamma, e)$. 
\end{lemma}
\begin{proof}
    Take a proper line $\sigma:\R \to M_{\Gamma}$. If $\sigma$ intersects $\Gamma'$, then we may homotope $\sigma$ by a small homotopy to $\sigma'$ which does not intersect $\Gamma'$. By making this homotopy small enough, it is clear that $r\circ \sigma$ and $r\circ \sigma'$ are properly homotopic lines in $\Gamma$. But by flowing away from $\Gamma'$ we may homotope $\sigma'$ to $r\circ \sigma'$. Thus $\sigma$ is properly homotopic to $r\circ \sigma$. (This flow, which is a deformation retraction from $M_{\Gamma}\setminus \Gamma'$ to $\Gamma$, is not proper, but when restricted to a homotopy defined on a properly embedded line it is proper).
    \par 
    On the other hand, if $\eta, \eta':\R \to M_{\Gamma}$ are two properly homotopic lines via the proper homotopy $H:[0,1]\times \R \to M_{\Gamma}$, then $r\circ H$ is a proper homotopy equivalence between $r\circ \eta$ and $r\circ \eta'$. 
    \par 
    It follows that the map from the set of proper homotopy classes of proper lines in $M_{\Gamma}$ to those of $\Gamma$ induced by post composition with $r$ is a bijection. In particular, this restricts to a bijection between $\pi_1(M_{\Gamma}, e)$ to $\pi_1(\Gamma, e)$.
\end{proof}
We will see later that this bijection respects the passing from the mapping class group of $M_{\Gamma}$ to that of $\Gamma$, see Proposition \ref{MCGHom}.

\begin{remark}\label{Identifications}
    In what follows, we will identify the end spaces of all these spaces as induced by the inclusion maps given in Lemma \ref{EndsHomeo}. We will often just denote the space by $E$, with $E_{\ell}$ being the space of ends accumulated by loops (or genus). Similarly, we will identify the fundamental groups of $\Gamma$ and $M_{\Gamma}$, both at basepoints in $\Gamma$, and at basepoints in the end space. 
\end{remark}

We now discuss the Hatcher normal form for finite sphere systems in $M_{\Gamma}$, and the related notion of equivalence of sphere systems. The definitions and proofs of the following can all be found in \cite{hatcher_homological_1995}, \cite{hatcher_isoperimetric_1996}, and \cite{hensel_realisation_2014}. 
\par 
Fix a maximal sphere system $\mathcal{P}$ in $M_{\Gamma}$ (i.e. a sphere system so every essential non-peripheral sphere either intersects or is isotopic to some component of it). If $\Gamma$ is of infinite type, then $\mathcal{P}$ is infinite. Such a maximal system splits $M_{\Gamma}$ into a collection $\{P_k\}$ of copies of $M_{0,3}$ (sometimes one has to replace a boundary component of a copy of $M_{0,3}$ with a puncture corresponding to an end of $M_{\Gamma}$). 
\begin{definition}
    A finite sphere system $S$ is in \textit{normal form} with respect to $\mathcal{P}$ if a given component $S_i$ of $S$ either coincides with a component of $\mathcal{P}$, or $S_i$ meets $\mathcal{P}$ transversely in a nonempty collection of circles which split $S_i$ into components called \textit{sphere pieces} which satisfy the following for all $P_k$:
    \begin{enumerate}
        \item Each sphere piece in $P_k$ meets each component of $\partial P_k$ in at most one circle.
        \item No sphere piece in $P_k$ is a disc which is isotopic relative to its boundary to a disc in $\partial P_k$.
    \end{enumerate}
    Given any sphere system $Q$, we say that $S$ is in normal form with respect to $Q$ if it is in normal form with respect to some maximal system containing $Q$.
\end{definition}

\begin{definition}\label{def:equivSphere}
    Fix two finite sphere systems $S$ and $S'$ which are in normal form with respect to a sphere system $Q$. We say that they are \textit{equivalent} if there is a homotopy $h_t:S\to M_{\Gamma}$ from $S$ to $S'$ so that
    \begin{enumerate}
        \item On the components $S_i$ of $S$ which coincide with a component of $Q$, $h_i$ is the identity.
        \item On the other components $S_i$, $h_t$ remains transverse to $Q$ for all $t$, and $h_t(S)\cap Q$ varies by isotopy in $Q$.
    \end{enumerate}
\end{definition}

Each part of the following lemma for the manifold $M_{n,s}$ can be found in one of \cite{hatcher_homological_1995}, \cite{hatcher_isoperimetric_1996} or \cite{hensel_realisation_2014}. However, as the sphere systems we are dealing with are finite, they are contained in some copy of $M_{n,s}$ in $M_{\Gamma}$ and thus the argument can be replicated.
\begin{lemma}\label{lem:NormalAndEquivalent}
    Fix $Q$ a sphere system in $M_{\Gamma}$.
    \begin{enumerate}[label=(\alph*)]
    \item Every finite sphere system $S$ can be isotoped into normal form with respect to $Q$.
        \item A finite sphere system $S$ is in normal form with respect to $Q$ if and only if the number of components of $S\cap Q$ is minimal amongst representatives of the isotopy class of $S$.
        \item Isotopic sphere systems in normal form with respect to $Q$ are equivalent.
    \end{enumerate}
\end{lemma}
For a proof of Lemma \ref{lem:NormalAndEquivalent}(c), see Lemma 7.3 of \cite{hensel_realisation_2014} (the proof there assumes $Q$ is a single sphere, but it is easy to see how to extend it to any sphere system).

\subsection{The Sphere Complex}\label{SphereGraph}
We define \textit{sphere complex} of $M_{\Gamma}$ to be the simplicial complex, denoted by $\mathcal{S}(M_{\Gamma})$, with $k$-simplices given by isotopy classes of sphere systems with $k+1$ elements. Recall that all the spheres in a sphere system are essential and non-peripheral. For notational simplicity, when we say that $S\in \mathcal{S}(M_{\Gamma})$, what we really mean is that $S$ is a vertex of $\mathcal{S}(M_{\Gamma})$ corresponding to some sphere in $M_{\Gamma}$, unless stated otherwise. Given any subcomplex $X$ of $\mathcal{S}(M_{\Gamma})$, we let $X^{(1)}$ denote the $1$-skeleton of $X$.
\par 
As discussed in the introduction, this complex is well studied. In particular, for $n\geq 2$, $\mathcal{S}(M_{n,0})$ and $\mathcal{S}(M_{n,1})$ have infinite diameter \cite{handel_free_2019}\cite{iezzi_sphere_2016}. Further, for $n\geq 2$, $\mathcal{S}(M_{n,0})$ is known to be hyperbolic, and the loxodromic elements are classified \cite{handel_free_2013}\cite{handel_free_2019}. See \cite{hilion_hyperbolicity_2017} for more discussion about hyperbolicity. It is known that for $n\geq 4$, $\mathcal{S}(M_{n,2})$ contains quasi-isometrically embedded copies of $\R^k$ for all $k\geq 0$, and is thus in particular not hyperbolic, and has infinite asymptotic dimension \cite{hamenstadt_spotted_2023}. For all $n\geq 0$ and $s\geq 3$, $\mathcal{S}(M_{n,s})$ has finite diameter \cite{iezzi_sphere_2016}. We also note the special case of $\mathcal{S}(M_{1,2})$. In this case, every separating sphere is disjoint from exactly $1$ nonseparating sphere, and every nonseparating sphere is disjoint from exactly $1$ separating sphere and $2$ nonseparating spheres. In particular, $\mathcal{S}(M_{1,2})$ is isomorphic to a graph which consists of the real line with vertices at every integer point, whose vertices correspond to the nonseparating spheres, along with one edge for each such vertex which connects the nonseparating sphere to the unique separating sphere it is disjoint from. We will make use of this graph in the proof of Theorem \ref{PositiveTranslationLength}.
\par 
We note the following simple fact about $\mathcal{S}(M_{\Gamma})$ when $\Gamma$ is of infinite type, mirroring the corresponding result for surfaces.
\begin{proposition}\label{SphereGraphFiniteDiameter}
    Suppose $\Gamma$ is of infinite type (i.e. it does not have finite rank or finitely many ends). Then $\mathcal{S}(M_{\Gamma})$ has diameter $2$.
\end{proposition}
\begin{proof}
    Any two vertices $S$ and $S'$ of $\mathcal{S}(M_{\Gamma})$ are contained simultaneously in some finite union of pieces $K$, up to isotopy. By the assumption that $\Gamma$ has infinite type, there is some vertex $S''$ contained in $M_{\Gamma}\setminus K$. In particular, \begin{equation*}
        d_{\mathcal{S}(M_{\Gamma})}(S, S') \leq d_{\mathcal{S}(M_{\Gamma})}(S, S'')+d_{\mathcal{S}(M_{\Gamma})}(S', S'') =2.\end{equation*}\end{proof}

We will discuss more properties of $\mathcal{S}(M_{\Gamma})$ and some of its subgraphs in Section \ref{Connectivity}. The following lemma is about the sphere complex of $\mathcal{S}(M_{\Gamma})$ when $\Gamma$ is simply connected. We will utilize this in the proof of Theorem \ref{SESMainTheorem}.

\begin{lemma}\label{lem:simConnectedSpherePartition}
    Let $\Gamma$ be a tree. Then two embedded spheres $S$ and $S'$ in $M_{\Gamma}$ not bounding balls are isotopic if and only if they induce the same partition of $E(\Gamma)$.
\end{lemma}
\begin{proof}
    An isotopy of a sphere in $M_{\Gamma}$ preserves the ends of $M_{\Gamma}$ it separates, so if two spheres are isotopic they must induce the same partition of $E(\Gamma)$. On the other hand, suppose $S$ and $S'$ induce the same partition of $E(\Gamma)$. As $S$ and $S'$ are compact, we can choose a finite number of disjoint spheres $T_1,\ldots, T_n$ which are disjoint from $S$ and $S'$ so that $S, S'$ lie inside the submanifold bounded by $T_1, \ldots, T_n$, and so that this submanifold is compact. As $S$ and $S'$ induce the same partition of $E(\Gamma)$, they induce the same partition of the spheres $T_1, \ldots, T_n$. But isotopy classes of spheres in manifolds homeomorphic to $M_{0, n}$ are determined by their partitions of boundary spheres, see Lemma 9 of \cite{bering_iv_finite_2024}.
\end{proof}

\par 
Fix a sphere system $Q$ and an embedded sphere $S$ in normal form with respect to $Q$. We may identify an \textit{innermost disc} $D$ of this intersection, which is a disc in a component of $Q$ so that $D$ does not contain any other component of $S\cap Q$. The \textit{surgery} of $S$ along $D$ is the pair of spheres which result from cutting $S$ along $D$, and then attaching copies of $D$ to each of the two components. The resulting spheres are embedded and essential, though not always non-peripheral.
\par 

In Subsection \ref{nonseparating Sphere Finite Rank}, we will study the \textit{nonseparating sphere complex} of $M_{\Gamma}$. This complex, denoted $\mathcal{S}_{ns}(M_{\Gamma})$, is the full subcomplex of $\mathcal{S}(M_{\Gamma})$ whose vertices are the nonseparating spheres of $M_{\Gamma}$. We note the following.

\begin{lemma}\label{lem:nonSepConnected}
    The complex $\mathcal{S}_{ns}(M_{\Gamma})$ is connected.  
\end{lemma}
\begin{proof}
    It suffices to show that $\mathcal{S}_{ns}(M_{n,s})$ is connected for all $n\geq 1$, $s\geq 0$, as any two spheres in $M_{\Gamma}$ are contained in a homeomorphic copy of $M_{n,s}$ for sufficiently large $n$, $s$. Suppose $S$ and $S'$ are two nonseparating spheres in $M_{n,s}$. Isotope them so that they are in normal form with respect to each other. If they are disjoint, then there is clearly a path between them. Otherwise, choose an innermost disc in $S'$, and surger $S$ along this disc to obtain two spheres $S_1$ and $S_2$ disjoint from $S$, both of which intersect $S'$ in a smaller number of components. Since $S$ is nonseparating, as an element in homology it is the sum of the homology classes of $S_1$ and $S_2$ (with appropriate orientations). Thus, one of $S_1$ or $S_2$ is nontrivial in homology, and in particular nonseparating. Suppose it is $S_1$. As $S$ and $S_1$ are connected by an edge and $S_1$ intersects $S'$ less than $S$ does, we may continue by induction on intersection number to find a path consisting of nonseparating spheres between $S$ and $S'$, showing that $\mathcal{S}_{ns}(M_{n,s})$ is connected.
    \end{proof}

In Subsection \ref{nonseparating Sphere Finite Rank}, we will utilize the free factor complex of a free group $F_n$, which was initially defined by Hatcher-Vogtmann in 1998 (see \cite{hatcher_complex_2022}) in the proof of Theorem \ref{PositiveTranslationLength}. Following the notation of Bestvina-Bridson in \cite{bestvina_rigidity_2023}, we denote this complex by $\mathcal{AF}_n$. For $n\geq 3$, this complex is the geometric realization of the poset of proper free factors of $F_n$. Namely, its simplices correspond to an ordered sequence of proper free factors of $F_n$. For $n=2$, two free factors $\<a\>$ and $\<b\>$ are joined by an edge if $a$ and $b$ generate the entire free group. In either case, there is a natural action of $\text{Aut}(F_n)$ on this complex. (Note that the complexes we are considering differ from what is usually called the free factor complex, whose vertices are conjugacy classes of free factors or generators).
\par 
 Recall that the \textit{translation length} of an element $g\in G$ where $G$ acts by isometries on a metric space $X$ is defined to be
\begin{equation*}
    l_X(g)=\lim_{k \to \infty} \frac{d_X(x, g^k(x))}{k}
\end{equation*}
where $x\in X$. This quantity is well defined, finite, and independent of the choice of $x$.
\par 
The following lemma is a consequence of the proof of Proposition 2.12 in \cite{bestvina_rigidity_2023}.
\begin{lemma}\label{FreeFactorPositiveTranslation}
    Fix $n\geq 2$. Then the action of $\mathrm{Aut}(F_n)$ on $\mathcal{AF}_n$ admits elements with positive translation length.
\end{lemma}

Fix a basepoint $p\in \partial M_{n,1}$. We need one more auxiliary complex, which we will denote by $\widehat{\mathcal{S}}_{ns}(M_{n,1})$. It is the flag complex whose vertex set agrees with $\mathcal{S}_{ns}(M_{n,1})$. If $n\geq 3$, then there is an edge between two disjoint spheres $S$ and $S'$ if $\pi_1(M_{n,1}\setminus(S \cup S'),p)$ is nontrivial. If $n=2$, then there is an edge between disjoint spheres $S$ and $S'$ if $\pi_1(M_{n,1}\setminus S, p)$ and $\pi_1(M_{n,1} \setminus S', p)$ generate $\pi_1(M_{n,1}, p)$. Note that for any $n\geq 2$, an edge in $\mathcal{S}_{ns}(M_{n,1})$ between vertices $S$ and $S'$ is not in $\widehat{\mathcal{S}}_{ns}(M_{n,1})$ if and only if $S$, $S'$, and $\partial M_{n,1}$ bound a copy of $M_{0,3}$.
\par 
Then we have the following.
\begin{lemma}\label{FreeFactorEquivariantMap} Fix $n\geq 2$. 
\begin{enumerate}[label=(\alph*)]
    \item The complex $\widehat{\mathcal{S}}_{ns}(M_{n,1})$ is connected, quasi-isometric to $\mathcal{S}_{ns}(M_{n,1})$, and $\text{Aut}(F_n)$ invariant.
    \item There is an $\mathrm{Aut}(F_n)$-equivariant map $\Phi:\widehat{\mathcal{S}}_{ns}^{(1)}(M_{n,1})\to \mathcal{AF}_n$ so that for all vertices $S, S'\in \widehat{\mathcal{S}}_{ns}(M_{n,1})$,
    \begin{equation*}
        d_{\mathcal{AF}_n}(\Phi(S), \Phi(S'))\leq 2d_{\widehat{\mathcal{S}}_{ns}(M_{n,1})}(S, S').
    \end{equation*}
    That is, $\Phi$ is $2$-Lipschitz.
\end{enumerate}
\end{lemma}
    We note that the action of $\text{Aut}(F_n)$ on $\mathcal{S}_{ns}(M_{n,1})$ is induced from the short exact sequence given in Proposition \ref{HatcherVSES}.
\begin{proof}
\begin{enumerate}[label=(\alph*)]
    \item Fix two disjoint vertices $S$ and $S'$. If $S$ and $S'$ are not joined by an edge in $\widehat{\mathcal{S}}_{ns}(M_{n,1})$, then it follows from the end of the paragraph after Lemma \ref{FreeFactorPositiveTranslation} that the complementary component of $S\cup S'$ not containing $p$ is homeomorphic to $M_{n-1, 2}$. As $n\geq 2$, there is a nonseparating sphere $S''$ in this component. 
    \par 
    If $n=2$, then it is clear that $\pi_1(M_{n,1}\setminus S'',p)$ along with either $\pi_1(M_{n,1}\setminus S,p)$ or $\pi_1(M_{n,1}\setminus S',p)$ generates $\pi_1(M_{n,1},p)$ (this can be seen via the Van Kampen theorem), so in this case we have an edge path of length $2$ in $\widehat{\mathcal{S}}_{ns}(M_{n,1})$ connecting $S$ and $S'$. If $n\geq 3$, then $\pi_1(M_{n,1}\setminus (S\cup S''), p)$ and $\pi_1(M_{n,1}\setminus (S'\cup S''), p)$ are both nontrivial free factors of $\pi_1(M_{n,1}, p)$. In particular, there is also an edge path of length $2$ between $S$ and $S'$. This shows that $\widehat{\mathcal{S}}_{ns}(M_{n,1})$ is connected, and further it gives the following inequalities for arbitrary vertices $S$ and $S'$ of $\mathcal{S}_{ns}(M_{n,1})$:
    $$\frac{1}{2}d_{\widehat{\mathcal{S}}_{ns}(M_{n,1})}(S, S') \leq d_{\mathcal{S}_{ns}(M_{n,1})}(S, S') \leq d_{\widehat{\mathcal{S}}_{ns}(M_{n,1})}(S, S')$$
    showing that $\widehat{\mathcal{S}}_{ns}(M_{n,1})$ is quasi-isometrically embedded in $\mathcal{S}_{ns}(M_{n,1})$. Since their vertex sets agree, the complexes are also quasi-isometric.
    \par 
    The $\text{Aut}(F_n)$ equivariance follows from the topological characterization of edges of $\mathcal{S}_{ns}(M_{n,1})$ which are not in $\widehat{\mathcal{S}}_{ns}(M_{n,1})$. Namely, such an edge between $S$ and $S'$ exists if and only if $S$, $S'$, and $\partial M_{n,1}$ bound a copy of $M_{0,3}$. This condition is preserved by homeomorphisms, and thus by $\text{Aut}(F_n)$.
    \item Given a vertex $S\in \widehat{\mathcal{S}}_{ns}(M_{n,1})$, we define $\Phi(S)=\pi_1(M_{n,1}\setminus S, p)$.   This is $\text{Aut}(F_n)$ equivariant by construction (we can extend the map to the edges linearly and equivariantly). 
    \par 
    If $n=2$ and $S$ and $S'$ are two vertices joined by an edge, then by definition $\Phi(S)$ and $\Phi(S')$ generate $\pi_1(M_{n,1}, p)$ and thus $\Phi(S)$ and $\Phi(S')$ are connected by an edge. This implies that $\Phi$ is $1$-Lipschitz in this case.
    \par 
    If $n\geq 3$ and $S$ and $S'$ are two vertices joined by an edge, then $\pi_1(M_{n,1} \setminus (S \cup S'), p)$ is a nontrivial free subfactor of both $\Phi(S)$ and $\Phi(S')$. Thus $d_{\mathcal{AF}_n}(\Phi(S), \Phi(S'))\leq 2$. Hence $\Phi$ is $2$-Lipschitz.
    
\end{enumerate}
    
\end{proof}

We remark that the map $\Phi$ is similar to a map used by Kapovich--Rafi in \cite{kapovich_hyperbolicity_2014}, as well as a map used by Hilion--Horbez in \cite{hilion_hyperbolicity_2017}.

\subsection{Borel--Moore Homology}\label{subsec:BorelMoore}
To obtain the splitting of the short exact sequence in Theorem \ref{SESMainTheorem}, one needs to understand how the first cohomology of $M_{\Gamma}$ relates to the embedded spheres of $M_{\Gamma}$. As $M_{\Gamma}$ is typically not compact, the standard second homology will not suffice, as Poincare duality only gives an isomorphism from $H_2(M_{\Gamma}, \Z)$ to $H_c^1(M_{\Gamma}, \Z)$, the first cohomology with compact support. Instead, we utilize the notion of \textit{Borel--Moore} homology, introduced in \cite{borel_homology_1960}

\begin{definition}[{\cite[Borel--Moore Homology]{borel_homology_1960}}]\label{BorelMooreHomology}
    Suppose $X$ is locally compact space. Consider the locally finite singular chains of $X$ of dimension $i$ with coefficients in an abelian group $A$, denoted $C_i^{BM}(X;A)$ (in this paper, one can always assume $A = \Z/2$). This is the group of formal infinite sums of the form
    \begin{equation*}
        \sum_{\sigma} a_{\sigma}\sigma
    \end{equation*}
    where $a_{\sigma}\in A$ and $\sigma$ runs over all continuous maps from the standard $i$-simplex into $X$, so that for each compact set $K\subset X$, there are only finitely $\sigma$'s so that $\text{im}(\sigma)\cap K \neq \varnothing$ and $a_{\sigma}\neq 0$.
    \\
    There are well defined boundary maps $\partial:C_{i+1}^{BM}(X; A)\to C_i^{BM}(X; A)$ defined in the same way as for singular homology. The $i$th \textit{Borel--Moore homology group}, denoted $H_i^{BM}(X; A)$, is defined as
    \begin{equation*}
        H_i^{BM}(X; A)=\text{ker}(\partial:C_{i}^{BM}(X; A)\to C_{i-1}^{BM}(X; A))/\text{im}(\partial:C_{i+1}^{BM}(X; A)\to C_{i}^{BM}(X; A)).
    \end{equation*}
\end{definition}
It is easy to see that if $X$ is compact, then the Borel--Moore homology agrees with standard singular homology.
\par 
We will make use of the following extension of Poincare duality.

\begin{theorem}[{\cite[Section IX Theorem 4.7]{iverson_cohomology_1984}}]
    Fix a manifold $X$ of dimension $n$ and a finitely generated abelian group $A$. There is a natural isomorphism
    \begin{equation*}
        H_k^{BM}(X; A) \cong H^{n-k}(X; A).
    \end{equation*}
\end{theorem}

The second Borel--Moore homology of $M_{\Gamma}$ takes on a particularly simple form, which we will exploit in the proof of the second part of Theorem \ref{SESMainTheorem}. For simplicity, we will assume $\Gamma$ is in standard form. Thus, we can identify each piece of $M_{\Gamma}$ with $M_{1,s}$ or $M_{0,s}$, for some $s \geq 0$, depending on the piece. We think of $M_{1,s}$ as $S^2\times S^1$ with $s$ open balls removed. Define the \textit{core spheres} of $M_{\Gamma}$ to be a choice of one nonseparating spheres in each piece homeomorphic to $M_{1,s}$. Then we have the following lemma, see Figure \ref{fig:BorelMoore} as well.
\begin{lemma}\label{BorelMooreGenerators}
    The group $H_2^{BM}(M_{\Gamma}; A)$ is generated by the core spheres of $M_{\Gamma}$.
\end{lemma}
\begin{proof}
    Fix a nonzero $\alpha \in H_2^{BM}(M_{\Gamma};A)$, represented by a sum of $A$-multiples of $2$-simplices. We may assume by a small homotopy that these simplices are smooth maps which transversely intersect the boundary spheres of the pieces of $M_{\Gamma}$. By subdividing, we may then assume that $\alpha$ is written as a sum of chains, each of which lies in some piece $K$ so that its boundary lies in $\partial K$. Note that by the local finiteness of the original representation of $\alpha$, only finitely many simplices of $\alpha$ intersect a given piece, so this element is well defined.
    \par 
    We can realize the summand in the piece $K$ as an element of $H_2(K, \partial K;A)$. Lefschetz duality implies that every element of $H_2(K, \partial K;A)\cong H^1(K;A)$ can be written as an $A$-multiple of the core sphere of $K$, if it exists. This is because the dual of the standard generator of $H^1(K; A)\cong A$ is the core sphere. Thus, in particular, $\alpha$ can be written as a sum over the core spheres of $M_{\Gamma}$.
\end{proof}

\begin{figure}
    \centering
    \includegraphics[scale=.5]{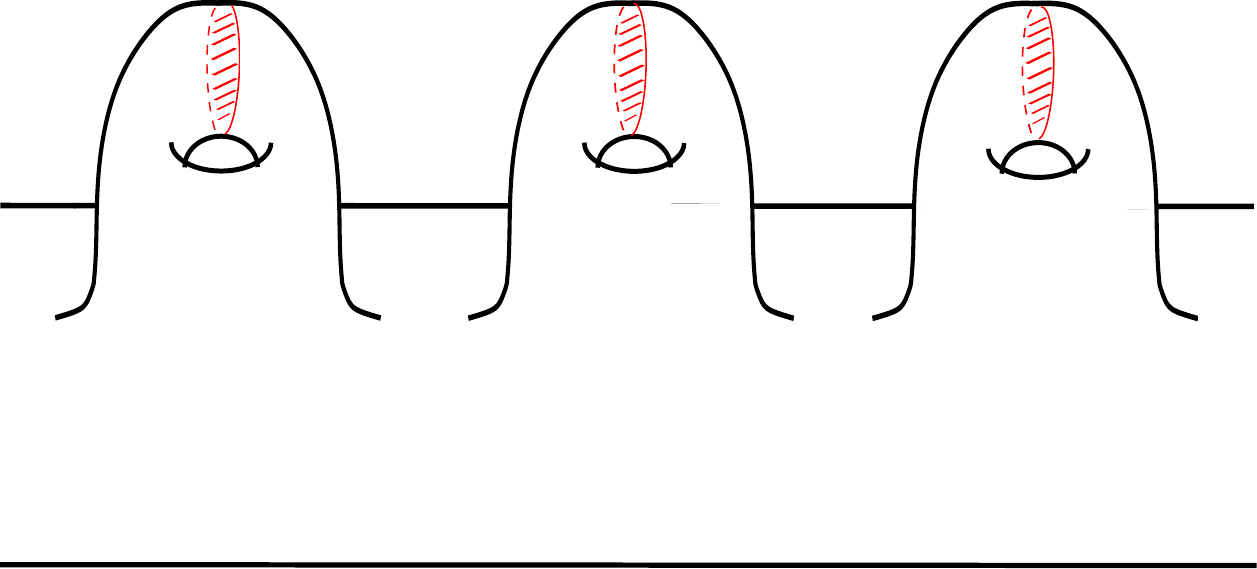}
    \caption{The doubles of the red discs generate the second Borel--Moore homology of the doubled handlebody.}
    \label{fig:BorelMoore}
\end{figure}

In particular, there is a natural bijection between $H_2^{BM}(M_{\Gamma}, \Z/2)$ and $2^{\text{rk}(\Gamma)}$. If $\text{rk}(\Gamma)=\infty$, we can make $H_2^{BM}(M_{\Gamma}, \Z/2)$ into a topological group so that it is homeomorphic to the Cantor set. We will see later that there is an isomorphism (as topological groups) between $H_2^{BM}(M_{\Gamma}, \Z/2)$ and $\text{Twists}(M_{\Gamma})$, see Lemma \ref{CrossedHomContinuous}.

\subsection{Coarse Geometry}
In Section \ref{sec:CoarseGeo} we will discuss the coarse geometry of mapping class groups of both doubled handlebodies and graphs. We will state the relevant definitions and results here. Much of this discussion is imported from \cite{domat_coarse_2023} and \cite{domat_generating_2023}, and one can refer to there for more details.
\par 
First we introduce the class of groups that we will consider.
\begin{definition}
    A topological group $G$ is \textit{Polish} if it is separable and completely metrizable.
\end{definition}
Recall that a pseudo-metric on a set $X$ is a function $d:X\times X \to \R_{\geq 0}$ which is symmetric and satisfies the triangle inequality (but is not required to have the condition that $d(x,y)=0$ implies $x=y$).
\begin{definition}[{\cite[Proposition 2.15]{rosendal_coarse_2022}}]
    Let $A$ be a subset of a topological group $G$. Then we say that $A$ is \textit{coarsely bounded (CB)} in $G$ if one of the following equivalent conditions holds.
    \begin{enumerate}
        \item (Rosendal's Criterion) For any neighborhood $\mathcal{V}$ of $\text{id}\in G$, there is a finite set $\mathcal{F}\subset G$ and an $n\geq 1$ so that $A\subset (\mathcal{F}\mathcal{V})^n$. 
        \item For all continuous group actions of $G$ on a metric space $X$ by isometries, $\text{diam}(A\cdot x)<\infty$ for all $x\in X$.
        \item For all left invariant pseudo-metrics $d$ on $G$ which is compatible with its topology, $\text{diam}_d(A)< \infty$.
    \end{enumerate}
\end{definition}

\begin{definition}
    We say that a topological group $G$ is \textit{CB generated} if there is a CB subset $A$ of $G$ which generates $G$. We also say that $G$ is \textit{locally CB} if there exists a CB neighborhood of the identity. 
\end{definition}

We will have need of the following notion of equivalence between metric spaces, which is a loosening of the more standard notion of quasi-isometry.
\begin{definition}[Coarse Equivalence]
    Let $f:(X, d_X)\to (Y, d_Y)$ be a map between pseudo-metric spaces. Then $f$ is \textit{coarsely Lipschitz} if there exists a non-decreasing function $\Phi_+:[0, \infty)\to [0, \infty)$ so that for all $x_1, x_2\in X$, 
    \begin{equation*}
        d_Y(f(x_1), f(x_2)) \leq \Phi_+(d_X(x_1,x_2)).
    \end{equation*}
    Similarly, $f$ is \textit{coarsely expanding} if there is a non-decreasing function $\Phi_-:[0,\infty)\to [0,\infty]$ with $\lim_{t\to \infty}\Phi_-(t)=\infty$ so that
    \begin{equation*}
       \Phi_-(d_X(x_1,x_2))\leq d_Y(f(x_1), f(x_2))
    \end{equation*}
    We say that $f$ is a \text{coarse embedding} if it is coarsely Lipschitz and coarsely expanding. Further, $f$ is said to be \textit{coarsely surjective} if there is a $C\geq 0$ so that for all $y\in Y$ there is some $x\in X$ with $d_Y(f(x),y)\leq C$. Finally, $f$ is a \textit{coarse equivalence} if it is a coarsely surjective coarse embedding.
\end{definition}

Details on much of the following definitions and results in the rest of this subsection can be found in \cite{rosendal_coarse_2022}, with some brief discussion in \cite{domat_coarse_2023} and \cite{domat_generating_2023} as well.

\begin{theorem}[{\cite[Theorem 1.2]{rosendal_coarse_2022}}]\label{CBGeneratedQIType}
    Let $G$ be a CB generated Polish group with two word metrics $d_1$ and $d_2$ arising from CB generating sets. Then $\text{id}:(G,d_1)\to (G,d_2)$ is a quasi-isometry. 
\end{theorem}
    
\begin{proposition}\label{CoarseEquivBothCB}
    If $G_1$ and $G_2$ are two coarsely equivalent groups, then $G_1$ is CB if and only if $G_2$ is CB.
\end{proposition}
We remark that the previous proposition holds more generally in the category of coarse spaces, though we do not need this generality.

\begin{proposition}[{\cite[Proposition 4.33]{rosendal_coarse_2022}}]\label{CBKernelCoarseEquiv}
    Let $K$ be a closed normal subgroup of a Polish group $G$. Suppose $K$ is CB in $G$. Then the quotient map $\pi:G \to G/K$ is a coarse equivalence. 
\end{proposition}

\begin{definition}
    A topological group $G$ has \textit{arbitrarily small subgroups} if for any neighborhood $V$ of the identity, there is a subgroup $H$ of $G$ with $H\subset V$.
\end{definition}

\begin{definition}
    A pseudo-metric $d$ on a Polish group $G$ is said to be \textit{coarsely proper} if for all $R\geq 0$, the metric ball $B_R(\text{id})$ is CB in $G$.
\end{definition}

We arrive at the following result of interest, which will be relevant in defining and comparing the coarse geometry of mapping class groups of graphs and their associated doubled handlebodies.

\begin{proposition}[{\cite[Proposition 2.25 and Proposition 2.27]{domat_coarse_2023}}]\label{locallyCBCoarseType}
    Let $G$ be a separable, metrizable, locally CB topological group with arbitrarily small subgroups. Then the following holds.
    \begin{enumerate}
        \item $G$ admits a continuous left-invariant and coarsely proper pseudo-metric.
        \item If $d$ and $d'$ are continuous, left-invariant, pseudo-metrics on $G$ and $d$ is coarsely proper, then the identity map $\text{Id}:(G,d)\to (G,d')$ is coarsely Lipschitz. In particular, if both $d$ and $d'$ are coarsely proper, then the identity map is a coarse equivalence.
    \end{enumerate}
\end{proposition}
In particular, a group as in Proposition \ref{locallyCBCoarseType} has a well defined class of pseudo-metrics, which are all coarsely equivalent to each other. We shall see that all the groups we are interested in satisfy the conditions of this proposition when they are locally CB.
\par 
We end with some discussion of asymptotic dimension. We begin with the definition.
\begin{definition}\label{AsymptotidDimensionDef}
    A pseudo-metric space $(X,d)$ has \textit{asymptotic dimension at most} $n$, denoted $\text{asdim}(X)\leq n$, if for every $R>0$ there is a cover $\mathcal{V}$ of $X$ by uniformly bounded sets so that every $R$ ball in $X$ intersects at most $n+1$ elements of $\mathcal{V}$. We say $\text{asdim}(X)=n$ if $\text{asdim}(X)\leq n$ and $\text{asdim}(X)\nleq n-1$.
\end{definition}

We will never work with the definition of asymptotic dimension directly. The main fact we need is the following.

\begin{lemma}[{\cite[Proposition 22]{bell_asymptotic_2008}}]\label{CoarseEquivSameAsDim}
  If $f:X\to Y$ is a coarse embedding, then $\text{asdim}(X)\leq \text{asdim}(Y)$. Thus, if $f$ is a coarse equivalence, then $\text{asdim}(X)=\text{asdim}(Y)$.  
\end{lemma}

We note one final definition and result which we will use to produce quasi-isometries between groups.

\begin{definition}[{\cite[Definition 2.43]{rosendal_coarse_2022}}]\label{def:coarselyProperMap}
    A map $\phi:G\to H$ between topological groups is \textit{coarsely proper} if for any CB set $A\subset H$, $\phi^{-1}(A)$ is CB in $G$.
\end{definition}

We then have the following Milnor--Schwartz type lemma. Here we give a greatly simplified version of it in terms of Polish groups to keep our definitions to a minimum, but a more general version is given in \cite{rosendal_coarse_2022}.

\begin{lemma}[{\cite[Theorem 2.77]{rosendal_coarse_2022}}]\label{lem:milnorSchwartzRosendal}
    Suppose $G$ is a Polish group, and $H$ is a CB generated Polish group with word metric $d_H$ given by a CB generating set. Let $\phi:G \to H$ be a coarsely proper surjective continuous homomorphism. Then $G$ is CB generated and $\phi$ is a quasi-isometry between $(G,d_G)$ and $(H, d_H)$, where $d_G$ is any word metric on $G$ given by a CB generating set.
\end{lemma}

\section{Mapping class groups}\label{sec:MCG}
\subsection{Mapping class groups of doubled handlebodies and graphs}
We now define the primary objects of interest. Fix a smooth structure on $M_{\Gamma}$. Let $\text{Diff}^+(M_{\Gamma})$ denote the group of orientation preserving diffeomorphisms of $M_{\Gamma}$. We equip it with the compact open topology.
\begin{definition}
    The \textit{mapping class group} of $M_{\Gamma}$, denoted $\text{Map}(M_{\Gamma})$ is $\text{Diff}^+(M_{\Gamma})/\sim$, where two diffeomorphisms are equivalent if there is an isotopy between them. We give $\text{Map}(M_{\Gamma})$ the quotient topology.
\end{definition}

We will occasionally work with doubled handlebodies with finitely many balls removed. The mapping class groups of these manifolds are defined in the same way, but we assume that the boundary is fixed by every map and every isotopy.
\par 
Recall that in dimension $3$, the group arising from quotienting the group of orientation preserving homeomorphisms $\text{Homeo}^+(M_{\Gamma})$ by isotopy is isomorphic to the mapping class group as defined above. We will make use of this when discussing the topology of $\text{Map}(M_{\Gamma})$. See the book of Moise for more on this \cite{moise_geometric_1977}, along with results of Hatcher \cite{hatcher_linearization_1980}\cite{hatcher_proof_1983}. We can thus write 
\begin{equation*}
    \text{Map}(M_{\Gamma})=\text{Homeo}^+(M_{\Gamma})/\text{Homeo}_0(M_{\Gamma})
\end{equation*}
where $\text{Homeo}_0(M_{\Gamma})$ is the path component of the identity.
\par 
 Given $K\subset M_{\Gamma}$ a compact submanifold, we define
\begin{equation*}
    \mathcal{V}_K=\{[f]\in \text{Map}(M_{\Gamma})| \exists g\in [f] \text{ so that } g|_K=id\}.
\end{equation*}
In Subsection \ref{subsec:Topology} we show that the topology generated by the $\mathcal{V}_K$ sets is the same as the quotient topology coming from the compact open topology on $\text{Homeo}^+(M_{\Gamma})$. The proof does not depend on other results in this paper, and we will make use of the equivalence of the two topologies without reference in the rest of the paper.
\par
Typically we will drop the brackets on elements of the mapping class group, for simplicity. 
\par 
We will study the subgroup of \textit{sphere twists} of $\text{Map}(M_{\Gamma})$. Sphere twists are defined as follows. First fix a generator $\ell:[0,1] \to \text{SO}(3)$ of $\pi_1(\text{SO}(3), id)\cong \Z/2$. Take a smoothly embedded essential non-peripheral sphere $S$ of $M_{\Gamma}$. Fix a regular neighborhood $N \cong S\times [0,1]$ of $S$. The sphere twist $T_S$ is defined as
\begin{equation*}
    T_S(p) = \begin{cases}
        (\ell(t)\cdot x, t) & \text{if } p=(x,t)\in N \\ 
        id & p\notin N 
    \end{cases}
\end{equation*}
where $\ell(t)\cdot x$ denotes the action of $\ell(t)\in \text{SO}(3)$ on $x\in S\cong S^2 \subset \R^3$. We denote by $\text{Twists}(M_{\Gamma})$ the subgroup of $\text{Map}(M_{\Gamma})$ generated by compositions of the sphere twists on (potentially infinite) sphere systems of $M_{\Gamma}$. We equip $\text{Twists}(M_{\Gamma})$ with the subspace topology.
\par 
In \cite{algom-kfir_groups_2021}, the authors proposed a definition for the mapping class group of an infinite, locally finite graph $\Gamma$.

\begin{definition}
        A map $f:\Gamma \to \Gamma$ is a \textit{proper homotopy equivalence} if $f$ is proper and a homotopy equivalence, and there exists a proper map $g:\Gamma \to \Gamma$ so that $fg$ and $gf$ are properly homotopic to the identity. Let $\text{PHE}(\Gamma)$ be the group of proper homotopy equivalences of $\Gamma$, equipped with the compact open topology. 
        \par 
        The \textit{mapping class group} $\text{Map}(\Gamma)$ of $\Gamma$ is defined as
        $$\text{Map}(\Gamma)=\text{PHE}(\Gamma)/\text{proper homotopy}.$$
        We give $\text{Map}(\Gamma)$ the quotient topology.
\end{definition}
Note that not every homotopy equivalence of $\Gamma$ which is proper is a proper homotopy equivalence as defined above. See Example 4.1 of \cite{algom-kfir_groups_2021}.
\par 
The relationship between the groups $\text{Map}(M_{\Gamma})$ and $\text{Map}(\G)$ is well understood in the case of when $\Gamma=\Gamma_{n,r}$ is a finite rank graph of rank $n$ with $r$ rays. Namely, we have the following result, which follows easily from Proposition $1$ of \cite{hatcher_homology_2004}. This is the finite type version of Theorem \ref{SESMainTheorem}. Note that the result in \cite{hatcher_homology_2004} is not phrased in terms of mapping class groups as how we have defined them, but everything is easily translated. Also, their result is for the pure mapping class group, but it clearly extends to the full mapping class group.
\begin{proposition}\label{HatcherVSES}
    There is a short exact sequence
    $$1\to \mathrm{Twists}(M_{\Gamma_{n,r}})\to \mathrm{Map}(M_{\Gamma_{n,r}})\to \mathrm{Map}(\Gamma_{n,r})\to 1.$$
\end{proposition}
\par 
\begin{remark}\label{Rmk:SameMap}
    The map from $\mathrm{Map}(M_{\Gamma_{n,r}})$ to $\mathrm{Map}(\Gamma_{n,r})$ is defined in the same way as the map $\Psi$ that we will define in Proposition \ref{MCGHom}. We will make use of this fact in the proof of that result, as well as in Lemma \ref{lem:rank1KernelStructure}.
\end{remark}
Note here that $\text{Twists}(M_{\Gamma_{n,r}})$ is generated by $n$ sphere twists on the core spheres of $M_{\Gamma_{n,r}}$, i.e. the unique sphere in each factor of $S^2\times S^1$ whose connect sum (with added punctures) makes up $M_{\Gamma_{n,r}}$. Further, if we replace the punctures of $M_{\Gamma_{n,r}}$ with sphere boundary components, we only need to add $r$ many more sphere twists, generated on spheres that are parallel to the boundary components (in fact, only $r-1$ of these are needed, see Figure 2.1 in \cite{hatcher_stabilization_2010} and the discussion before it). We also remark that it is possible to obtain this result from the methods of this paper, via a simplified argument following the proof of Theorem \ref{SESMainTheorem} (though we utilize Proposition \ref{HatcherVSES} in our proofs).
\par 
It is a consequence of work of Algom-Kfir--Bestvina that the topology on $\text{Map}(\Gamma)$ is generated by the following basis at the identity \cite{algom-kfir_groups_2021}. Here $K$ denotes a compact subgraph of $\Gamma$.\begin{equation*}
\mathcal{V}_K=\{[f]\in \text{Map}(\Gamma)\ | \ \exists g\in [f] \text{ satisfying conditions (1)-(4)}\}
\end{equation*}
where here, conditions (1)--(4) are
\begin{enumerate}[label = (\arabic*)]
    \item $g=\text{id}$ on $K$,
    \item $g$ preserves the complementary components of $K$,
    \item there is a representative $k\in [f]^{-1}$ so that $k$ also satisfies (1) and (2),
    \item there are proper homotopies of $gk$ and $kg$ to the identity map that are stationary on $K$ and preserve the complementary components of $K$.
\end{enumerate}
We use the notation $\mathcal{V}_K$ for subgroups in both $\text{Map}(M_{\Gamma})$ and in $\text{Map}(\Gamma)$, but it should always be clear in context which one meant.
\par 
The primary goal of this paper is to establish a relationship between the mapping class groups of $\Gamma$ and its doubled handlebody $M_{\Gamma}$. To understand this relationship, we first need to better understand the mapping class group of $M_{\Gamma}$. We have the following result, analogous to the classification statement of Richards for surfaces \cite{richards_classification_1963}.
\begin{proposition}\label{RichardsHandlebody}
    Given an isomorphism $$\phi:(r, E(M_{\Gamma}), E_{\ell}(M_{\Gamma}))\to (r, E(M_{\Gamma'}), E_{\ell}(M_{\Gamma'})),$$ there is a homeomorphism $f:M_{\Gamma}\to M_{\Gamma'}$ which induces the map $\phi$.
\end{proposition}
\begin{proof}
    Let $\Sigma$ denote the surface corresponding to $\Gamma$ (as defined after Definition \ref{def:DoubledHandlebody}), and $\Sigma'$ the surface associated to $\Gamma'$. The proof of this is essentially exactly the same as Richards' proof, and we will not repeat it here \cite{richards_classification_1963}. The primary idea to note is that the proof uses pairs of compact exhaustions in $\Sigma$ and $\Sigma'$, which can be defined by their boundary components. These boundary components are curves in $\Sigma$ and $\Sigma'$ which we may assume bound discs inside the handlebodies associated to $\Sigma$ and $\Sigma'$. Such a choice of boundary components is possible due to the inductive construction in \cite{richards_classification_1963}. Indeed, the construction of this pair of compact exhaustions depends on a choice of two arbitrarily chosen compact exhaustions (in the notation of \cite{richards_classification_1963}, these are the compact exhaustions $\{B_i\}$ and $\{B_i'\}$), along with choices of curves which divide certain compact surfaces up in prescribed ways (these curves correspond to the various applications of Lemma (C) in \cite{richards_classification_1963}). For both the boundary components of the arbitrarily chosen compact exhaustions and the curves chosen when applying Lemma (C), we may assume they all bound discs in the associated handlebodies.
    \par 
    Such curves then give rise to spheres in $M_{\Gamma}$ and $M_{\Gamma'}$ by bounding them off by discs in both of the pairs of handlebodies making up $M_{\Gamma}$ and $M_{\Gamma'}$. Given these sequences of compact exhaustions of $M_{\Gamma}$ and $M_{\Gamma'}$, we can define sequences of homeomorphisms as Richards does for these compact doubled handlebodies. Indeed, this is possible as the boundary components of a doubled handlebody can be permuted in any way desired by homeomorphisms, so in particular this sequence of homeomorphisms can be chosen so that the sphere corresponding to a given boundary curve is sent to the sphere corresponding to the image of the boundary curve under the homeomorphism constructed in \cite{richards_classification_1963} for $\Sigma$ and $\Sigma'$. This implies that we can choose $f$ inducing the map $\phi$.
\end{proof}

We note that this along with Lemma \ref{EndsHomeo} implies that given any two graphs $\Gamma$ and $\Gamma'$ which are proper homotopy equivalent, their doubled handlebodies $M_{\Gamma}$ and $M_{\Gamma'}$ are homeomorphic (even if the graphs are not assumed to be in standard form).
    \par 
We define the \textit{pure mapping class group} of $\Gamma$, denoted $\text{PMap}(\Gamma)\leq\text{Map}(\Gamma)$, to be the subgroup of elements acting trivially on $E(\Gamma)$. Similarly, the pure mapping class group of $M_{\Gamma}$, denoted $\text{PMap}(M_{\Gamma})\leq \text{Map}(M_{\Gamma})$, is the subgroup of elements acting trivially on $E(M_{\Gamma})$. We will also consider the subgroups of these groups which can be realized by a compactly supported homeomorphism or proper homotopy equivalence, which we will denote by $\text{Map}_c(M_{\Gamma})\leq \text{Map}(M_{\Gamma})$ and $\text{Map}_c(\Gamma)\leq \text{Map}(\Gamma)$.

\subsection{Constructing the homomorphism}\label{subsec:ConstructHom}
In Proposition \ref{MCGHom} we will produce a surjective homomorphism from $\text{Map}(M_{\Gamma})$ to $\text{Map}(\Gamma)$. It will be induced by analyzing what an element of $\text{Map}(M_{\Gamma})$ does to the copy of $\Gamma$ in $M_{\Gamma}$. We give a couple of examples to motivate the idea. 
\par 
Before we do this, we note the following helpful picture for visualizing $M_{\Gamma}$. If one cuts along the core spheres of $M_{\Gamma}$, the resulting manifold is homeomorphic to $S^3$ with a collection of balls removed, along with a set of points homeomorphic to $E(M_{\Gamma})$. The sphere boundary components come in pairs, and if we glue them back together we get $M_{\Gamma}$ back. See Figure \ref{fig:doubledHandlebodyasS^3}.
\begin{figure}
    \centering
    \includegraphics[scale=.5]{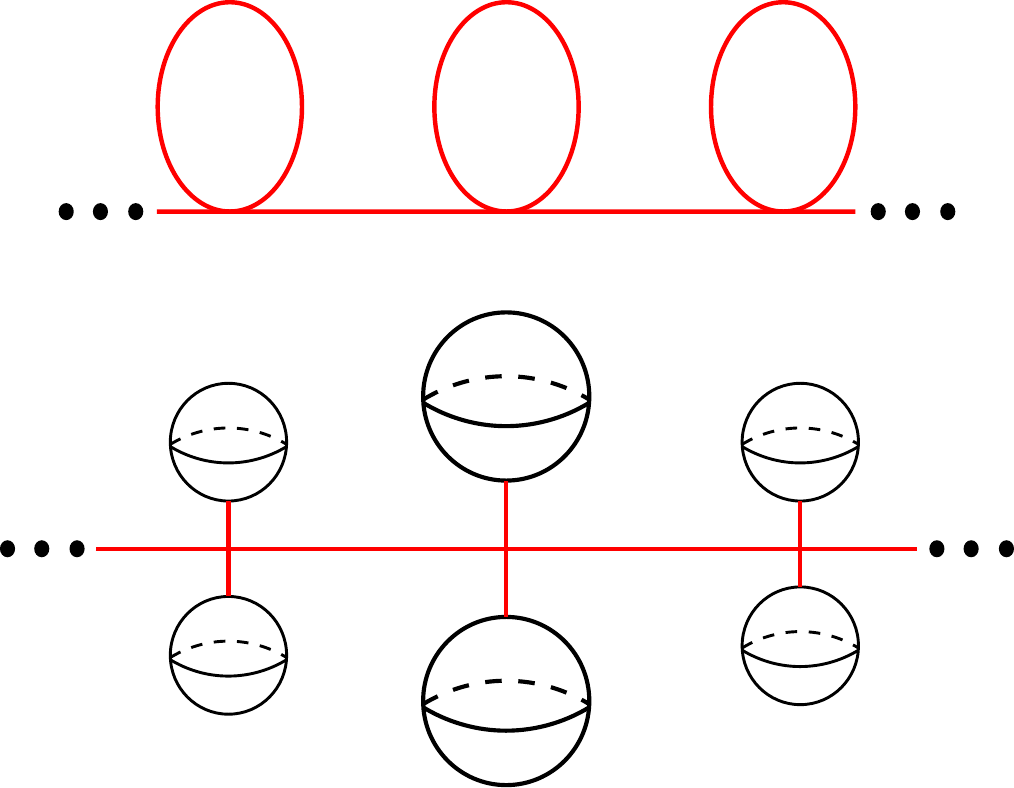}
    \caption{Here is a graph $\Gamma$ with its doubled handlebody. Pairs of vertically aligned spheres are identified, and these pairs accumulate to two removed points. The red object in the bottom picture is a copy of $\Gamma$. For simplicity in the figures, we not be too precise on how the sphere pairs are glued, although it should always be done in an orientation preserving way.}
    \label{fig:doubledHandlebodyasS^3}
\end{figure}
\par 
We consider first the word maps on graphs. These are defined in Subsection 3.3 of \cite{domat_coarse_2023}, see Figure 6 in particular. These maps are analogous to point push maps (though an actual point push map is homotopic to the identity). Fix an oriented interval $I$ around a point $p$ in the interior of an edge of $\Gamma$. A word map then monotonically maps $I$ around an immersed path $\ell$ whose starting and end points are at the endpoints of $I$, so that $\ell\cup I$ represents a word $w$ in the fundamental group at the initial point of $I$. This can be realized by a "sphere push" in $M_{\Gamma}$ around $w$. These can be visualized as maps which take the sphere 
$r^{-1}(p)$ and drags it along an embedded loop whose $r$ image is a copy of the closed loop $\ell\cup I$ based at $p$. See Figure \ref{fig:spherePush}.
\begin{figure}
    \centering
    \includegraphics[scale=.6]{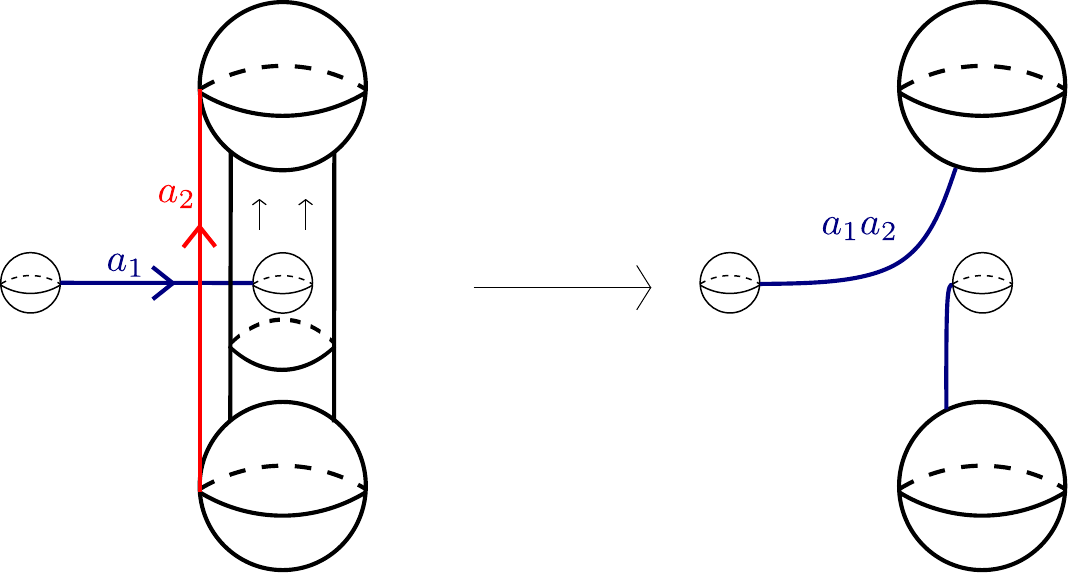}
    \caption{The sphere push pictured drags the sphere once around the solid torus via a map that tapers off to the identity near the boundary of the torus (drawn as a cylinder on the left). The loop labeled $a_1$ on the left gets mapped to $a_1a_2$ on the right. }
    \label{fig:spherePush}
\end{figure}
\par 
Another important type of map is loop shifts. See Subsection 3.4 of \cite{domat_coarse_2023}. Let us suppose that $\Gamma$ is in standard form. Given an embedded copy $\Lambda$ of the graph pictured in Figure \ref{fig:doubledHandlebodyEx} in $\Gamma$, we can take a small neighborhood of it, and shift the loops $\Lambda$ one spot to the right, taking points in the neighborhood of $\Lambda$ along with to obtain a continuous map. We can realize these in $M_{\Gamma}$ as "cylinder shifts". Namely, we take a cylinder in $M_{\Gamma}$ which cuts off the neighborhood of $\Lambda$, and inside of it we shift the sphere pairs in such a way that it tapers off to the identity as one approaches the cylinder. See Figure \ref{fig:cylinderShift}.
\begin{figure}
    \centering
    \includegraphics[scale=.6]{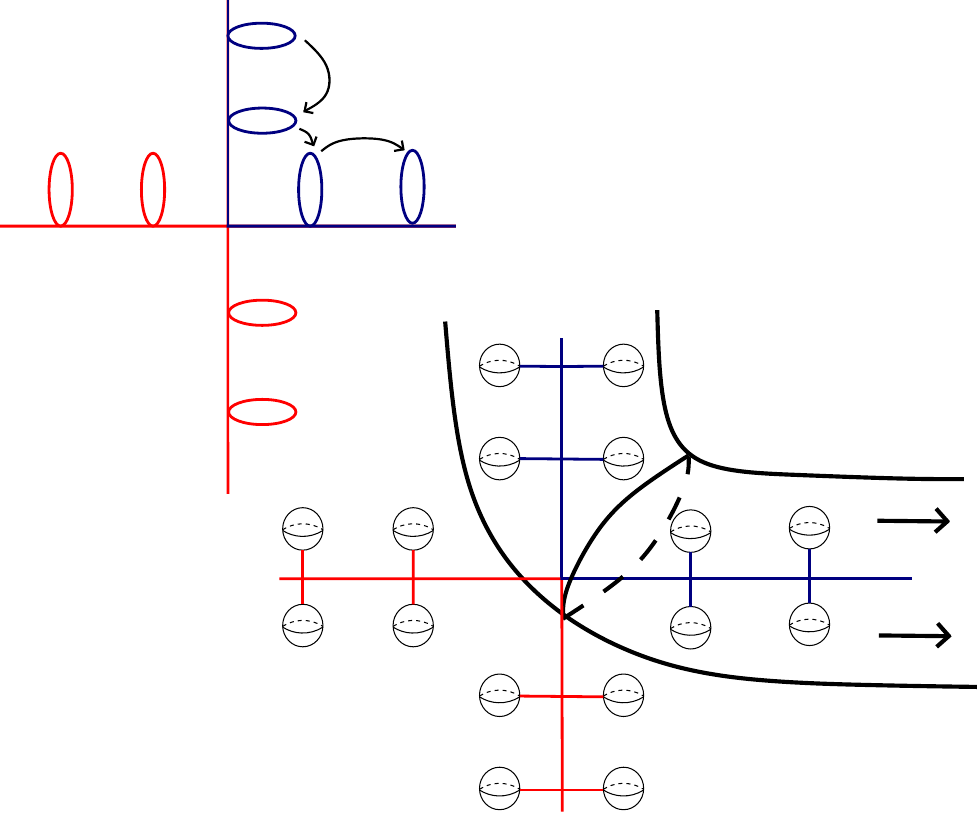}
    \caption{In the bottom right is a cylinder shift corresponding to a shift along the blue copy of $\Lambda$ in the graph on the top left, in the direction indicated by the arrows in the top left picture. In the bottom right, the sphere pairs are pushed along in the interior of the cylinder in the direction indicated by the arrows. }
    \label{fig:cylinderShift}
\end{figure}
\par 
We now construct the homomorphism from $\text{Map}(M_{\Gamma})$ to $\text{Map}(\Gamma)$ which appears in Theorem \ref{SESMainTheorem}.

\begin{proposition} \label{MCGHom}
There is a map $\Psi:\text{Map}(M_{\Gamma})\to \text{Map}(\Gamma)$ which restricts to a map $\Psi_P:\text{PMap}(M_{\Gamma})\to \text{PMap}(\Gamma)$ with the following properties.
\begin{enumerate}[label=(\alph*)]
    \item The maps $\Psi$ and $\Psi_P$ are continuous homomorphisms.
    \item The action of $f\in \text{Map}(M_{\Gamma})$ on $E(M_{\Gamma})$ is the same as the action of $\Psi(f)$ on $E(\Gamma)$ (with respect to the identifications of $E(M_{\Gamma})$ and $E(\Gamma)$ given in Remark \ref{Identifications}).
    \item $\Psi$ and $\Psi_P$ are surjective.
    \item $\Psi$ and $\Psi_P$ are open maps, and in particular are also quotient maps.
    \item The image of $f \in \text{Map}(M_{\Gamma})$ in $\text{Out}(\pi_1(\Gamma))$ is the same as the image of $\Psi(f)$.
    \item Given an end $\alpha \in E(\Gamma)\setminus E_{\ell}(\Gamma)$, the isomorphism induced by $f \in \text{Map}(M_{\Gamma})$ from $\pi_1(\Gamma, \alpha)$ to $\pi_1(\Gamma, f(\alpha))$ is the same as the isomorphism induced by $\Psi(f)$.
\end{enumerate}
\end{proposition}
\begin{proof}
\phantom{Starting line}
    \begin{enumerate}[label=(\alph*)]
        \item We let $\Psi$ and $\Psi_P$ be the maps induced by the map sending $f\in \text{Diff}^+(M_{\Gamma})$ to $r\circ f \circ i$, which we'll write as $rfi$. The map $rfi$ is a proper map from $\Gamma$ to itself, although it is not immediately clear that it is a proper homotopy equivalence. To show this, we will first show that, given another map $g\in \text{Diff}^+(M_{\Gamma})$, the maps $rfirgi$ and $rfgi$ are properly homotopic. Once this is known, it will follow that a proper homotopy inverse for $rfi$ is given by $rf^{-1}i$, and since $rfirgi$ and $rfgi$ are properly homotopic, it will follow that $\Psi$ is a actually a homomorphism from $\text{Map}(M_{\Gamma})$ to $\text{Map}(\Gamma)$. Further, it is clear from how $\Psi$ is defined that the map $\Psi_P$ will also have the same desired properties (strictly speaking we haven't showed that the image of $\Psi_P$ lies in $\text{PMap}(\Gamma)$, but this will follow from part (b)).
        \par 
        To give a proper homotopy from $rfirgi$ to $rfgi$, we first properly homotope $gi:\Gamma \to M_{\Gamma}$ to a map $g'$ whose image lies in $\Gamma$. It will then follow that $rfirgi$ is properly homotopic to $rfirg'=rfg'$ which is properly homotopic to $rfgi$. 
        \par 
        We build the proper homotopy from $gi$ to $g'$ as follows (compare to the proof of Lemma \ref{FundGroupIso}). Namely, first on the vertices of the graph $g(\Gamma)$ we homotope so that they lie in the interior of $B$. Then the new intersection of $g(\Gamma)$ with $B'$ consists of a collection of properly embedded arcs. We may homotope these arcs so that they are disjoint from $\Gamma'$, and then flow away from $\Gamma'$ into $\Gamma$ to obtain the map $g'$ as desired. This homotopy is proper as the collection of intersecting arcs in $B'$ is proper, as $g$ is a homeomorphism. This implies the desired result, and in particular we know that $\Psi$ and $\Psi_P$ are homomorphisms.
        \par 
         To see continuity, take a compact connected subgraph $K$ of $\Gamma$. For simplicity, assume $K$ is chosen so that $r^{-1}(K)$ is a submanifold of $M_{\Gamma}$. This can be done by assuming that $K$ is the image of a union of pieces. Consider $\Psi^{-1}(\mathcal{V}_K)$. As this is a subgroup of $\text{Map}(M_{\Gamma})$, it suffices to show that any sequence $\{f_n\}$ in $\text{Map}(M_{\Gamma})$ converging to $\text{id}$ is eventually in $\Psi^{-1}(\mathcal{V}_K)$. By assumption, given any compact submanifold of $M_{\Gamma}$, for all sufficiently large $n$, $f_n$ can be isotoped so that it is the identity on this submanifold. In particular, this is true for $r^{-1}(K)$. But then it is clear for such $n$ that $\Psi(f_n)\in \mathcal{V}_K$, proving the result. The same proof works for $\Psi_P$.
        \item This follows as $r$ and $i$ induce the identity map on the ends, and $E$ is a functor.
        \item Once we show that $\Psi_P$ is surjective, it follows from Proposition \ref{RichardsHandlebody} and part (b) that $\Psi$ is surjective as well. Theorem B of \cite{domat_generating_2023} implies that $\text{PMap}(\Gamma)$ is generated by closure of the compactly supported elements $\overline{\text{Map}_c(\Gamma)}$ and a subgroup generated by a collection of commuting loop shifts of $\Gamma$. To show that $\Psi_P$ is surjective, it suffices to show that all maps in either subgroup are in the image of $\Psi_P$.  
        \par 
        Let $\phi \in \overline{\text{Map}_c(\Gamma)}$. We fix a sequence of compactly supported elements $g_n$ converging to $\phi$. By Proposition \ref{HatcherVSES} and Remark \ref{Rmk:SameMap}, we can find compactly supported elements $f_n\in \text{Map}_c(M_{\Gamma})$ so that $\Psi(f_n)=g_n$, and the only ambiguity in the choice of $f_n$ is a finite collection of sphere twist in a compact submanifold of $M_{\Gamma}$. As $\{g_n\}$ converges to $\phi$, the elements of this sequence eventually coincide on any given compact subgraph $K$ of $\Gamma$. By possibly composing by sphere twists, we may then assume that the elements of $\{f_n\}$ eventually coincide on $r^{-1}(K)$, up to perhaps slightly expanding $K$ so that $r^{-1}(K)$ is a union of pieces. It follows that $\{f_n\}$ converges to some $\eta \in \text{Map}(M_{\Gamma})$. By continuity, $\Psi(\eta)=\phi$.
        \par 
        Given a loop shift of $\Gamma$, one can realize it explicitly as the image of a cylinder shift, which we discussed above. In Proposition 3.20 of \cite{domat_generating_2023}, they realize the collection of generating loop shifts as maps whose supports are in disjoint small neighborhoods of copies of $\Lambda$ (recall this is the graph pictured in Figure \ref{fig:doubledHandlebodyEx}) in $\Gamma'$, a locally finite graph properly homotopy equivalent to $\Gamma$. We can then realize every such loop shift as the image of a cylinder shift in $M_{\Gamma'}$, which is homeomorphic to $M_{\Gamma}$. By continuity of $\Psi_P$, it follows that we can realize an infinite product of commuting loop shifts as the image of a limit of products of cylinder shifts. Surjectivity of $\Psi_P$, and hence $\Psi$, follows.
        \item  If $K$ is a connected union of pieces of $M_{\Gamma}$, it follows that $\Psi(\mathcal{V}_K)=\mathcal{V}_{r(K)}\subset \text{Map}(\Gamma)$. Indeed, we can treat the boundary points of $r(K)$ as ends of the complementary components. Then just repeat the proof of (c) on this complementary component, looking at the subgroup which fixes the ends corresponding to the boundary points of $r(K)$. The same proof works for $\Psi_P$.
        \item As $i$ and $r$ induce natural isomorphisms between the fundamental groups of $\Gamma$ and $M_{\Gamma}$, the result follows.
        \item This follows from the arguments of Lemma \ref{FundGroupIso}, as every element of $\pi_1(M_{\Gamma}, e)$ is realized by a line inside $\Gamma$, and every proper line in $M_{\Gamma}$ is properly homotopic to its $r$ image in $\Gamma$.
    \end{enumerate}
\end{proof}


\section{Short exact sequence}\label{sec:SES}
In this section, we give a proof of the first part of Theorem \ref{SESMainTheorem}. The proof of the second part requires Theorem \ref{TwistGroupStructure}, and will be done in Section \ref{sec:SphereTwists}. We split the proof into two cases. The first case, where we assume $\Gamma$ has rank $0$ or rank at least $2$, will be handled in Lemmas \ref{pi_2Trivial} and \ref{KernelStructure}. The rank $1$ case will be handled in Lemma \ref{lem:rank1KernelStructure}. For simplicity, let us assume that $\Gamma$ is in standard form. We can also assume that $\Gamma$ has no vertices of valence $1$, so in particular the boundary components of every piece of $M_{\Gamma}$ are essential (though potentially peripheral).
\par 
Often in this section we will cut manifolds along spheres. If $M$ is a $3$-manifold and $S$ is a sphere system, we will use the notation $M\setminus S$ to mean the manifold obtained by removing $S$, but also having boundary components which in pairs correspond to the components of $S$, in the sense that gluing these pairs to each other results in $M$ again. 
\par 
Fix an infinite locally finite graph $\Gamma$, with the associated doubled handlebody $M_{\Gamma}$. We note the following result of Laudenbach \cite{laudenbach_topologie_1974}.
\begin{theorem}[{\cite[Lemma V.4.2]{laudenbach_topologie_1974}}]\label{LaudenbachSphereIsotopy}
    Suppose $i:\sqcup_{i=1}^n S^2\to M$ and $i':\sqcup_{i=1}^n S^2\to M$ are two homotopic sphere systems in a $3$-manifold $M$ so that no two distinct spheres in $\text{im}(i)$ are homotopic. Then there is an ambient isotopy of $M$ sending $i$ to $i'$. That is, there is an isotopy $H_t:M\times [0,1]\to M$ with $H_0=id$ and $H_1\circ i = i'$.
\end{theorem}

We remark that there is a subtlety in Theorem \ref{LaudenbachSphereIsotopy}, in that it deals with homotopies and isotopies of \textit{embeddings} of spheres in $M$. This is as opposed to the definition of $\mathcal{S}(M_{\Gamma})$, which are isotopy classes of spheres considered as sets, instead of embeddings. As we will see, accounting for this difference is important in the proof of Theorem \ref{SESMainTheorem}(1).
\par 
We will apply Theorem \ref{LaudenbachSphereIsotopy}, along with the following lemma, to obtain Theorem \ref{SESMainTheorem} in the case when $\Gamma$ has rank $0$ or rank at least $2$. This lemma is due to Laudenbach in the case that $\Gamma$ is compact \cite{laudenbach_topologie_1974}. Our proof is based on the proof of the same result by Brendle--Broaddus--Putman. \cite{brendle_mapping_2023}. 
\begin{lemma}\label{pi_2Trivial}
    Suppose $\Gamma$ either has rank $0$ or rank at least $2$. Let $f:M_{\Gamma}\to M_{\Gamma}$ be an element of $\text{Diff}^+(M_{\Gamma})$, and fix a point $x_0\in \Gamma$. Suppose $[f]\in \text{ker}(\Psi)$ and $f_*:\pi_1(M_{\Gamma}, x_0)\to \pi_1(M_{\Gamma}, x_0)$ is the identity. Then $f$ acts trivially on $\pi_2(M_{\Gamma}, x_0)$.
\end{lemma}
\begin{proof}
    We lift to the universal cover $\widetilde{\Gamma} \subset \widetilde{M_{\Gamma}}$ with covering map $\pi:\widetilde{M_{\Gamma}}\to M_{\Gamma}$ and with a fixed lift $\widetilde{x_0}$ of $x_0$. Let $\tilde{f}:\widetilde{M_{\Gamma}}\to \widetilde{M_{\Gamma}}$ be the unique lift of $f$ fixing $\widetilde{x_0}$. We also denote by $\widetilde{i}:\widetilde{\Gamma}\to \widetilde{M_{\Gamma}}$ and $\widetilde{r}:\widetilde{M_{\Gamma}}\to \widetilde{\Gamma}$ the inclusion and retraction maps so that $\pi \circ \widetilde{i} = i \circ \pi$ and $\pi \circ \widetilde{r} = r \circ \pi$.  To show that $f$ acts trivially on $\pi_2(M_{\Gamma}, x_0)$, it suffices to show that $\widetilde{f}$ acts trivially on $\pi_2(\widetilde{M_{\Gamma}}, \widetilde{x_0})\cong H_2(\widetilde{M_{\Gamma}})$, by the Hurewicz Theorem.
    \par 
    By Poincare duality and the reduced long exact sequence for relative cohomology, we have
    $$H_2(\widetilde{M_{\Gamma}})\cong H_c^1(\widetilde{M_{\Gamma}})=\varinjlim_{K} H^1(\widetilde{M_{\Gamma}}, \widetilde{M_{\Gamma}}\setminus K)\cong \varinjlim_{K} \tilde{H}^0(\widetilde{M_{\Gamma}}\setminus K).$$
    Here, the limits are taken over compact subsets of $\widetilde{M_{\Gamma}}$. The final map is induced by the boundary maps of the long exact sequence for relative cohomology with the pair $(\widetilde{M_{\Gamma}}, \widetilde{M_{\Gamma}}\setminus K)$, each of which is an isomorphism because $\pi_1(\widetilde{M_{\Gamma}})$ is trivial and $\widetilde{M_{\Gamma}}$ is connected. 
    \par 
    Elements of $\tilde{H}^0(\widetilde{M_{\Gamma}}\setminus K)$ can be interpreted as the locally constant functions $\kappa:\widetilde{M_{\Gamma}}\setminus K \to \Z$, quotiented by the constant functions. Fix such a $\kappa$ defined on $\widetilde{M_{\Gamma}}\setminus K$, and let $K'$ be a compact set containing $K\cup \tilde{f}(K)$ so that no component of $\widetilde{M_{\Gamma}}\setminus K'$ is bounded, and $K'=\widetilde{r}^{-1}(L)$ for some compact subgraph $L$ of $\widetilde{\Gamma}$. Then the image in $\tilde{H}^0(\widetilde{M_{\Gamma}}\setminus K)$ of the $\tilde{f}$ image of the element of $H_2(\widetilde{M_{\Gamma}})$ representing $\kappa$ is represented by
    $$\kappa \circ \tilde{f}^{-1}:\widetilde{M_{\Gamma}}\setminus K' \to \Z.$$
    It thus suffices to show that $\kappa = \kappa\circ \tilde{f}^{-1}$ on $\widetilde{M_{\Gamma}}\setminus K'$. To do this, we show that $\widetilde{f}$ fixes the ends of $\widetilde{M_{\Gamma}}$. Once we know this, then points that are far from $K'$ in some component $U$ of $\widetilde{M_{\Gamma}}\setminus K'$ will still lie in that component. It follows then that $\kappa = \kappa \circ \widetilde{f}^{-1}$ on $U$ (as $\kappa$ and $\kappa\circ\widetilde{f}^{-1}$ are both constant on $U$), and thus on $\widetilde{M_{\Gamma}}\setminus K'$.
    \par 
    If $\Gamma$ is rank $0$ then it is simply connected and we are done, as $rfi$ is properly homotopy equivalent to the identity, and thus $f$ acts trivially on the ends. Assume that $\Gamma$ has rank at least $2$. Then the map $\widetilde{r}\widetilde{f}\widetilde{i}$ is a lift of $rfi$, and the latter map is properly homotopic to the identity by the assumption that $[f]\in \text{ker}(\Psi)$. We can lift this to a proper homotopy from $\widetilde{r}\widetilde{f}\widetilde{i}$ to a covering transformation $\gamma:\widetilde{\Gamma}\to \widetilde{\Gamma}$. The map $\widetilde{r}\widetilde{f}\widetilde{i}$ fixes ends that have lifts of the basepoint $x_0$ accumulating to it, as such lifts are fixed by this map due to the assumption that the induced map of $f$ on $\pi_1(M_{\Gamma}, x_0)$ is trivial. As the rank of $\Gamma$ is at least $2$, there are infinitely many such ends. If $\gamma$ is nontrivial it cannot fix more than $2$ ends. This follows, as we can choose a metric on $\widetilde{\Gamma}$ so that $\gamma$ acts by isometries and is a hyperbolic element. If $\gamma$ fixed $3$ ends then it would have to fix the center of a geodesic triangle with end points at these ends. It follows that $\gamma=\text{id}$, so since $\widetilde{r}$ and $\widetilde{i}$ act trivially on the ends, so does $\widetilde{f}$ and we are done.
\end{proof}


The next lemma gives a more explicit form of elements of $\text{ker}(\Psi)$. First, recall the core spheres of $M_{\Gamma}$, as 
 defined in Subsection \ref{subsec:BorelMoore}. This is a collection of nonseparating spheres in $M_{\Gamma}$, one for each piece of $M_{\Gamma}$ homeomorphic to $M_{1, s}$. This collection is not unique, but we will fix one from this point on.
\par 

\begin{lemma}\label{KernelStructure}
    Let $\Gamma$ be a locally finite graph. Fix $f\in \text{Diff}^+(M_{\Gamma})$ acting trivially on $\pi_2(M, x_0)$ for some $x_0\in M$. Then $f$ is isotopic to a product of sphere twists on the core spheres and the boundary components of the pieces of $M_{\Gamma}$. 
    In particular, if $\Gamma$ has rank $0$ or rank at least $2$, then $\mathrm{ker}(\Psi)$ consists of (possibly infinite) products of sphere twists on a spheres in a fixed sphere system.
\end{lemma}
\begin{proof}
    In this proof every sphere is actually implicitly an embedding $i:S^2\to M_{\Gamma}$. Thus when we say that two spheres are homotopic or isotopic, they are homotopic or isotopic as embeddings.
    \par 
    Fix a compact exhaustion $\{K_i\}_{i\geq 1}$ of $M_{\Gamma}$ which consists of connected unions of pieces. We will assume that $K_i$ is not homeomorphic to $M_{0,2}$, and that the components of $\partial K_i$ are not isotopic to boundary components of $K_{i+1}$ for all $i$, unless one such component bounds a puncture on one side. We will write $K_0=\varnothing$. We also denote by $\mathcal{C}_i$ the collection of core spheres in $K_i \setminus K_{i-1}$.
    \par 
    Let $\mathcal{S}=\sqcup_i (\partial K_i\cup \mathcal{C}_i)$. By Theorem \ref{LaudenbachSphereIsotopy}, $f$ preserves the isotopy class of every component of $\mathcal{S}$. The goal is to produce an ambient isotopy of $M_{\Gamma}$ so that after applying this isotopy to $f$, every component of $\mathcal{S}$ is actually fixed by $f$. Note that we cannot apply Theorem \ref{LaudenbachSphereIsotopy} directly, as $\mathcal{S}$ has infinitely many elements. However, we will see that it can be inductively applied with care to achieve the desired result.
    \par 
    For the first step, apply Theorem \ref{LaudenbachSphereIsotopy} to obtain an isotopy $H_1:[0,1/2]\times M_{\Gamma}\to M_{\Gamma}$ so that $H_1(0, \cdot)=\text{id}$ and $H_1(1/2, f(p))=p$ for all $p\in \partial K_1 \sqcup \mathcal{C}_1$. We may further assume that if any component of $\partial K_1$ bounds a single puncture, then $H_1(1/2, f(p))=p$ for any point $p$ in the same component as this puncture. 
    \par 
    By induction, for $1\leq i \leq n$ assume we have an isotopy $H_i:[1-2^{1-i}, 1-2^{-i}]\times M_{\Gamma}\to M_{\Gamma}$ so that the following points hold. We write $f_0(p)=f(p)$ and inductively define $$f_i(p)=H_i(1-2^{-i}, f_{i-1}(p))$$
    so that
    \begin{itemize}
        \item $H_i(t, p)=f_{i-1}(p)$ for $p\in K_{i-1}$ and $t\in [1-2^{1-i}, 1-2^{-i}]$. In particular, $H_i$ restricted to $K_{i-1}$ does not depend on $t$.
        \item $f_i(p)=p$ for $p\in \partial K_i\sqcup \mathcal{C}_i$ and $f_i$ stabilizes $K_i$. 
        \item If any component of $\partial K_i$ bounds a single puncture, then $f_i(p)=p$ for any point $p$ in the same component as this puncture.
    \end{itemize}
    To build the next isotopy $H_{n+1}$ in the sequence, we will show that the sphere system $f_n(\partial K_{n+1}\sqcup \mathcal{C}_{n+1})$ is isotopic to $\partial K_{n+1} \sqcup \mathcal{C}_{n+1}$ via an ambient isotopy supported outside of $K_n$. 
    \par 
    Fix a component $S$ of $\partial K_{n+1}\sqcup \mathcal{C}_{n+1}$. Then $S$ and $f_n(S)$ are homotopic as $S$ and $f(S)$ are. Further, both spheres are disjoint from $K_n$, since by definition $S$ is, and as $f_n$ stabilizes $K_n$, the same is true for $f_n(S)$. 
    \par 
    Let $\widetilde{M_{\Gamma}}$ denote the universal cover of $M_{\Gamma}$ with $\pi: \widetilde{M_{\Gamma}}\to M_{\Gamma}$ the covering map. We can lift the homotopy between $S$ and $f_n(S)$ to one between lifts of these two spheres in $\widetilde{M_{\Gamma}}$. Let us denote these lifts by $T_1$ and $T_2$. By Theorem \ref{LaudenbachSphereIsotopy}, they are isotopic as well. Thus these spheres lie in the same component $U$ of $\widetilde{M_{\Gamma}}\setminus \pi^{-1}(K_n)$, or else there is no way they could induce the same partition of $E(\widetilde{M_{\Gamma}})$ (using the assumption that $K_i$ is not homeomorphic to $M_{0,2}$), contradicting Lemma \ref{lem:simConnectedSpherePartition}. But then they also induce the same partition of $E(U)$, thinking of its boundary as an end. Again by Lemma \ref{lem:simConnectedSpherePartition}, they are isotopic in $U$. This restricts to a homotopy from $T_1$ to $T_2$ whose image lies in $U$. Strictly speaking Lemma \ref{lem:simConnectedSpherePartition} only gives isotopy on the level of sets, instead of embeddings. However, the original homotopy between $T_1$ and $T_2$ is defined on embeddings, so it is clear that the homotopy between them whose image lies in $U$ can also be taken to be through embeddings.
    \par

    
    This homotopy descends to a homotopy from $S$ to $f_n(S)$ whose image is disjoint from $K_n$. We can do this same process for every other component of $\partial K_{n+1}\sqcup \mathcal{C}_{n+1}$. Thus by Theorem \ref{LaudenbachSphereIsotopy} there is an ambient isotopy $G_{n+1}:[1-2^{1-n}, 1-2^{-n}]\times M_{\Gamma}\to M_{\Gamma}$ sending the $f_n$ image of $\partial K_{n+1}\sqcup \mathcal{C}_{n+1}$ to $\partial K_{n+1}\sqcup \mathcal{C}_{n+1}$ itself which is supported outside of $K_n$. We can then define 
    \begin{equation*}
        H_{n+1}(t,p)=\begin{cases}
            f_n(p) & \text{if } p\in K_n  \\G_{n+1}(t, p) & \text{otherwise}
        \end{cases}
    \end{equation*}

    By construction, this is continuous as the two defining functions agree on $\partial K_n$, and it satisfies the three bullet points (we can easily guarantee the third bullet point to hold as $G_{n+1}$ can be assumed to be trivial on the punctured ball components).




    We may concatenate these isotopies together to obtain an isotopy $H:[0,1]\times M_{\Gamma} \to M_{\Gamma}$ where $H(t, p)=H_i(t,p)$ if $t\in [1-2^{1-i}, 1-2^{-i}]$ for $t<1$ and $H(1, p)= \lim_{t\to 1}H(t, p)$. This map is well defined by construction, and it is also continuous as $H(t, p)$ is eventually constant in the $t$ variable on every compact set. 
    \par 
    In particular, we obtain a map $f'(p)=H(1, f(p))$ that is isotopic to $f$ and fixes pointwise every core sphere of $M_{\Gamma}$, as well as every component of $\partial K_i$ for all $i$. Thus we can decompose $f'$ into maps supported on the components of $M_{\Gamma}\setminus \mathcal{S}$, each of which is homeomorphic to some copy of $M_{0, s}$, for $s$ depending on the component. But the mapping class group of $M_{0,s}$ is generated by sphere twists on its boundary components (see Lemma 2.5 of \cite{brendle_mapping_2023}, or Proposition \ref{HatcherVSES}), so it follows that $f$ is isotopic to a composition of sphere twists on the desired collection of spheres in $M_{\Gamma}$.
    \par 
    If we assume that $\Gamma$ has rank $0$ or rank at least $2$, then by Lemma \ref{pi_2Trivial}, every element of $\text{Ker}(\Psi)$ has a representative which acts trivially on $\pi_2(M_{\Gamma}, x_0)$. By the above, it follows then that $\text{Ker}(\Psi)$ is generated by compositions of sphere twists on the spheres of $\mathcal{S}$. 
\end{proof}

The following lemma proves the analogous result for $\Gamma$ having rank $1$.
\par 
\begin{lemma}\label{lem:rank1KernelStructure}
    Suppose $\Gamma$ is a rank $1$ graph. Let $f\in \mathrm{Diff}^+(M_{\Gamma})$ be so that $[f]\in \mathrm{ker}(\Psi)$. Then $f$ is isotopic to a product of sphere twists on the core spheres and the boundary components of the pieces of $M_{\Gamma}$. 
\end{lemma}
\begin{proof}
    Our strategy will be to reduce down to the finite type case, which is covered by Proposition \ref{HatcherVSES}, and the rank $0$ case, which is covered by Lemma \ref{KernelStructure}. As in the proof of Lemma \ref{KernelStructure}, homotopies and isotopies between spheres are actually between embeddings of spheres. We may assume that $\Gamma$ is not of finite type, as otherwise we could apply Proposition \ref{HatcherVSES} directly to show that $[f]\in \text{ker}(\Psi)$ is a product of sphere twists.
    \par 
    We may assume that $\Gamma$ is a loop and a tree $\Omega$ attached by an edge $e$. Fix $p$ an interior point of $e$ and let $S=r^{-1}(p)$. Given a sphere $T$ in the noncompact component of $M_{\Gamma}\setminus S$, we let $T^*$ denote the component of $M_{\Gamma}\setminus T$ which does not contain $S$. We need the following claim.
    \begin{claim}\label{claim:StayOutThere}
        Fix two disjoint spheres $T_1$ and $T_2$ so that $T_1^*$ and $T_2^*$ are disjoint and every end of $M_{\Gamma}$ is contained in one of these two sets. Let $\mathcal{U}$ be the collection of spheres which are the $r$ preimage of a midpoint of an edge of $\Omega$. Then for all but finitely many $U\in \mathcal{U}$, $f(U)$ lies in $T_1^*$ or $T_2^*$. 
    \end{claim}
    \begin{proof}[Proof of Claim]
        If not, it would follow that there is a collection of distinct spheres $\{U_n\}$ which accumulate to an end of $M_{\Gamma}$ so that $f(U_n)$ intersects either $T_1$ or $T_2$ for all $n$. But this is impossible by properness.
    \end{proof}
    It follows that there is a finite collection $\{U_i\}_{i=1}^m$ of spheres in $\mathcal{U}$ so that $$K=M_{\Gamma}\setminus \big( \cup_{i=1}^m U_i^*\big)$$ is a compact submanifold with sphere boundary components, and so that $f(U_i)$ is disjoint from $T_1$ and $T_2$ for all $i=1,\ldots, m$. It follows from Lemma \ref{lem:simConnectedSpherePartition} that $U_{i}$ is isotopic to $f(U_{i})$. Indeed, $f$ acts trivially on $E(\Gamma)$, and as $U_{i}$ and $f(U_{i})$ are disjoint from $T_1$ and $T_2$, $S$ (thought of as an end of the unbounded component of $M_{\Gamma}\setminus S$) lies in the complements of $U_{i}^*$ and $f(U_{i})^*$ as subsets of the unbounded component of $M_{\Gamma}\setminus S$. Thus by Lemma \ref{lem:simConnectedSpherePartition} and Theorem \ref{LaudenbachSphereIsotopy} we may isotope $f$ so that it fixes every sphere in $\{U_{i}\}$.
    \par 
    Thus we may decompose $f$ as a map on the compact submanifold $K$ which is homeomorphic to $M_{1, m}$, and on a finite collection of noncompact simply connected manifolds. Then $\Psi$ restricts to a homomorphism on each component, and the restricted $\Psi$ image of each restriction of $f$ is trivial. 
    \par 
    By Proposition \ref{HatcherVSES} and the discussion after it, we can isotope $f$ restricted to $K$ so that it is a composition of sphere twists on the core sphere of $K$ and sphere twists on spheres parallel to components of $\partial K$. By Lemmas \ref{pi_2Trivial} and \ref{KernelStructure}, we can isotope $f$ on the components of $M_{\Gamma}\setminus K$ so that it a composition of sphere twists on the boundary components of pieces of $M_{\Gamma}$.
\end{proof}
The previous proof only relied on $\Gamma$ having finite rank, there was nothing special about rank $1$ in the argument. 
\par 
We remark that, from now on, we will only consider spheres and sphere systems as embedded submanifolds, and forget about the actual embedding itself, unless stated otherwise. In particular, all homotopies and isotopies are for embedded submanifolds instead of for embeddings.
\par 
Recall that a locally finite graph $\Gamma$ is \textit{sporadic} if it is rank $0$ with at most $4$ ends, a rank $1$ graph with $0$ or $1$ end(s), or a compact rank $2$ graph. These cases are explicitly excluded from the second part of Theorem \ref{SESMainTheorem}(1). We will discuss after the proof of Theorem \ref{SESMainTheorem}(1) why they need to be excluded. 
\begin{lemma}\label{lem:peripheralSpheresFixed}
    Suppose $\Gamma$ is non-sporadic, and $[g]\in \text{Map}(M_{\Gamma})$ acts trivially on $\mathcal{S}(M_{\Gamma})$. Then it also fixes isotopy classes of embedded peripheral spheres of $M_{\Gamma}$.
\end{lemma}
\begin{proof}
     It suffices to show that $g$ fixes the ends of $M_{\Gamma}$, as a peripheral sphere is fixed if and only if the end it encloses is fixed. 
    \par 
    If $\Gamma$ has rank $0$, then by assumption $\Gamma$ has at least $5$ ends. Suppose $e\in E(M_{\Gamma})$ is such that $g(e)\neq e$. Let $S\in \mathcal{S}(M_{\Gamma})$ be a sphere containing $e$ and $g(e)$ on different sides, and so that the side containing $e$ contains at most $2$ ends. Then the side of $g(S)$ containing $g(e)$ contains at most $2$ ends, which is a contradiction by Lemma \ref{lem:simConnectedSpherePartition} as clearly $S$ and $g(S)$ do not induces the same partition on $E(\Gamma)$. 
    \par 
    If the rank of $\Gamma$ is at least $1$, then we can assume it has at least $2$ ends (the rank $2$ case with $1$ end is trivial). Fix a sphere $S\in \mathcal{S}(M_{\Gamma})$ which cuts off a once punctured copy of $M_{1,1}$, where the puncture corresponds to an end $e$ of $M_{\Gamma}$. If $g(e)\neq e$, then it follows that $g(S)\neq S$ as $S$ and $g(S)$ partition the ends of $M_{\Gamma}$ differently which is impossible if they are isotopic. This is a contradiction, finishing the proof.
\end{proof}

\begin{proof}[Proof of Theorem \ref{SESMainTheorem}(1)]
    Lemmas \ref{KernelStructure} and \ref{lem:rank1KernelStructure} show that $\text{ker}(\Psi)$ is equal to $\text{Twists}(M_{\Gamma})$ (regardless of if $\Gamma$ is sporadic or not). Now assume that $\Gamma$ is non-sporadic. To finish, we need to show that if $[g]\in \text{Map}(M_{\Gamma})$ acts trivially on $\mathcal{S}(M_{\Gamma})$, then it lies in $\text{Twists}(M_{\Gamma})$.
    \par    
    Fix a compact exhaustion of $M_{\Gamma}$ by connected unions of pieces. By assumption, $g$ fixes the isotopy classes of the core spheres of $M_{\Gamma}$ as well as the nonperipheral boundary components of the pieces. Lemma \ref{lem:peripheralSpheresFixed} shows that the isotopy classes of peripheral boundary spheres of pieces are also fixed. Thus, the proof of Lemma \ref{KernelStructure} shows that, for any graph $\Gamma$, that there is a $f\in [g]$ which stabilizes the core spheres and the nonperipheral boundary spheres of pieces of $M_{\Gamma}$. 
    \par    
    We are not assuming $g$ fixes isotopy classes of embeddings of spheres. As $g$ is orientation preserving, either $f$ fixes a given sphere pointwise (up to isotopy), or the diffeomorphism $\bar{f}$ induced from $f$ by cutting along each of these spheres permutes a pair of spheres obtained from cutting along a single sphere in the collection above. This happens if $f$ flips the co-orientation of the sphere in question. If we show these sorts of permutations don't happen, then we will be done, as the restriction of $\bar{f}$ to each component will be a product of sphere twists, again by Lemma 2.5 of \cite{brendle_mapping_2023} or Proposition \ref{HatcherVSES}.
    \par 
    Suppose first that $S$ is a separating sphere so that $\bar{f}$ permutes the two boundary components corresponding to $S$. Let $U$ and $U'$ denote the two components of $M_{\Gamma}\setminus S$. Then $f$ swaps $U$ and $U'$ by the assumptions on $\bar{f}$ and $S$. If $U$ contains an element $S'\in \mathcal{S}(M_{\Gamma})$ not equal to $S$, then $S'\neq g(S')$. But this is impossible by the assumption on $g$, so no such $S'$ exists. This can only happen if $U$ and $U'$ have empty sphere graphs, which implies $M_{\Gamma}$ is a sphere with at most $4$ punctures. Hence $\Gamma$ is sporadic, a contradiction.  
    
    \par 
    Now suppose $S$ is nonseparating so that $\bar{f}$ permutes the two boundary components corresponding to $S$. Hence $\Gamma$ must have rank at least $1$. As $\Gamma$ is non-sporadic by assumption, we can take a diffeomorphic copy $K$ of $M_{1,2}$ embedded in $M_{\Gamma}$ containing $S$ as a nonseparating sphere and so that the boundary spheres of $K$ are essential and separating (but potentially peripheral) in $M_{\Gamma}$. By the above discussion about separating spheres and the assumption on $\partial K$, $f$ restricts (up to isotopy) to a diffeomorphism on $K$ fixing $\partial K$ and stabilizing $S$. But if the diffeomorphism induced on $K\setminus S$ swapped the two boundary components coming from $S$, then it also swaps the two nonperipheral spheres in $K\setminus S \cong M_{0,4}$ that separate the two boundary components coming from $S$. Thus $f$ itself could not fix these two spheres, so we again have a contradiction. The result follows.

    \end{proof}
We now detail why the sporadic cases are excluded. If $\Gamma$ is rank $0$ with at most $3$ ends, then $\mathcal{S}(M_{\Gamma})$ is empty and the result is vacuous. If $\Gamma$ is a rank $0$ graph with $4$ ends, then elements of $\text{Map}(\Gamma)$ which correspond to products of disjoint cycles acting on the ends of $\Gamma$ fix all three elements of $\mathcal{S}(M_{\Gamma})$. In other words, the map $\text{Map}(\Gamma) \to \text{Aut}(\mathcal{S}(M_{\Gamma}))$ is just the exceptional homomorphism $S_4 \to S_3$. If $\Gamma$ is rank $1$ with at most $1$ end, then $\mathcal{S}(M_{\Gamma})$ is a single vertex, and $\text{Map}(\Gamma)\cong \Z/2$ and thus there is trivially not a faithful action. Finally, if $\Gamma$ is a compact rank $2$ graph, then $\text{Map}(\Gamma)\cong \text{Out}(F_2)\cong GL(2,\Z)$ is well known to surject onto $\text{Aut}(\mathcal{S}(M_{\Gamma}))\cong \text{PGL}(2,\Z)$ with an order $2$ kernel generated by the automorphism that flips the signs of a fixed generating of $F_2$. In particular, the action is not faithful.
\par 
We note a conjecture which would easily imply Lemma \ref{KernelStructure}.
\begin{conjecture}\label{conj:LaudenbachExtension}
    Let $i$ and $i'$ be two proper homotopic sphere systems (with possibly infinitely many components) in a $3$-manifold $M$ so that no two distinct spheres in $\text{im}(i)$ are homotopic. Then there is an ambient isotopy of $M$ sending $i$ to $i'$.
\end{conjecture}

This conjecture could potentially be useful in other circumstances. For example, one could obtain a normal form for infinite sphere systems, in the sense given in \cite{hatcher_homological_1995}. It could also be valuable in making sense of a general notion of Outer space for infinite type graphs (see Subsection \ref{subsec:OuterSpace} for some discussion about Outer space for infinite type graphs). The main roadblock in proving it in an inductive manner as in Lemma \ref{KernelStructure} is that it seems difficult to control how isotopies at later steps interact with compact sets containing the spheres that were already isotoped previously (this ended up not being a problem in the proof of Lemma \ref{KernelStructure} due to our special choice of sphere systems). This gives a concern that any sort of concatenation of the isotopies may not be well defined or continuous. 


\section{Isotopy classes of sphere twists}\label{sec:SphereTwists}

In this section, we analyze the structure of the subgroup $\text{Twists}(M_{\Gamma})$ of sphere twists, with the goal of proving Theorem \ref{TwistGroupStructure}. We will use this result to prove the second part of Theorem \ref{SESMainTheorem} in Subsection \ref{subsec:SplittingSES}.
\par 
\subsection{Generating $\text{Twists}(M_{\Gamma})$}
We again assume in this section that $\Gamma$ is in standard form and that $\Gamma$ has no vertices of valence $1$. As discussed before, this implies that the boundary components of every piece are all separating spheres.
\par    
Recall the collection of spheres given in Lemmas \ref{KernelStructure} and \ref{lem:rank1KernelStructure}. This collection consists of the core spheres of the pieces of $M_{\Gamma}$, along with a collection of boundary spheres of pieces. By the two lemmas, every element of $\text{Twists}(M_{\Gamma})$ can be written as a product of sphere twists on these spheres. The next lemma shows that we can remove the sphere twists on boundary components of pieces.

\begin{lemma}\label{TrivialSeperatingTwists}
    Let $S$ denote a boundary component of a piece of $M_{\Gamma}$. Then $[T_S]=\text{id}$. 
\end{lemma}
\begin{proof}
    The primary fact needed is that we can move sphere twists along pairs of pants. Namely, if $P$ is an embedded copy of $S^3$ with three open balls removed in $M_{\Gamma}$, and $S_1, S_2$, and $S_3$ the boundary components of $P$, then $T_{S_1}=T_{S_2}T_{S_3}$ in $\text{Map}(M_{\Gamma})$. This can be seen in Figure 2.1 of \cite{hatcher_stabilization_2010}.
    \par 
    Using this, one can push $T_S$ by an isotopy so that its support lies outside of any arbitrarily large compact subset of $M_{\Gamma}$. To do this, one pushes $T_S$ across the pieces of $M_{\Gamma}$ via sequences of pairs of pants, as pictured in Figure \ref{fig:SphereTwistSwindle}. This implies that $T_S$ is isotopic to the identity.
\end{proof}

\begin{figure}
    \centering
    \includegraphics[scale=.5]{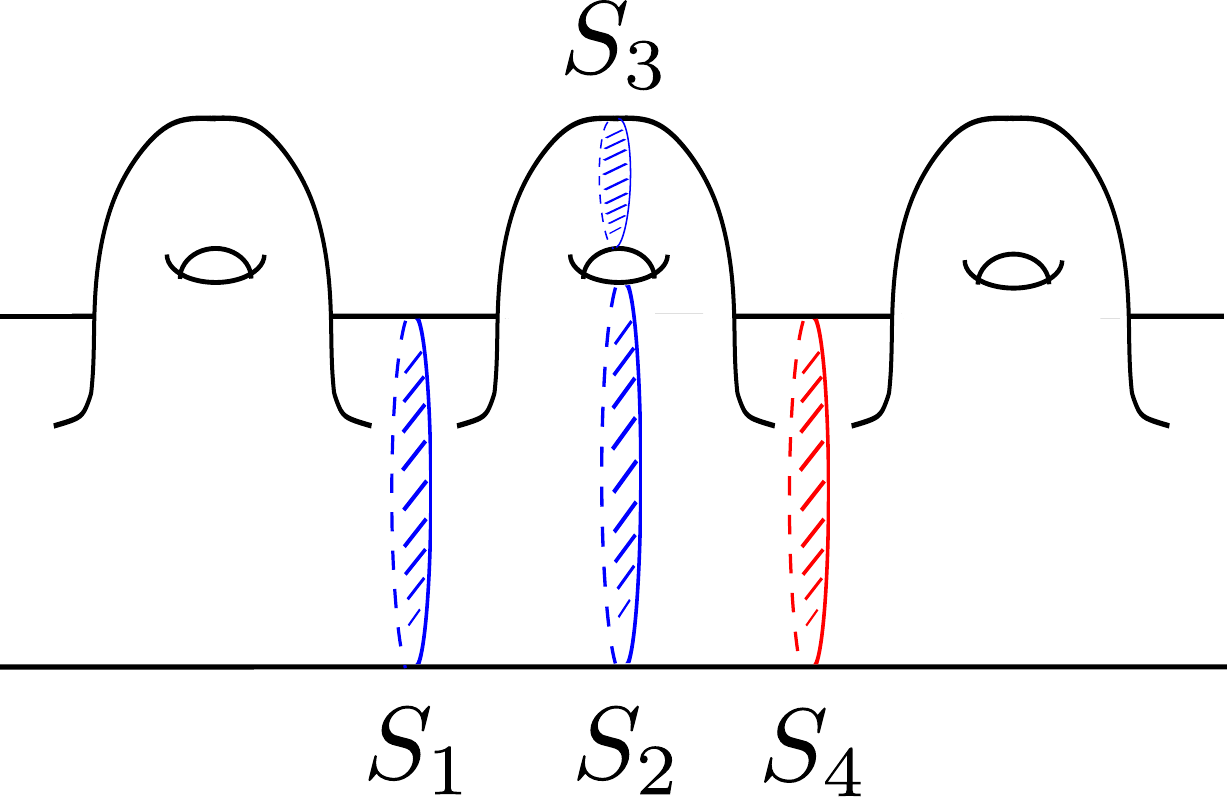}
    \caption{The sphere twist $T_{S_1}$ is equal to the composition of sphere twists $T_{S_2}T_{S_3}$, which is in turn isotopic to $T_{S_4}$. One can continue these moves forever to push the support outside of any compact set.}
    \label{fig:SphereTwistSwindle}
\end{figure}

This implies that any (possibly infinite) product of sphere twists $f$ on the spheres from Lemmas \ref{KernelStructure} or \ref{lem:rank1KernelStructure} is isotopic to a product of sphere twists on the core spheres of the pieces. Indeed, we can write $f=f'f''$ where $f'$ consists of the sphere twists on the core spheres of the pieces, and $f''$ consists of those on the boundary spheres. Then $f''$ is a limit of finite products of sphere twists on the boundary components of pieces, which by Lemma \ref{TrivialSeperatingTwists} are all equal to the identity in $\text{Map}(M_{\Gamma})$. Thus $f''=\text{id}$ in $\text{Map}(M_{\Gamma})$ as well. 
\par 
We note the following simple corollary. 

\begin{corollary}\label{RankZeroPMapTrivial}
    Suppose $\Gamma$ has rank $0$. Then $\text{PMap}(M_{\Gamma})$ is trivial.
\end{corollary}
\begin{proof}
    Theorem \ref{ADMQGraphClassification} shows that $\text{PMap}(\Gamma)$ is trivial. Thus, the short exact sequence in Theorem \ref{SESMainTheorem} restricted to pure mapping class groups implies that $\text{PMap}(M_{\Gamma})\cong \text{Twists}(M_{\Gamma})$. But Corollary \ref{TrivialSeperatingTwists} implies that $\text{Twists}(M_{\Gamma})$ is trivial, as all spheres in $M_{\Gamma}$ are separating. 
\end{proof}

Let $\mathcal{C}$ denote the collection of core spheres of the pieces of $M_{\Gamma}$. To finish the proof of Theorem \ref{TwistGroupStructure}, it suffices to show that no composition of sphere twists on $\mathcal{C}$ is isotopic to the identity. To see this, we utilize tools developed in \cite{brendle_mapping_2023}, which we will describe in the proof.
\begin{lemma}\label{nonseparatingNontrivial}
    Let $f$ denote a nonempty (possibly infinite) product of sphere twists on the elements of $\mathcal{C}$. Then $[f]\neq \text{id}$ in $\text{Map}(M_{\Gamma})$.
\end{lemma}
\begin{proof}
    In \cite{brendle_mapping_2023}, the authors construct (in the case when $\Gamma$ is a compact graph) a crossed homomorphism
    \begin{equation*}
        \mathfrak{T}:\text{Map}(M_{\Gamma}) \to H^1(M_{\Gamma}; \Z/2).
    \end{equation*}
    By a crossed homomorphism, we mean that $\mathfrak{T}$ satisfies the property
    \begin{equation*}
        \mathfrak{T}(f_1f_2)=(f_2)^*\mathfrak{T}(f_1)+\mathfrak{T}(f_2)
    \end{equation*}
    where $(f_2)^*$ denotes the induced map on $H^1$ of $f_2$. The construction of this map does not rely on the compactness of the manifold, it relies only on the fact that there is a trivialization of the tangent bundle (which all oriented $3$-manifolds have, see Theorem 1 in Section VII of \cite{kirby_topology_1989}). Thus we may define this map for any graph $\Gamma$.
    \par 
    Following the proof of Lemma 5.1 in \cite{brendle_mapping_2023}, we may then see the nontriviality of $f$. Indeed, suppose $S\in \mathcal{C}$ is so that $T_S$ appears in $f$. We may take a smoothly embedded curve $\gamma:S^1 \to M_{\Gamma}$ that crosses a neighborhood $S$ along the axis of rotation of $T_S$ exactly one time, and so that it intersects no other sphere in $\mathcal{C}$. The argument of Lemma 5.1 in \cite{brendle_mapping_2023} shows that $\mathfrak{T}(f)([\gamma])$ counts in $\Z/2$ how many times $\gamma$ intersects the spheres that $f$ twists on. It follows, as $\gamma$ crosses $S$ exactly once, that the image of $f$ under $\mathfrak{T}$ is nontrivial, so $f$ is nontrivial ($\mathfrak{T}(id)=(id)^*\mathfrak{T}(id)+\mathfrak{T}(id)=\mathfrak{T}(id)+\mathfrak{T}(id)$, so $\mathfrak{T}(id)=0$).
\end{proof}

\begin{proof}[Proof of Theorem \ref{TwistGroupStructure}]
    Since $\text{Twists}(M_\Gamma) \subset \text{ker}(\Psi)$, Lemmas 4.3 and 4.4 in particular imply that any element of $\text{Twists}(M_{\Gamma})$ can be written as a product of sphere twists on the core spheres, $\mathcal C$, and boundary components of pieces of $M_\Gamma$. Lemma \ref{TrivialSeperatingTwists} shows that the sphere twists on boundary components of pieces are trivial, and Lemma \ref{nonseparatingNontrivial} implies that any nontrivial composition of sphere twists on the core spheres of $M_{\Gamma}$ is not isotopic to the identity. 
    \par 
    Fixing a bijection $\alpha: \mathcal{C}\to\{1, \ldots, \text{rk}(\Gamma)\}$ (allowing for $\text{rk}(\Gamma)=\infty$, in which it is not included in the set), we obtain an induced isomorphism $\tau: \text{Twists}(M_{\Gamma})\to \Pi_{i=1}^{\text{rk}(\Gamma)}\Z/2$. The map $\tau$ is obtained by sending a composition of sphere twists on $\mathcal{C}'\subset \mathcal{C}$ to the tuple $(a_i)_{i=1}^{\text{rk}(\Gamma)}$ with $a_i=1$ if $\alpha^{-1}(i)\in \mathcal{C}'$, and $a_i=0$ otherwise.
    \par 
     We wish to show that $\tau$ is also a homeomorphism. Fix a finite union of pieces $K$. Let $U_K\leq \Pi_{i=1}^{\text{rk}(\Gamma)}\Z/2$ be the subgroup of tuples so that the $i$'th component is $0$ if $\alpha^{-1}(i)$ lies in $K$. The collection of all $U_K$'s is a basis at the identity, so to show that $\tau$ is a homeomorphism it suffices to check that $\tau(\mathcal{V}_K\cap \text{Twists}(M_{\Gamma}))=U_K$, i.e. $\tau$ sends a basis to a basis. To see this, note that an element $f\in \text{Twists}(M_{\Gamma})$ lies in $\mathcal{V}_K$ precisely when no sphere twist $T_S$ that appears in $f$ is so that $S\subset K$. This is because, for a core sphere $S'\subset K$, $T_{S'}$ cannot be isotoped so that it fixes $K$, as $\mathfrak{T}([T_{S'}])$ would then evaluate to $0$ on any $\gamma \in H_1(M_{\Gamma}; \Z/2)$ which lies in $K$. But this is clearly contradicts the proof of Lemma \ref{nonseparatingNontrivial} as one can can choose the $\gamma$ in that proof to lie in $K$. Hence $\tau$ sends elements of $\mathcal{V}_K\cap \text{Twists}(M_{\Gamma})$ to $U_K$. Further, it is easy to see that $\tau(\mathcal{V}_K\cap \text{Twists}(M_{\Gamma}))$ contains $U_K$ by the definition of $\tau$, finishing the proof.
\end{proof} 
We make the following conjecture, which in the case when $\Gamma$ is compact is due to Laudenbach \cite{laudenbach_sur_1973}.
\begin{conjecture}
    Suppose $f\in \text{Diff}^+(M_{\Gamma})$ is homotopic to the identity. Then $f$ is also isotopic to the identity.
\end{conjecture}

The proof of Theorem \ref{SESMainTheorem}(1) shows that if $f\in \text{Diff}^+(M_{\Gamma})$ is homotopic to the identity, then it is isotopic to a product of sphere twists, and Lemma \ref{TrivialSeperatingTwists} specifies which sphere twists are necessary. It thus suffices to show that if such a product of sphere twists is homotopic to the identity, then it is isotopic to the identity as well.

\subsection{Splitting the short exact sequence}\label{subsec:SplittingSES}
To obtain the second part of Theorem \ref{SESMainTheorem}, we use the following lemma, a proof of which can be found in Lemma 3.1 of \cite{brendle_mapping_2023}. 
\begin{lemma}\label{AbelianKernelSplitting}
    Let $G$ be a group and $A<G$ an abelian normal subgroup, so that $G$ acts on $A$ on the right via the formula
    \begin{equation*}
        a^g = g^{-1}ag \quad (a\in A, g\in G)
    \end{equation*}
    Leting $Q=G/A$, the short exact sequence
    \begin{equation*}
        1 \to A \to G \to Q \to 1
    \end{equation*}
    splits if and only if there exists a crossed homomorphism $\lambda :G \to A$ that restricts to the identity on $A$. Moreover, if such a $\lambda$ exists, then we can choose a splitting $Q\to G$ whose image is $\text{ker}(\lambda)$, so $G= A\rtimes \text{ker}(\lambda)$.
\end{lemma}

As we are interested in splitting the short exact sequence in Theorem \ref{SESMainTheorem} as topological groups, we need the following fact. This result likely exists somewhere in the literature, but we give a proof here for completeness. This proof is adapted from (and is nearly identical to) a proof provided in \cite{francis_short_nodate}, which proves this result for direct product splittings.
\begin{lemma}\label{ContinuousSplitting}
    Let $G$ be a topological group, and $N$ a normal subgroup. Let $\pi:G\to G/N$ denote the canonical projection map. Suppose there exists a continuous crossed homomorphism $r:G\to N$ with $r(x)=x$ for all $x\in N$. Then the continuous bijective homomorphism $(r, \pi):G \to N\rtimes (G/N)$ has a continuous inverse.
\end{lemma}
Note that it is standard that $\pi$ is an open map, and $G/N$ with the quotient topology is a topological group. 
\begin{proof}
    Let $s=(\pi|_{\text{ker}(r)})^{-1}: G/N\to \text{ker}(r)$. Then the inverse map of $(r, \pi)$ is the map $(x, y)\mapsto xs(y)$. To obtain continuity, it suffices to show that $s$ is continuous, or equivalently that $\pi|_{\text{ker}(r)}$ is an open map. 
    \par 
    Fix an open set in $\text{ker}(r)$, which we can write as $U\cap \text{ker}(r)$ for some $U\subset G$ open. It may not be the case that $\pi(U\cap \text{ker(r)})=\pi(U)$ (if it were, we would be done with this open set as $\pi$ is an open map). Instead, we define another open set $U'$ given by
    \begin{equation*}
        U'=\{x \in G \ | \ xr(x)^{-1} \in U\}.
    \end{equation*}
    As $U$ is open and $r$ is continuous, it follows that $U'$ is open. As $xr(x)^{-1}=x$ when $x\in \text{ker}(r)$, it follows that $U'\cap \text{ker}(r)=U\cap \text{ker}(r)$. 
    \par
     Fix $x\in U'$. Then $xr(x)^{-1}\in U \cap \text{ker}(r)=U' \cap \text{ker}(r)$, as by the definition of $U'$, $xr(x)^{-1}\in U$, and $$r(xr(x)^{-1})=r(x)r(r(x)^{-1})=r(x)r(x)^{-1}=\text{id}$$ where the second equality is using the fact that $r$ fixes $N$. Thus, $xr(x)^{-1}\in \text{ker}(r)$. Also, $\pi(x)=\pi(xr(x)^{-1})$ since $r(x)\in \text{ker}(\pi)$. This implies that $\pi(U')=\pi(U'\cap \text{ker}(r))$. Thus $\pi(U\cap \text{ker}(r))=\pi(U'\cap \text{ker}(r))=\pi(U')$ is open as $\pi$ is an open map.
\end{proof}

To utilize this, we prove the following result about the crossed homomorphism $\mathfrak{T}$ discussed in Lemma \ref{nonseparatingNontrivial}. Recall that $H^1(M_{\Gamma}; \Z/2)\cong H_2^{BM}(M_{\Gamma}; \Z/2)$ has either the discrete topology if $\text{rk}(\Gamma)$ is finite, and the topology of a Cantor set if it is infinite. The following lemma is a generalization of Corollary 5.2 in \cite{brendle_mapping_2023}.
\begin{lemma}\label{CrossedHomContinuous}
    The crossed homomorphism $\mathfrak{T}:\mathrm{Map}(M_{\Gamma})\to H^1(M_{\Gamma}; \Z/2)$ is continuous, and restricts to a topological group isomorphism between $\mathrm{Twists}(M_{\Gamma})$ and $H^1(M_{\Gamma}; \Z/2)$.
\end{lemma}
\begin{proof}
The topology on $H^1(M_{\Gamma}; \Z/2)\cong H_2^{BM}(M_{\Gamma}, \Z/2)$ is generated by subgroups about the identity corresponding to fixing the coefficient on a finite set of core spheres to be $0$, and letting the others vary, by Lemma \ref{BorelMooreGenerators}. Fixing such a subgroup $U$, the corresponding finite set of spheres where the coefficient is fixed to be $0$ is contained in some compact submanifold $K$, and it is clear that $\mathcal{V}_K\subset \mathfrak{T}^{-1}(U)$. On the other hand, given $f \in \mathfrak{T}^{-1}(U)$ and $g\in \mathcal{V}_K$, we have
$$\mathfrak{T}(fg)=g^*(\mathfrak{T}(f))+\mathfrak{T}(g).$$
By above, $\mathfrak{T}(g)\in U$. Since $g$ preserves $K$ and and $\mathfrak{T}(f)\in U$, it is clear that $g^*(\mathfrak{T}(f))\in U$. This is because $g$ stabilizes the homology classes of both the spheres in $K$ and in $M_{\Gamma}\setminus K$. Hence $fg \in \mathfrak{T}^{-1}(U)$ and thus $f\mathcal{V}_K\subset \mathfrak{T}^{-1}(U)$. Since $f\in \mathfrak{T}^{-1}(U)$ is arbitrary, it follows that $\mathfrak{T}^{-1}(U)$ is open, as it is the union of open sets. Thus $\mathfrak{T}$ is continuous.
\par 
    The statement and argument in Lemma \ref{nonseparatingNontrivial} along with Lemma \ref{BorelMooreGenerators} implies that $\mathfrak{T}$ restricts to a bijection from $\text{Twists}(M_{\Gamma})$ to $H^1(M_{\Gamma}; \Z/2)$. Further, as $\text{Twists}(M_{\Gamma})$ acts trivially on $H^1(M_{\Gamma}; \Z/2)$, this restriction is a homomorphism. Both groups are compact and Hausdorff, and hence the inverse map is also continuous. (We remark also that the restriction of $\mathfrak{T}$ to $\text{Twists}(M_{\Gamma})$ agrees with the map $\tau$ defined in the proof of Theorem \ref{TwistGroupStructure}, and there it is shown that $\tau$ is a homeomorphism as well).
    
\end{proof}

\begin{proof}[Proof of Theorem \ref{SESMainTheorem}(2)]
We obtain a continuous crossed homomorphism $\text{Map}(M_{\Gamma})\to \text{Twists}(M_{\Gamma})$ which restricts to the identity on $\text{Twists}(M_{\Gamma})$ via composing the map $\mathfrak{T}$ with the continuous isomorphism $\mathfrak{T}^{-1}:H^1(M_{\Gamma}; \Z/2)\to \text{Twists}(M_{\Gamma})$ from Lemma \ref{CrossedHomContinuous}. Then Lemma \ref{AbelianKernelSplitting} produces the desired splitting, and by Lemma \ref{ContinuousSplitting} as the crossed homomorphism is continuous it follows that $\text{Map}(M_{\Gamma})$ and $\text{Twists}(M_{\Gamma})\rtimes \text{Map}(\Gamma)$ are isomorphic as topological groups.
\end{proof}
\par
In \cite{brendle_mapping_2023}, it is further shown that when $\Gamma$ is compact and of rank $n$, the image of $\text{Out}(F_n)$ in $\text{Map}(M_{\Gamma})$ is equal to the stabilizer of some fixed homotopy class of a trivialization of the tangent bundle of $M_{\Gamma}$. Doing this in greater generality seems difficult. Their proof produces a homomorphism from $\text{ker}(\mathfrak{T})$ to $\Z$ so that if one can show that the image of this homomorphism is trivial, then the elements of $\text{ker}(\mathfrak{T})$ fix a particular homotopy class of a trivialization of the tangent bundle. The result follows in their case, as $\text{ker}(\mathfrak{T})\cong \text{Out}(F_n)$ has finite abelianization (see Lemma 6.1 of \cite{brendle_mapping_2023} for more details). In the general case, it is not even clear how to the construct the homomorphism from $\text{ker}(\mathfrak{T})$ to $\Z$. Further, the abelianization of $\text{Map}(\Gamma)$ is unknown. 
\par
It would also be interesting if one could explicitly produce the section of $\text{Map}(\Gamma)$ to $\text{Map}(M_{\Gamma})$. In \cite{arrieta_explicit_2023}, Arrieta produces an explicit section in the case when $\Gamma$ is compact, utilizing the results of \cite{brendle_mapping_2023}. To directly generalize their argument, one would need to be able to characterize the image of $\text{Map}(\Gamma)$ in $\text{Map}(M_{\Gamma})$ as discussed above (via stabilizing the homotopy class of a trivialization of the tangent bundle), as well as be able to choose a generating set of $\text{Map}(\Gamma)$ which has the required properties.

\section{Contractibility of the sphere complex and its subcomplexes}\label{Connectivity}

\par 
In this section, we consider the topological properties of $\mathcal{S}(M_{\Gamma})$ and some of its subcomplexes. We note the following fact.
\begin{proposition}\label{SphereComplexContractible}
    If the rank of $\Gamma$ is not $0$, the complex $\mathcal{S}(M_{\Gamma})$ is contractible.
\end{proposition}
\begin{proof}
    By Theorem 2.1 and Lemma 2.2 in \cite{hatcher_homological_1995}, the sphere complex of a finite type graph of positive rank is contractible. But $\mathcal{S}(M_{\Gamma})$ is the direct limit of the sphere complexes of the connected unions of pieces of $M_{\Gamma}$, each of which is contractible. Contractibility of $\mathcal{S}(M_{\Gamma})$ follows from Whitehead's theorem, as for any $m\geq 0$, any given element of $\pi_m(\mathcal{S}(M_{\Gamma}))$ is contained in the sphere complex of some connected union of pieces because it only intersects finitely many vertices and interiors of simplices of $\mathcal{S}(M_{\Gamma})$ by compactness.
\end{proof}

In the rank $0$ case it is not clear precisely when $\mathcal{S}(M_{\Gamma})$ is contractible. In the finite type case it is not contractible. For example, $\mathcal{S}(M_{0,5})$ is a graph with positive rank. On the other hand, we have the following.

\begin{proposition}\label{prop:rank0ContractibleExamples}
    Suppose $\Gamma$ is a rank $0$ graph with end space homeomorphic to the union of a Cantor set and a (possibly empty) finite collection of isolated points. Then $\mathcal{S}(M_{\Gamma})$ is contractible.
\end{proposition}
\begin{proof}
    Suppose first that the finite set is empty or contains one point. The argument of Theorem 2.1 of \cite{hatcher_homological_1995} goes through verbatim. This is because the flow of $\mathcal{S}(M_{\Gamma})$ that is constructed only relies on the guarantee that an innermost disc surgery of a sphere along a component of a maximal system does not produce two spheres which both bound a once punctured ball. But this does not happen in either of the two cases of $\Gamma$ given, as for this to occur the sphere being surgered must contain two isolated ends in one of its complementary components. 
    \par 
    The proof of Lemma 2.2 of \cite{hatcher_homological_1995} also goes through to show that if the end space is a Cantor set along with any finite number of isolated points. Indeed, its proof only relies on inductively reducing down to the case of one less isolated end.
\end{proof}

Things become much less clear once $E(\Gamma)$ contains infinitely many isolated points. It would be interesting if there were other examples of contractible $\mathcal{S}(M_{\Gamma})$.
\par
Hatcher defines a variety of other complexes of spheres in \cite{hatcher_homological_1995}. The first, which we denote by $Y(M_{\Gamma})$, is a subcomplex of $\mathcal{S}(M_{\Gamma})$ with its simplices the sphere systems $S$ so that $M_{\Gamma}\setminus S$ is connected. The subcomplex $Z(M_{\Gamma})$ consists of the sphere systems $S$ so that for one component $(M_{\Gamma})_S$ of $M_{\Gamma}\setminus S$, the map induced by inclusion $\pi_1((M_{\Gamma})_S)\to \pi_1(M_{\Gamma})$ is an isomorphism. The last complex, denoted $Y^{\pm}(M_{\Gamma})$, is a complex whose simplices are sphere systems $S$ with $M_{\Gamma}\setminus S$ connected, together with a choice of orientation for each component of $S$. This is not a subcomplex of $\mathcal{S}(M_{\Gamma})$, but it is closely related to $Y(M_{\Gamma})$.
\par 
We note the following results of Hatcher about $Y$, $Y^{\pm}$, and $Z$ in the finite type case \cite{hatcher_homological_1995}. Recall that a topological space $X$ is $m$-\textit{connected} if $\pi_i(X, x_0)=0$ for $i=0, \ldots m$.
\begin{lemma}[{\cite[Proposition 3.1]{hatcher_homological_1995}}]\label{lem:HatcherOtherComplexes}
    Fix $n, s\geq 0$.
    \begin{enumerate}
        \item Both $Y(M_{n,s})$ and $Y^{\pm}(M_{n,s})$ are $n-2$ connected.
        \item $Z(M_{n,s})$ is $s-3$ connected.
    \end{enumerate}
\end{lemma}
\par 
We now prove the following application of Lemma \ref{lem:HatcherOtherComplexes}. 

\begin{proposition}\label{HatcherSubcomplexConnectivity}
        If $\Gamma$ has finite rank $n$, then $Y(M_{\Gamma})$ and $Y^{\pm}(M_{\Gamma})$ are $n-2$ connected. If $\Gamma$ has infinite rank, then $Y(M_{\Gamma})$ and $Y^{\pm}(M_{\Gamma})$ are contractible. If $\Gamma$ has finite rank and an infinite end space, then $Z(M_{\Gamma})$ is contractible.
\end{proposition}
\begin{proof}
We note first that if $K$ is a connected union of pieces of $M_{\Gamma}$, then there is a natural inclusion of $Y(K), Y^{\pm}(K)$ into $Y(M_{\Gamma}), Y^{\pm}(M_{\Gamma})$, respectively.
\par
 Given any fixed representative of an element of $\pi_m(Y(M_{\Gamma}))$ or $\pi_m(Y^{\pm}(M_{\Gamma}))$, there is some connected union of pieces $K$ so that  $Y(K)$ or $Y^{\pm}(K)$ contains this representative. In the case that the rank of $\Gamma$ is $n<\infty$, by taking the rank of $K$ to be $n$, it follows from Lemma \ref{lem:HatcherOtherComplexes}(a) that the given element of $\pi_m(Y(M_{\Gamma}))$ or $\pi_m(Y^{\pm}(M_{\Gamma}))$ is trivial for $m\leq n-2$ as $Y(K)$ is $n-2$-connected by Lemma \ref{lem:HatcherOtherComplexes}(a). If the rank of $\Gamma$ is infinite, we can instead take the rank of $K$ to be arbitrarily large (at least $m+2$) to see that the given element of $\pi_m(Y(M_{\Gamma}))$ or $\pi_m(Y^{\pm}(M_{\Gamma}))$ is trivial.
        \par 
        For $Z(M_{\Gamma})$, the same argument works. Once $K$ has rank equal to that of $\Gamma$, then $Z(K)$ naturally embeds as a subcomplex of $Z(M_{\Gamma})$. Choosing $K$ so that $Z(K)$ contains a fixed representative of an element of $\pi_m(Z(M_{\Gamma}))$ so that $K$ has at least $m+3$ boundary components, it follows that this element is trivial by Lemma \ref{lem:HatcherOtherComplexes}(b).
\end{proof}

One would like to take the previous proposition a step further, by generalizing Proposition 3.1 of \cite{hatcher_homological_1995} in the case of $Z(M_{\Gamma})$ by allowing infinite rank and/or finite end space. However, if $\Gamma$ has infinite rank and $K$ is any connected union of pieces, $Z(K)$ is not a subcomplex of the $Z(M_{\Gamma})$. This is because there are spheres in the $Z(K)$ which cut off a boundary component where on the other side there is rank in $M_{\Gamma}$. 

\section{Topology and group actions}\label{Applications}
In this section we give a variety of applications of the previous results. Much of the following is motivated by results from the study of infinite type surfaces and big mapping class groups.
\subsection{Topology}\label{subsec:Topology}
We define the \textit{permutation topology} on $\text{Map}(\Gamma)$ to be generated by the following subbasis about the identity, where here $S$ is a fixed vertex of $\mathcal{S}(M_{\Gamma})$:
\begin{equation*}
    U_{S}=\{f\in \text{Map}(\Gamma)\ |\ f(S)=S\}.
\end{equation*}
It is a consequence of the Alexander method for infinite type surfaces that the corresponding topology for infinite type mapping class groups given by the action on the curve complex of the surface is the same as the quotient topology it inherits from the homeomorphism group \cite{vlamis_notes_nodate}. See also \cite{hernandez_alexander_2019} and \cite{shapiro_alexander_2022}.
\par 
We prove an analogous result for infinite type graphs.
\begin{proposition}\label{PermutationTopology}
    The permutation topology on $\mathrm{Map}(\Gamma)$ agrees with the quotient topology.
\end{proposition}
\begin{proof}
    Recall from the work of Algom-Kfir--Bestvina that the quotient topology has a basis at the identity given by $\{\mathcal{V}_K\}$ where $K$ varies over the compact subgraphs of $\Gamma$ (these subgroups were defined right before Proposition \ref{RichardsHandlebody}). We may thus make use of these sets to show the equivalence of the two topologies. Note that both $\mathcal{V}_K$ and $U_S$ are subgroups, so it suffices to show any $\mathcal{V}_K$ contains some intersection of finitely many $U_S$'s for varying spheres $S$, and vice versa.
    \par 
    Fix a sphere $S\in \mathcal{S}(M_{\Gamma})$. Then $S$ is contained in a finite union of pieces $P$. It follows that $\mathcal{V}_{r(P)}\subset U_S$. Indeed, recall from the proof of Proposition \ref{MCGHom}(d) that $\Psi(\mathcal{V}_P)=\mathcal{V}_{r(P)}$, so the action of every element of $\mathcal{V}_{r(P)}$ on $\mathcal{S}(M_{\Gamma})$ agrees with the action of some element of $\mathcal{V}_P$, which clearly acts trivially on $S$. 
    \par 
     If $K$ is any compact subgraph of $\Gamma$, then take a graph $L$ containing $K$ so that $r^{-1}(L)$ is a union of pieces. Let $\mathcal{S}$ denote the finite collection of core spheres and boundary components of $L$. Then $\cap_{S\in \mathcal{S}} U_S \subset \mathcal{V}_L \subset \mathcal{V}_K$. Indeed, any lift to $\text{Map}(M_{\Gamma})$ of an element $f\in \cap_{S\in \mathcal{S}} U_S$ stabilizes $L$ up to isotopy, as the core spheres and boundary spheres of $L$ are fixed, There is a well defined restriction of $f$ to $L$. But up to isotopy, as the core spheres of $L$ are also fixed, the restriction of any element in $\Psi^{-1}(f)$ to $L$ must be a product of sphere twists by Proposition \ref{HatcherVSES}, and thus $f\in \mathcal{V}_L$.
\end{proof}

We will now show that the topology on $\text{Map}(M_{\Gamma})$ generated by the $\mathcal{V}_K$ sets agrees with the quotient topology arising from $\text{Homeo}^+(M_{\Gamma})$ (which recall we can use instead of $\text{Diff}^+(M_{\Gamma})$), and also that this topology is Polish. The argument here is inspired by the results in the appendix of \cite{vlamis_notes_nodate}. 
\begin{proposition}\label{TopologyEquivalence}
    The topology generated by the basis about the identity $\{\mathcal{V}_K\}$ where $K$ varies over the compact submanifolds of $M_{\Gamma}$ agrees with the quotient topology on $\mathrm{Map}(M_{\Gamma})$ arising from the compact open topology on $\mathrm{Homeo}^+(M_{\Gamma})$.
\end{proposition}
\begin{proof}
    By work of Arens, the compact open topology on $\text{Homeo}^+(M_{\Gamma})$ can be metrized as follows \cite{arens_topology_1946}. Fix a complete metric $d$ on $M_{\Gamma}$, and consider a compact exhaustion $\{K_n\}_{n\in \N}$ of $M_{\Gamma}$ by unions of pieces. For $f, g\in \text{Homeo}^+(M_{\Gamma})$, define 
    \begin{equation*}
        \delta_n(f,g) = \min\{\max\{d(f(x), g(x)) \ | \ x\in K_n\}, 2^{-n}\}
    \end{equation*}
    and $\rho(f,g)=\sum_{n}\delta_n(f,g)$. Then $\rho$ induces the compact open topology on $\text{Homeo}^+(M_{\Gamma})$. 
    \par
    Let $\pi$ denote the quotient map from $\text{Homeo}^+(M_{\Gamma})$ to $\text{Map}(M_{\Gamma})$. Fix a connected union of pieces $K\subset M_{\Gamma}$. If $\rho(f, \text{id})$ is sufficiently small then $\partial K$ and $f(\partial K)$, are homotopic (as embeddings). Thus by Theorem \ref{LaudenbachSphereIsotopy}, we may modify $f$ by a small isotopy so that $f$ fixes $\partial K$ and also stabilizes $K$. But the group $\text{Map}(K, \partial K)$ is discrete. Indeed, this follows from work of Chernavskii or of Edwards--Kirby, which both say that homeomorphism groups of compact manifolds are locally contractible, and thus locally path connected \cite{chernavskii_local_1969}\cite{edwards_deformations_1971}. As $f$ restricted to $K$ is close to the identity on $K$ (assuming $\rho(f, id)$ is sufficiently small), it follows that $\pi(f)\in \mathcal{V}_K$.
    \par 
     Note that $\pi$ is an open map as quotient homomorphisms of topological groups are open. It follows that for any $[f]\in \mathcal{V}_K$, the $\pi$ image of a small $\rho$-ball about a representative $f$ of $[f]$ fixing $K$ is an open set (in the quotient topology) containing $[f]$ which is also a subset of $\mathcal{V}_K$. Indeed, if $\rho(f, f')$ is small, then since $f$ is assumed to fix $K$ pointwise, the work in the previous paragraph shows that $f'$ can be isotoped to fix $K$ pointwise as well. Thus $\mathcal{V}_K$ is a union of open sets in the quotient topology, so it is itself open.
    \par 
    On the other hand, for any open set $U$ in the quotient topology and any $[g]\in U$ there is some sufficiently large union of pieces $K$ so that $[g]\mathcal{V}_K\subset U$. This is because homeomorphisms in $\pi^{-1}([g]\mathcal{V}_K)$ have isotopy representatives whose restrictions to $K$ are the same as that of $g$, a fixed representative of $[g]$. In particular, by the definition of the metric $\rho$, every element of $\pi^{-1}([g]\mathcal{V}_K)$ has a representative $g'$ so that $\rho(g, g')<\epsilon$, where $\epsilon>0$ is small enough so that the $\pi$ image of the $\epsilon$ $\rho$-ball about $g$ lies in $U$ (allowing $K$ to be as large as desired is what allows us to guarantee this for any $\epsilon$). Thus $U$ is the union of open sets in the topology generated by the collection of $\mathcal{V}_K$'s, and is thus open in that topology as well. It follows that the topologies are equivalent.
\end{proof}

\begin{remark}
It seems likely that one can obtain a more "down to earth" proof of the fact that the mapping class group of a compact doubled handlebody with a finite collection of open balls removed equipped with the quotient topology is discrete, but the above proof suffices.     
\end{remark}

\par 

As a consequence of the previous proposition, we have the following fact.

\begin{corollary}\label{cor:doubledHandlebodyPolish}
    The topological group $\mathrm{Map}(M_{\Gamma})$ with either of the two topologies is Polish. In particular, $\mathrm{PMap}(M_{\Gamma})$ is Polish as well.
\end{corollary}
\begin{proof}
    This follows from Theorem \ref{SESMainTheorem}, as $\text{Map}(M_{\Gamma})$ is topologically the product of two Polish spaces, and is thus Polish. As $\text{PMap}(M_{\Gamma})$ is a closed subgroup of $\text{Map}(M_{\Gamma})$, it is Polish because closed subgroups of Polish groups are Polish. 
\end{proof}

It follows from this corollary that $\text{Homeo}_0(M_{\Gamma})$ is closed, as a quotient group of a topological group is Hausdorff if and only if the kernel is closed.
\par 
We also obtain the following result, which is analogous to Proposition 4.11 in \cite{algom-kfir_groups_2021}.
\begin{corollary}
    Suppose $\Gamma$ is an infinite type graph. Then $\mathrm{Map}(M_{\Gamma})$ is homeomorphic to $\Z^{\infty}$, i.e. to the set of irrationals.
\end{corollary}
\begin{proof}
    The space $\text{Map}(M_{\Gamma})$ is topologically the product of a finite set or a Cantor set, and a space homeomorphic to $\Z^{\infty}$. This follows from Theorem \ref{SESMainTheorem} and Proposition 4.11 in \cite{algom-kfir_groups_2021}, which implies that $\text{Map}(\Gamma)$ is homeomorphic to $\Z^{\infty}$. In particular, $\text{Map}(M_{\Gamma})$ is separable, completely metrizable, zero dimensional (i.e. has a basis of clopen sets), and every compact subset has empty interior. It follows from a theorem of Hausdorff (see for example \cite{eberhart_remarks_1977}
    )
    that $\text{Map}(M_{\Gamma})$ is homeomorphic to $\Z^{\infty}$.
\end{proof}

\subsection{Nonseparating spheres and finite rank graphs}\label{nonseparating Sphere Finite Rank}
Recall from Subsection \ref{SphereGraph} the nonseparating sphere complex $\mathcal{S}_{ns}(M_{\Gamma})$, which is the full subcomplex of $\mathcal{S}(M_{\Gamma})$ with vertices the nonseparating spheres of $M_{\Gamma}$. We remark that the vertices of this complex coincide with those of the subcomplex $Y(M_{\Gamma})$ defined in Section \ref{Connectivity}, but $\mathcal{S}_{ns}(M_{\Gamma})$ is strictly larger than $Y(M_{\Gamma})$ in almost all cases.
\par 
Our goal is to prove Theorem \ref{PositiveTranslationLength}. To do this, we utilize projection maps defined on spheres, using ideas analogous to those described by Aramayona--Valdez in Theorem 1.4 of \cite{aramayona_geometry_2018} and Section 8 of the work of Durham--Fanoni--Vlamis in \cite{durham_graphs_2018}. Projection maps for spheres were initially studied by Sabalka--Savchuk in \cite{sabalka_submanifold_2012}, but their definition is too restrictive for what we need here. Namely, in their definition they require that the components of the projection of a sphere $S$ to a submanifold $K$ with sphere boundary components can only arise from surgeries of components of $S\cap K$ which are discs. This is to avoid components whose projections are not essential or non-peripheral. Consider, for example, a cylinder component of $S\cap K$ intersecting a single component of $\partial K$. There are three choices of surgery one could perform with this component and the boundary of $K$ it intersects arising from capping off by disc in the boundary. One of them would produce a sphere bounding a ball, and two of them would produce a sphere parallel to the boundary. Various other types of projections have been defined in the study of $\text{Out}(F_n)$, see \cite{hamenstadt_spheres_2011}, \cite{bestvina_subfactor_2014} and \cite{hamenstadt_submanifold_2024} for some examples.
\par 
 To define the notion of projection we will use, we need the following lemma which will imply the independence of the choice of representative of the sphere we are projecting. Recall the notions of normal form and equivalence of sphere systems defined in Subsection \ref{subsec:graphsandtheirdoubledhandlebodies}.
 \begin{lemma}\label{lem:ProjRepIndependence}
     Suppose $S$ and $S'$ are two isotopic spheres which are in normal form with respect to the boundary of a submanifold $K$ with sphere boundary components which do not bound a ball in $M_{\Gamma}$. Let $T$ be an embedded essential non-peripheral sphere which arises from capping off the boundary components of a connected component of $S\cap K$ with discs in $\partial K$. Then there is an embedded essential non-peripheral sphere $T'$ which is given by the same procedure with a component of $S'\cap K$ so that $T$ and $T'$ are isotopic inside of $K$.
 \end{lemma}
 \begin{proof}
     We can assume that $S$, and therefore $S'$ nontrivially intersect $\partial K$, and thus such a $T$ as described in the statement of the lemma actually exists. Let us clarify that if the discs in $\partial K$ used to define $T$ are nested, to obtain $T$ as an embedded sphere we need to move some of the discs slightly into $K$. We can assume that this happens on some fixed small neighborhood of $\partial K$.
     \par 
     By Lemma \ref{lem:NormalAndEquivalent}(c), $S$ and $S'$ are equivalent. Fix a homotopy $h_t$ from $S$ to $S'$ as in Definition \ref{def:equivSphere}. Let $R$ denote the component of $S\cap K$ so that $T$ is the result of capping off boundary components of $R$ with discs in $\partial K$. By the definition of the homotopy $h_t$, the set $R'=h_1(R)$ is a component of $S'\cap K$. Further, the discs in $\partial K$ that cap off $R$ to obtain $T$ correspond to discs whose circle boundary components are the components of $R'\cap \partial K$, as $h_t$ restricts to an isotopy on $\partial K$. We can then define a sphere $T'$ which results from capping $R'$ with these discs (again, to obtain an embedded sphere one may need to slightly move things in the fixed small neighborhood of $\partial K$, as discussed above).
     \par 
     The restriction of the homotopy $h_t$ from $R$ to $R'$ extends to a homotopy from $T$ to $T'$ so that the image of the homotopy lies entirely in $K$, by the construction of $T'$. There is some minor care to be had here, again due to the fact that we may need to move the discs that cap off the boundary components of $R$ and $R'$ off of $\partial K$. But it is easy to see that one can modify $h_t$ on the fixed neighborhood of $\partial K$ so that the above statement is true. By Theorem \ref{LaudenbachSphereIsotopy}, it follows that $T$ and $T'$ are isotopic in $K$, since the homotopy between them is supported in $K$.
 \end{proof}
 
\par 
\begin{definition}\label{def:SphereProjection}
    Take a graph $\Gamma$, and fix an submanifold $K$ of $M_{\Gamma}$ with essential non-peripheral sphere boundary components. Let $S$ denote an essential non-peripheral embedded sphere in $M_{\Gamma}$, and suppose $S$ has been isotoped so that the intersection of $S$ and $\partial K$ is transverse and $S\cap \partial K$ has a minimal number of components among the isotopy class of $S$ (i.e. $S$ is in normal form with respect to $\partial K$, by Lemma \ref{lem:NormalAndEquivalent}). We define the \textit{projection of} $S$ \textit{to} $K$ as follows. We denote it by $\pi_K(S)$, where
    \begin{enumerate}
        \item if $S\subset K$, then $\pi_K(S)=\{S\}$,
        \item if $S\cap K=\varnothing$, then $\pi_K(S)=\varnothing$,
        \item if $S\cap K$ consists of a collection of proper submanifolds of $S$, then $\pi_K(S)$ is the set of essential non-peripheral embedded spheres which arise from capping off the boundary components of each of these submanifolds by discs in $\partial K$.
    \end{enumerate}
\end{definition}
By Lemma \ref{lem:ProjRepIndependence}, the definition of $\pi_K(S)$ does not depend on which representative of the isotopy class which is in normal form with respect to $\partial K$ that one chooses. 
\par 
As discussed above, $\pi_K(S)$ may be empty even when $S$ intersects $K$ minimally and nontrivially. Luckily, for projections of nonseparating spheres, this is not the case for all $K$ sufficiently large, in the following sense. 

\begin{definition}\label{def:FullSubmfld}
    A connected submanifold $K\subset M_{\Gamma}$ is \textit{full} if its boundary components are spheres in $M_{\Gamma}$ which are not homotopically trivial and so that every component of $M_{\Gamma}\setminus K$ is simply connected.
\end{definition}

For this class of submanifolds, we can show that projections behave in the desired way, i.e. all nonseparating spheres have a nonempty projection to any full submanifold. Before we show this, we give an auxiliary lemma.

\begin{lemma}\label{lem:arcReplacement}
    Let $U$ be a manifold which is a simply connected doubled handlebody with a single ball removed. Suppose $\{D_i\}$ is a finite set of embedded discs in $U$ so that $D_i \cap \partial U=\partial D_i$. Let $a$ be an arc connecting two points of $\partial U$ which does not intersect any disc in $\{D_i\}$. Then the endpoints of $a$ are in the same component of $\partial U \setminus \bigcup_i \partial D_i$.
\end{lemma}
\begin{proof}
    Let $C_i= \partial D_i$. Then the collection $\{C_i\}$ decomposes $\partial U$ into a collection of subsurfaces. Suppose an endpoint of $a$ is contained in one of the subsurfaces which is a disc. This disc has boundary equal to the boundary component of some $D_i$. But $D_i$ cuts off a punctured ball (the punctures being some subcollection of ends of $U$), so it is clear that the other endpoint of $a$ must be in this disc. If no endpoint of $a$ is in such a disc, we can modify $U$ by taking each disc as above in $\partial U$ and removing the punctured ball that the corresponding $D_i$ cuts off. The resulting manifold is still a doubled handlebody with a single ball removed, but some ends have been filled in. We can inductively do this removal until some endpoint of $a$ is in a disc, and then proceed as above. 
\end{proof}

The next lemma is stated in full generality, though we will only need it when $\Gamma$ is a finite rank graph. We note the following essential fact that is used in the following lemma. A sphere $S\subset K$ for $K$ a submanifold of $M_{\Gamma}$ is nonseparating in $K$ if and only if there is a closed curve that lies in $K$ which intersects $S$ exactly one time.

\begin{lemma}\label{nonseparatingProjections}
    Let $\Gamma$ be a graph, and fix a sphere $S\in \mathcal{S}_{ns}(M_{\Gamma})$ and a full submanifold $K$ of $M_{\Gamma}$. Then $\pi_{K}(S)$ is nonempty and consists of nonseparating spheres in $K$. Further, every choice of surgery on a component of $S\cap K$ which results in an embedded sphere is also essential and non-peripheral.
\end{lemma}
\begin{proof}
    We may assume via an isotopy that $S$ is in normal form with respect to $\partial K$. As $S$ is nonseparating, we can take a closed curve $\gamma$ in $M_{\Gamma}$ which intersects it in one point. Further, we can assume by a homotopy of $\gamma$ that this intersection occurs in any fixed component $R$ of $S\cap K$, and that $\gamma$ has transverse intersection with $\partial K$. 
    \par 
    Our goal to replace $\gamma$ with a curve $\gamma'$ which lies in $K$ that intersects $R$ in the same point, and nowhere else. Once this is done, the result follows, as any sphere which results from surgery using the component $R$ will be nonseparating by the fact mentioned before the lemma. Let $U$ denote a complementary component of $K$, and suppose $a$ is a component of $\gamma \cap U$. If $S$ does not intersect $U$, it is clear we may modify $\gamma$ by replacing $a$ with another arc between the intersection points of $a$ with $\partial K$ with a small arc that lies entirely in $\partial K$. This new curve does not intersect $S$ at any new points.
    \par 
    Otherwise, suppose $S$ intersects $U$. By assumption $a$ does not intersect $S\cap U$. We want to show that the endpoints of $a$ lie in the same component of $\partial K\setminus S$. If we can do this, then we can modify $\gamma$ as above by replacing $a$ with an arc in $\partial K$ that does not intersect $S$ via connecting the endpoints of $a$ with an arc in the component of $\partial K \setminus S$ that they both lie in.
    \par 
    Let $\{D_i\}$ be the collection of components of $S\cap U$. As $U$ is simply connected, it follows that each $D_i$ is a disc with boundary in $\partial U$. To see this, recall that by the definition of normal form with respect to $\partial K$, $S$ is actually in normal form with respect to some maximal sphere system $\mathcal{P}$ containing $\partial K$. Consider the collection $\mathcal{P}_U$ which consists of the collection of elements of $\mathcal{P}$ in $U$, along with $\partial U$. Then $\mathcal{P}_U$ defines a tree $T(U)$, whose vertices are the components of $U\setminus \mathcal{P}_U$, with an edge between two vertices if they share a common element of $\mathcal{P}_U$ as a boundary component. The disc $D_i$ also has a well defined tree $T(D_i)$, whose vertices are components of $D_i\setminus \mathcal{P}_U$, and with an edge between two vertices if both intersect the same component of $\mathcal{P}_U$. 
    \par 
    There is a natural map $T(D_i)\to T(U)$ sending a component of $D_i \setminus \mathcal{P}_U$ to the component of $U\setminus \mathcal{P}_U$ containing it. From the definition of normal form, the map is locally injective. But a locally injective map of a tree into a tree is injective. In particular, $D_i$ intersects the component of $U\setminus \mathcal{P}_U$ containing $\partial U$ exactly once. It follows that $D_i$ is a disc.
    \par 
    But now, by Lemma \ref{lem:arcReplacement}, the endpoints of $a$ are in the same component of $\partial U \setminus S$, so we can replace $a$ with another arc as described above. Repeating this for all arcs of $\gamma$ that leave $K$, we obtain the curve $\gamma'$ as desired. 
\end{proof}

We have produced a map $\pi_{K}:\mathcal{S}_{ns}(M_{\Gamma})\to \mathcal{P}(\mathcal{S}_{ns}(K))\setminus \{\varnothing\}$, where the right hand side is the power set of $\mathcal{S}_{ns}(K)$ for $K$ a full submanifold of $M_{\Gamma}$, excluding the empty set. (Strictly speaking the map actually only inputs vertices of $\mathcal{S}_{ns}(M_{\Gamma})$, and not an arbitrary point). We now obtain the following Lipschitz bounds on this map.
\begin{lemma}\label{lem:LipschitzProjBounds}
     Fix $S_1, S_2\in \mathcal{S}_{ns}(M_{\Gamma})$ with $d_{\mathcal{S}_{ns}(M_{\Gamma})}(S_1, S_2)\leq 1$. If $K$ is any full submanifold, then $\mathrm{diam}_{\mathcal{S}_{ns}(K)}(\pi_K(S_1), \pi_K(S_2))\leq 3$. 
\end{lemma}
\begin{proof}
    Fix components $T_i\in \pi_K(S_i)$ for $i=1,2$. Then $T_i$ is the result of a surgery of a component $R_i$ of $S_i \cap K$ with discs in $\partial K$. As $S_1$ and $S_2$ are disjoint, $R_1$ and $R_2$ are as well. 
    \par 
    If $T_1$ and $T_2$ are not disjoint, let us consider all the components $\{C_{m}\}$ of $R_1\cap \partial K$ so that the corresponding surgery discs $\{D_m\}$ are not nested or disjoint from the surgery discs of $T_2$ (such discs are the only possible sources of intersections between $T_1$ and $T_2$). If $D_m'$ denotes the complement of $D_m$ in the component of $\partial K$ that $D_m$ lies in, then the sphere $T_1'\in \pi_K(S_1)$ given by surgering $R_1$ using the discs $\{D_m'\}$ (and keeping all other surgery discs the same, which still produces an embedded sphere as one can check) is disjoint from $T_2$, as by construction all the surgery discs of $T_1'$ are nested or disjoint from those of $T_2$. 
    \par 
    Thus to show the desired result it suffices to show that the diameter of the collection of spheres resulting from surgeries of $R_1$ with discs in $\partial K$ has diameter at most $2$. To do this, it suffices to find a nonseparating sphere in $K$ which is disjoint from $R_1$. But any fixed component of $\pi_K(S_1)$ is a nonseparating sphere, and as long as $\Gamma$ is not proper homotopy equivalent to a loop with a ray attached (which only has one nonseparating sphere and thus the result is trivial), then there is a distinct nonseparating sphere disjoint from this component, and thus from $R_1$.
\end{proof}
\par 

If $i:\mathcal{S}_{ns}(K)\to \mathcal{S}_{ns}(M_{\Gamma})$ denotes the inclusion of the first complex into the second, then $\pi_K(i(S))=\{S\}$. Using these maps, along with Subsection \ref{SphereGraph}, we obtain the following result.
\begin{theorem}\label{PositiveTranslationLength}
    For $\Gamma$ a graph with $1\leq \mathrm{rk}(\Gamma)<\infty$, where in the rank $1$ case we assume $\Gamma$ has more than one end, the complex $\mathcal{S}_{ns}(M_{\Gamma})$ has infinite diameter. Further, the action of $\mathrm{Map}(\Gamma)$ on $\mathcal{S}_{ns}(M_{\Gamma})$ admits elements with positive translation length.
\end{theorem}
\begin{proof}
    Note that $i$ is a $1$-Lipschitz map. On the other hand, for any full submanifold $K$ and any $S, S'\in \mathcal{S}(K)$,
    $$d_{\mathcal{S}_{ns}(K)}(S, S')=d_{\mathcal{S}_{ns}(K)}(\pi_{K}(i(S)), \pi_{K}(i(S')) \leq 3d_{\mathcal{S}_{ns}(M_{\Gamma})}(i(S), i(S'))$$
    where the inequality follows Lemma \ref{lem:LipschitzProjBounds}.
    \par 
    We will split up the remainder of the proof into two cases. First, when the rank of $\Gamma$ is at least $2$, and second, when its rank is equal to $1$.
    \par 
    To find elements with positive translation length when $\text{rk}(\Gamma)\geq 2$, we utilize Lemmas \ref{FreeFactorPositiveTranslation} and \ref{FreeFactorEquivariantMap}. Fix $K$ to be a full submanifold with a single boundary component. By Lemma \ref{FreeFactorPositiveTranslation}, we can fix an element $\psi \in \text{Aut}(F_n)$ which acts with positive translation length on $\mathcal{AF}_n$. 
    \par 
    Recall the complex $\widehat{\mathcal{S}}_{ns}(M_{\Gamma})$ defined before Lemma \ref{FreeFactorEquivariantMap}, as well as the $2$-Lipschitz map $\Phi:\widehat{\mathcal{S}}_{ns}^{(1)}(M_{n,1})\to \mathcal{AF}_n$. Given a vertex $S\in \widehat{\mathcal{S}}_{ns}(K)$, we obtain for all $m\geq 0$
    
    \begin{multline*}
        d_{\mathcal{AF}_n}(\psi^m(\Phi(S)), \Phi(S))\leq 2d_{\widehat{\mathcal{S}}_{ns}(K)}(\psi^m(S), S) \leq 4 d_{\mathcal{S}_{ns}(K)}(\psi^m(S), S)
        \\ 
        \leq 12d_{\mathcal{S}_{ns}(M_{\Gamma})}(\psi^m(S), S)
    \end{multline*}
    where the first inequality is Lemma \ref{FreeFactorEquivariantMap}(b), the second inequality is Lemma \ref{FreeFactorEquivariantMap}(a), and the third follows from the bi-Lipschitz lower bounds for $i$ established in the first set of inequalities above. As $\psi$ has positive translation length on $\mathcal{AF}_n$, it follows that the same is true on $\mathcal{S}_{ns}(M_{\Gamma})$.
    \par 
    Now suppose $\text{rk}(\Gamma)=1$. In this case, we take $K$ to be a full submanifold with $2$ boundary components. As discussed in Subsection \ref{SphereGraph}, $\mathcal{S}_{ns}(K)$ is isomorphic to the real line with vertices at the integers. The action of $\text{Map}(M_{1,2})$ has elements with positive translation length. For example, one can consider maps which push one of the boundary spheres around a nontrivial loop. Consider such a map $\psi$. Then, as 
    $$d_{\mathcal{S}_{ns}(K)}(\psi^m(S), S)
        \\ 
        \leq 3d_{\mathcal{S}_{ns}(M_{\Gamma})}(\psi^m(S), S)$$
        for any $S\in \mathcal{S}(K)$, it follows that there are elements of $\text{Map}(\Gamma)$ which have positive translation length for the action on $\mathcal{S}_{ns}(M_{\Gamma})$.
\end{proof}

We discuss briefly hyperbolicity for sphere complexes in the finite type case. As of the writing of this paper it is unknown if $\mathcal{S}_{ns}(M_{n,s})$ is hyperbolic for $s\geq 1$, except for in the special cases of $M_{1,1}$ where it is a point, and $M_{1,2}$ where it is a line.

For $s\leq 1$, we note that $\mathcal{S}(M_{n,s})$ is quasi-isometric to $\mathcal{S}_{ns}(M_{n,s})$. Indeed, suppose $S_1, S_2, S_3$ is a geodesic path of vertices in $\mathcal{S}(M_{n,s})$ with $S_1$ nonseparating, As $S_1$ and $S_3$ must intersect it follows that they lie in the same component of $M_{n,s}\setminus S_2$. If $S_2$ is separating, then it follows that, from the assumption that $s\leq 1$, there is a nonseparating sphere in the component not containing $S_1$ and $S_3$ of $M_{n,s}\setminus S_2$. Thus in the path we can replace $S_2$ with this nonseparating sphere. An inductive argument applying this idea shows that $\mathcal{S}_{ns}(M_{n,s})$ is actually isometrically embedded in $\mathcal{S}(M_{n,s})$. 
\par 
This argument fails for $s\geq 2$, as the step replacing $S_2$ doesn't work in general. This provides at least some evidence of the following conjecture. This conjecture has been asked in the closed case, see \cite{clay_uniform_2017} for example.
\begin{conjecture}\label{conj:NSHyperb}
    For all $s\geq 0$ and $n\geq 2$, $\mathcal{S}_{ns}(M_{n,s})$ is hyperbolic. Further, the hyperbolicity constant can be chosen to be uniform independent of $n$ and $s$.
\end{conjecture}

Recall that the corresponding conjecture for the full sphere complex $\mathcal{S}(M_{n,s})$ is trivial   if $s\geq 3$, as in this case $\mathcal{S}(M_{n,s})$ has finite diameter, and false when $n\geq 4$ and $s=2$ as $\mathcal{S}(M_{n,s})$ is not hyperbolic. See the discussion in Subsection \ref{SphereGraph}.
\par 
If one can mimic the results of Rasmussen which produces a uniform hyperbolicity constant for the nonseparating curve graphs of all finite type surfaces \cite{rasmussen_uniform_2020}, this conjecture along with Theorem \ref{PositiveTranslationLength} would imply the existence of loxodromic elements in the action of $\text{Map}(\Gamma)$ on $\mathcal{S}_{ns}(M_{\Gamma})$, which would be hyperbolic by the conjecture. Even a weaker result which gives a uniform hyperbolicity constant as the number of boundary components increases (but the rank is fixed) would be sufficient.
\par 
Note that regardless of the validity of this conjecture, it still follows from Theorem \ref{PositiveTranslationLength} that $\mathcal{S}_{ns}(M_{n,s})$ has infinite diameter and there are elements in $\text{Map}(\Gamma_{n,s})$ with positive translation length. Here, $\Gamma_{n,s}$ is the graph of rank $n$ with $s$ rays attached. 
\par 

\subsection{Outer space for locally finite graphs}\label{subsec:OuterSpace}
Let us fix a locally finite graph $\Gamma$ of finite rank. Outer space for compact graphs was initially defined by Culler--Vogtmann in \cite{culler_moduli_1986}, and has been used to great effect in the study of $\text{Out}(F_n)$. In the Appendix of \cite{hatcher_homological_1995}, it is shown that Outer space of a compact graph can be defined via sphere systems. We will mimic this definition to define the Outer space of $\Gamma$. Let $\mathcal{S}_{\infty}(M_{\Gamma})$ be the subcomplex of $\mathcal{S}(M_{\Gamma})$ so that the complements of its sphere systems have at least one non-simply connected component. We define $\mathbb{O}(\Gamma)=\mathcal{S}(M_{\Gamma})\setminus \mathcal{S}_{\infty}(M_{\Gamma})$ to be the \textit{Outer space} of $\Gamma$. One can think of the points of $\mathbb{O}(\Gamma)$ as weighted finite sphere systems, after identifying the $k$-simplices of $\mathcal{S}(M_{\Gamma})$ with the standard $k$-simplex in $\R^{k+1}$, for all $k$.
\par 
We have the following facts about the structure of $\mathbb{O}(\Gamma)$ and the action of $\text{Map}(\Gamma)$ on it.
\begin{proposition}\label{prop:OuterSpaceContrStab}
    Let $\Gamma$ be a locally finite graph with finite positive rank. Then $\mathbb{O}(\Gamma)$ is contractible. The point stabilizers of the action of $\mathrm{PMap}(\Gamma)$ on $\mathbb{O}(\Gamma)$ are subgroups are finite. 
\end{proposition}
\begin{proof}
     Contractibility follows from a similar argument to Hatcher's proof of the contractibility of the Outer space of a compact graph in \cite{hatcher_homological_1995}. See Proposition 2.1 and the discussion in the Appendix after the proposition there. Namely, fix a maximal sphere system $\mathcal{P}$. Given a point $p\in \mathbb{O}(\Gamma)$, we can first isotope it so that it is in normal form with respect to $\mathcal{P}$. We then surger the spheres of $p$ in a continuous way with respect to the weights along the spheres in $\mathcal{P}$ until it becomes a finite subset of $\mathcal{P}$. The precise details of this process is described in Proposition 2.1 of \cite{hatcher_homological_1995}. By the discussion in the appendix, the surgeries described preserve the simple connectivity of the complementary components of the sphere systems. The only minor difference is that here we are allowing punctures, so sometimes we have to throw away spheres containing on one side a puncture of $M_{\Gamma}$. This still preserves simple connectivity of the complementary components. Thus, the families arising from these surgeries are nonempty and lie in $\mathbb{O}(\Gamma)$. The continuity of the resulting deformation onto the infinite simplex defined by $\mathcal{P}$ follows from the same local argument as in Proposition 2.1 of \cite{hatcher_homological_1995}. This infinite simplex is contractible, and thus $\mathbb{O}(\Gamma)$ is as well.
    \par 
    For point stabilizers of elements pf $\mathbb{O}(\Gamma)$ in $\text{PMap}(\Gamma)$, note that any $\mathcal{S}\in \mathbb{O}(\Gamma)$ divides $M_{\Gamma}$ into finitely many components, all of which are simply connected and with finitely many sphere boundary components. By Theorems \ref{SESMainTheorem} and \ref{ADMQGraphClassification}, the subgroup of the stabilizer of $\mathcal{S}$ which stabilizes each complementary component is finite, as by the theorems, diffeomorphisms on these components are determined up to isotopy and sphere twists by how they act on the boundary spheres and ends of the component. But by assumption, as we are working in the pure mapping class group, the ends are fixed, so even noncompact complementary components of $\mathcal{S}$ contribute only finitely many elements. As there are only finitely many complementary components, it follows that the stabilizer of $\mathcal{S}$ is also finite.
\end{proof}
One can also consider stabilizers of the action of $\text{Map}(\Gamma)$ on $\mathbb{O}(\Gamma)$. Typically, these will be infinite as they will include mapping classes which stabilize some noncompact component of the complement of the sphere system, but acts by a nontrivial homeomorphism on the space of ends of the component.

\par
An unfortunate thing about the above definition for $\mathbb{O}(\Gamma)$ is that it doesn't "see" the entire graph. That is, for most $\Gamma$, many of the complementary components of a sphere system defining a point in $\mathbb{O}(\Gamma)$ will be of infinite type. In particular, if one were to try to think of points of Outer space as some sort of marked metric structures on $\Gamma$, this definition would fail to see any metric behavior of $\Gamma$ at infinity.


\section{Coarse geometry of (pure) mapping class groups}\label{sec:CoarseGeo}

In this final section we explore the coarse geometry of mapping class groups of graphs and their doubled handlebodies. We begin in Subsection \ref{subsec:CoarseGeoPure} with a simple collection of applications of the results of Domat--Hoganson--Kwak in \cite{domat_coarse_2023} and \cite{domat_generating_2023}. In Subsection \ref{subsec:CoarseGeoTranslatable} we extend the ideas of Schaffer-Cohen that appear in \cite{schaffer-cohen_graphs_2024}, which explores a particular class of infinite type surfaces, to infinite type graphs and doubled handlebodies. Recall that in Corollary \ref{cor:doubledHandlebodyPolish} we showed that $\text{PMap}(M_{\Gamma})$ is Polish. We will use this implicitly.

\subsection{Classifying the coarse geometry of pure mapping class groups}\label{subsec:CoarseGeoPure}
Fix a locally finite graph $\Gamma$, which for simplicity we assume is in standard form. In this subsection we show how to use the results of \cite{domat_coarse_2023} and \cite{domat_generating_2023} to provide a complete CB classification of the pure mapping class groups of doubled handlebodies. We will also note a variety of other results which can be extracted from these two papers. 
\par 
We split the classification into three separate statements, which appear in Corollaries \ref{CBPMap}, \ref{LocallyCBPMap}, and \ref{CBGeneratedPMap}. These three results exactly mirror Theorems A and D of \cite{domat_coarse_2023}, and Theorem A of \cite{domat_generating_2023}, respectively. 
\par 
In the following propositions, we let $\Gamma_c$ denote the \textit{core} of $\Gamma$, which is the smallest graph containing all immersed loops in $\Gamma$. 

\begin{lemma}
    Fix a subgroup $G\leq \mathrm{Map}(M_{\Gamma})$, and write $Q=\Psi(G)\leq \mathrm{Map}(\Gamma)$ and $K=\mathrm{Twists}(M_{\Gamma})\cap G$. Assume $K$ is CB in $G$. Then
    \begin{enumerate}
        \item $G$ is CB if and only if $Q$ is CB.
        \item $G$ is locally CB if and only if $Q$ is locally CB.
        \item $G$ is CB generated if and only if $Q$ is CB generated. Further, $\Psi|_G:G \to Q$ is a quasi-isometry from $G$ and $Q$ with word metrics arising from CB generating sets.
    \end{enumerate}
\end{lemma}
\begin{proof}
We can give a proof of the three if and only if results using the splitting given by Theorem \ref{SESMainTheorem}. Namely, we can write
\begin{equation*}
    G=K\rtimes Q
\end{equation*}
and note that the product of a CB set with a CB set is CB. Thus if $A$ is a CB subset of $Q$ (and thus $G$), then $KA$ is a CB subset of $G$. In particular, this shows that if $Q$ has any of the three properties (CB, locally CB, or CB generated), then $G$ must have the corresponding property. 
\par 
On the other hand, if $C\subset G$ is CB in $G$, then $C\cap Q$ is CB in $Q$. This follows as any continuous group action of $Q$ on a metric space by isometries gives such an action for $G$, and if $C \cap H$ has unbounded orbits then so does $C$. The results then quickly follow. Note in the locally CB case that $C\cap Q$ is an open subset of $Q$ if $C$ is open in $G$.
\par 
Now, if $G$ and $Q$ are CB generated, then by Lemma \ref{lem:milnorSchwartzRosendal}, to show that $\Psi|_G$ is a quasi-isometry it suffices to show that is coarsely proper. But this follows from the exact same arguments as above, as $\Psi|_G^{-1}(A)=KA$.
\end{proof}

\begin{corollary}
    The group $\mathrm{PMap}(M_{\Gamma})$ is CB generated if and only if $\mathrm{PMap}(\Gamma)$ is. If both are CB generated and given word metrics arising from CB generating sets, then $\Psi_P$ is a quasi-isometry.
\end{corollary}

One could give a more direct proof of the second part of this corollary by giving an explicit generating set for $\text{PMap}(M_{\Gamma})$ which mimics that given for $\text{PMap}(\Gamma)$ in \cite{domat_generating_2023}, which we can effectively lift to $\text{PMap}(M_{\Gamma})$ using the splitting in Theorem \ref{SESMainTheorem}. 
\par 
In particular, we obtain the following three results on the CB properties of $\mathrm{PMap}(M_{\Gamma})$ as corollaries of this and Theorems A and D in \cite{domat_coarse_2023}, as well as Theorem A of \cite{domat_generating_2023}.

\begin{corollary}\label{CBPMap}
    The group $\mathrm{PMap}(M_{\Gamma})$ is CB if and only if $\Gamma$ is of one of the following forms.
    \begin{enumerate}
        \item $\Gamma$ is a tree.
        \item $\Gamma$ is a rank $1$ graph with a single ray attached.
        \item $|E_{\ell}(\Gamma)|=1$ and $E(\Gamma)\setminus E_{\ell}(\Gamma)$ is discrete. 
    \end{enumerate}
\end{corollary}

\begin{corollary}\label{LocallyCBPMap}
    The group $\mathrm{PMap}(M_{\Gamma})$ is locally CB if and only if $\Gamma$ is one of the following forms.
    \begin{enumerate}
        \item $\Gamma$ has finite rank.
        \item $\Gamma$ satisfies both $|E_{\ell}(\Gamma)|<\infty$ and only finitely many components of $\Gamma \setminus \Gamma_c$ have an infinite end space. 
    \end{enumerate}
\end{corollary}

\begin{corollary}\label{CBGeneratedPMap}
    The group $\mathrm{PMap}(M_{\Gamma})$ is CB generated if and only if $\Gamma$ is a tree or $\Gamma$ satisfies both
    \begin{enumerate}
        \item $|E_{\ell}(\Gamma)|< \infty$, and
        \item $E(\Gamma)\setminus E_{\ell}(\Gamma)$ has no accumulation points.
    \end{enumerate}
\end{corollary}

We also have the following consequence.

\begin{corollary}\label{AsympDim}
    For any locally finite graph $\Gamma$, if $\mathrm{Map}(M_{\Gamma})$ is locally CB, then is has a well defined coarse equivalence type (i.e. it admits continuous left-invariance coarsely proper pseudo-metrics). The same statement holds for $\mathrm{PMap}(M_{\Gamma})$, and $\Psi$ and $\Psi_P$ are coarse equivalences. Hence, $\mathrm{asdim}(\mathrm{Map}(\Gamma))=\mathrm{asdim}(\mathrm{Map}(\Gamma))$ and $\mathrm{asdim}(\mathrm{PMap}(\Gamma))=\mathrm{asdim}(\mathrm{PMap}(\Gamma))$. Thus,
    \begin{equation*}
        \mathrm{asdim}(\mathrm{PMap}(M_{\Gamma}))=\begin{cases}
            0 & \mathrm{if } |E_{\ell}(\Gamma)|\leq 1
            \\
            \infty & \mathrm{if } |E_{\ell}(\Gamma)|\geq 2
        \end{cases}
    \end{equation*}
\end{corollary}
\begin{proof}
    We can apply Proposition \ref{locallyCBCoarseType} to $\mathrm{Map}(M_{\Gamma})$ and $\mathrm{PMap}(M_{\Gamma})$ to obtain the first statement (Proposition \ref{locallyCBCoarseType} shows that $\text{Map}(\Gamma)$ and $\text{PMap}(\Gamma)$ have a well defined coarse equivalence type as well, as was noted in Proposition 2.22 of \cite{domat_coarse_2023} for $\text{PMap}(M_{\Gamma})$). Then Theorem \ref{SESMainTheorem} and Proposition \ref{CBKernelCoarseEquiv} imply that $\Psi$ and $\Psi_P$ are coarse equivalences. Lemma \ref{CoarseEquivSameAsDim} implies the equalities between asymptotic dimensions. The values for $\text{asdim}(\text{PMap}(M_{\Gamma}))$ follow from Theorem E in \cite{domat_coarse_2023}.
\end{proof}

For completeness, we also note the following algebraic facts about $\text{PMap}(M_{\Gamma})$ using other results in \cite{domat_generating_2023}. 

\begin{corollary}
    For every locally finite graph $\Gamma$,
    \begin{equation*}
       \mathrm{rk}(H^1(\mathrm{PMap}(M_{\Gamma})); \Z)= \begin{cases}
           0 & \mathrm{if } |E_{\ell}(\Gamma)|\leq 1 \\
           n-1 & 2\leq |E_{\ell}(\Gamma)|=n<\infty\\
            \aleph_0 & \mathrm{otherwise.}
        \end{cases}
    \end{equation*}
\end{corollary}
\begin{proof}
    As every element of $\mathrm{Twists}(M_{\Gamma})$ has order $2$, all homomorphisms of $\mathrm{PMap}(M_{\Gamma})$ to $\Z$ correspond to a unique homomorphism from $\mathrm{PMap}(\Gamma)$ to $\Z$. The result then follows from Corollary C in \cite{domat_generating_2023}. 
\end{proof}

\begin{corollary}
    $\mathrm{PMap}(M_{\Gamma})$ is residually finite if and only if $\Gamma$ has finite rank.
\end{corollary}
\begin{proof}
    We can write $\mathrm{PMap}(M_{\Gamma})=\mathrm{Twists}(M_{\Gamma})\rtimes \mathrm{PMap}(\Gamma)$ by Theorem \ref{SESMainTheorem}. If the rank of $\Gamma$ is infinite, by Theorem D of \cite{domat_generating_2023} $\mathrm{PMap}(\Gamma)$ is not residually finite, so $\text{PMap}(M_{\Gamma})$ is not either. 
    \par 
    If the rank of $\Gamma$ is finite, then $\text{PMap}(M_{\Gamma})$ contains a finite index residually finite subgroup, namely $\text{PMap}(\Gamma)$, by Theorem \ref{KernelStructure}. Thus it is residually finite.
\end{proof}

The proof of the following is essentially identical to the previous proof. We say that a group $G$ satisfies the \textit{Tits alternative} if every subgroup of $G$ is either virtually solvable or contains a nonabelian free group.
\begin{corollary}
    $\mathrm{PMap}(M_{\Gamma})$ satisfies the Tits Alternative if and only if $\Gamma$ has finite rank.
\end{corollary}
\begin{proof}
    If $\Gamma$ has infinite rank, then $\text{PMap}(M_{\Gamma})$ contains a subgroup which fails the Tits alternative, namely $\text{PMap}(\Gamma)$ (by Theorem E of \cite{domat_generating_2023}), and thus $\text{PMap}(M_{\Gamma})$ also does not satisfy the Tits alternative.
    \par 
    If $\Gamma$ has finite rank, then $\text{PMap}(M_{\Gamma})$ contains a finite index subgroup $\text{PMap}(\Gamma)$ which satisfies the Tits alternative. It is easy to see that the Tits alternative is preserved when passing to finite index supergroups, so $\text{PMap}(M_{\Gamma})$ satisfies the Tits alternative as well.
\end{proof}

\par

\subsection{Translatable graphs}\label{subsec:CoarseGeoTranslatable}
In this subsection we note some modifications of the techniques of Schaffer-Cohen to produce results about the coarse geometry of mapping class groups of a particular collection of graphs \cite{schaffer-cohen_graphs_2024}. Namely, we will produce a collection of graphs whose mapping class groups are quasi-isometric to a particular subgraph of the $1$-skeleton of the sphere complex $\mathcal{S}(M_{\Gamma})$, see Theorem \ref{thm:TranslCBGen}.
\par 
Below, we will sketch the key results of interest from \cite{schaffer-cohen_graphs_2024} in the language of spheres, doubled handlebodies, and graphs. Very little work needs to be done to do this. Below we will summarize some of the ideas and remark when anything different needs to be done.
\par 
We first note a general result which is analogous to the classical Milnor--Schwartz lemma in the context of CB generation. One can compare this to Lemma \ref{lem:milnorSchwartzRosendal}.

\begin{lemma}[{\cite[Lemma 2.5]{schaffer-cohen_graphs_2024}}]\label{lem:MilnorSchwartz}
    Let $G$ be a locally coarsely bounded group acting transitively by isometries on a connected graph $\Gamma$ equipped with the edge metric. Suppose for some vertex $v_0$, the set
    \begin{equation*}
        A=\{f\in G \ | \ d(v_0, f(v_0))\leq 1 \}
    \end{equation*}
    generates $G$ and is coarsely bounded. Then orbit map of $G$ in $\Gamma$ with the word metric given by $A$ is a quasi-isometry.
\end{lemma}

We now define the collection of graphs and doubled handlebodies that we will consider here.

\begin{definition} [{\cite[Definition 3.1]{schaffer-cohen_graphs_2024}}]\label{def:SphereDivergence}
    Given an infinite type graph $\Gamma$, an end $e$ of $M_{\Gamma}$, and a sequence $S_1, S_2,\ldots$ of spheres in $M_{\Gamma}$, we say that $\lim_{n\to\infty} S_n =e$ if, for every neighborhood $V$ of the end $e$ in $M_{\Gamma}$, all but finitely many of the $S_i$'s are (up to isotopy) contained in $V$.
\end{definition}

\begin{definition} [{\cite[Definition 3.2]{schaffer-cohen_graphs_2024}}]\label{def:Transl}
    Given an infinite type graph $\Gamma$, a map $h\in \text{Map}(M_{\Gamma})$ is called a \textit{translation} if there are two distinct ends $e_-$ and $e_+$ of $M_{\Gamma}$ so that for any sphere $S$ of $M_{\Gamma}$, $\lim_{n \to \infty} h^n(S)=e_+$ and $\lim_{n\to \infty} h^{-n}(S)=e_-$. If such a translation exists, we call $M_{\Gamma}$, as well as $\Gamma$, \textit{translatable}.
\end{definition}

\begin{figure}
    \centering
    \includegraphics[scale=.5]{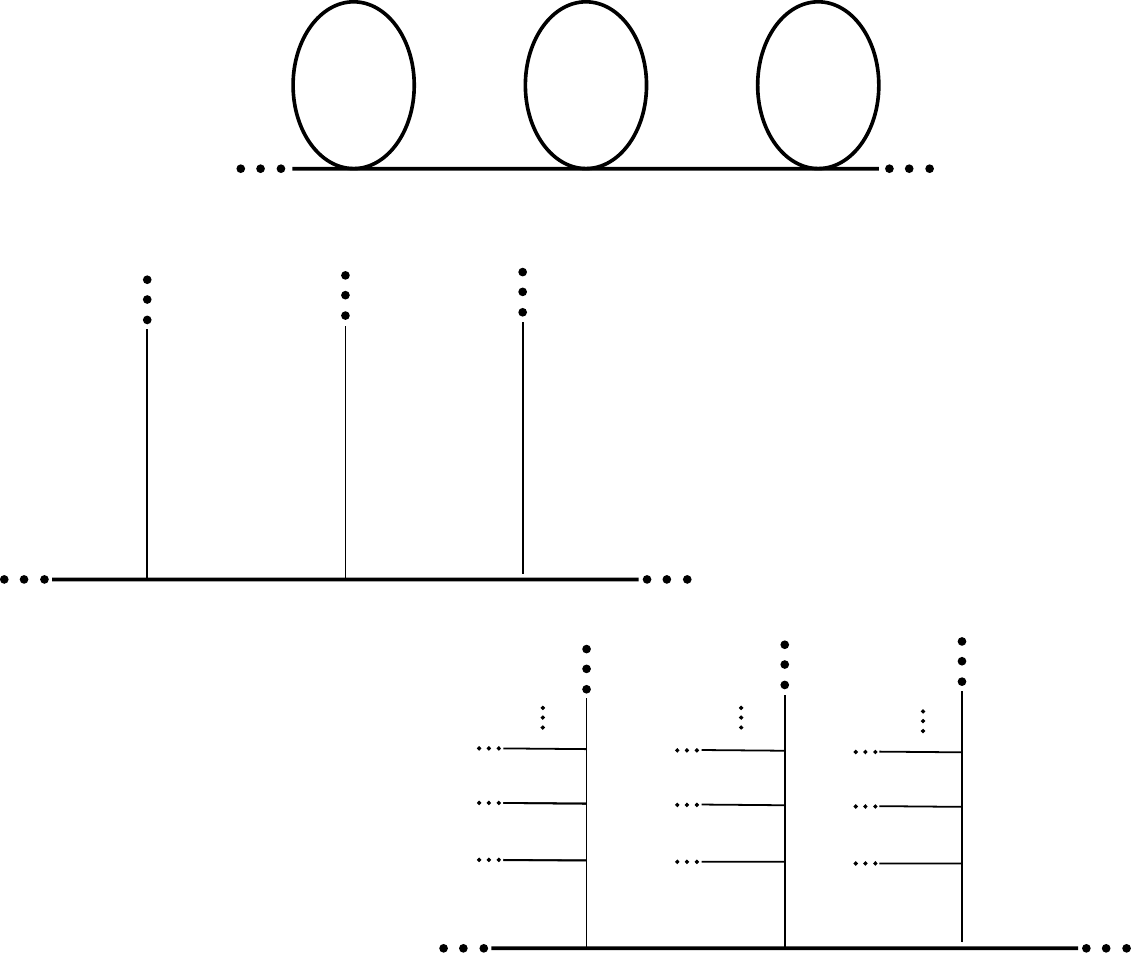}
    \caption{Some examples of translatable graphs. Note that a variation of Proposition 3.5 of \cite{schaffer-cohen_graphs_2024} implies that all translatable graphs look like this, i.e. up to proper homotopy they consist of an infinite wedge sum of some fixed graph with two distinguished vertices, with the wedge summing being arranged like $\Z$.}
    \label{fig:TranslExamples}
\end{figure}

\par 

We wish to consider a certain subgraph of $\mathcal{S}(M_{\Gamma})$ whose vertices separate the two ends $e_+$ and $e_-$, and with edges defined via a certain cobounding condition, described in Definition \ref{def:TranslSphereGraph}. Note that $\text{Map}(M_{\Gamma})$ acts transitively on this collection of spheres (see Lemma 3.6 in \cite{schaffer-cohen_graphs_2024}) which will allow for the use of Lemma \ref{lem:MilnorSchwartz}. Further, this is true even if we restrict to the subgroup fixing $e_-$ and $e_+$. We have the following result which is the first step in being able to apply Lemma \ref{lem:MilnorSchwartz}.

\begin{proposition}[{\cite[Corollary 3.9]{schaffer-cohen_graphs_2024}}]\label{lem:CBStabilizer}
    Let $\Gamma$ be a translatable graph. Fix a sphere $S\in \mathcal{S}(M_{\Gamma})$ separating $e_+$ and $e_-$. Then the subgroups
    \begin{equation*}
        H_{M_{\Gamma}}=\{f\in \mathrm{Map}(M_{\Gamma}) \ | \ f(S)=S\}
    \end{equation*}
    \begin{equation*}
        H_{\Gamma}=\{f\in \mathrm{Map}(\Gamma) \ | \ f(S)=S\}
    \end{equation*}
    are coarsely bounded in their respective groups. 
\end{proposition}
\begin{proof}
    The proof for $\text{Map}(M_{\Gamma})$ is identical to that given in Corollary 3.9 of \cite{schaffer-cohen_graphs_2024}. As $H_{\Gamma}$ is the $\Psi$ image of $H_{M_{\Gamma}}$ since $\Psi$ respects the action of the two groups on $\mathcal{S}(M_{\Gamma})$, it follows that $H_{\Gamma}$ is CB in $\mathrm{Map}(\Gamma)$ too.
    
\end{proof}

From this point on, we need to make some restrictions on the end space of the graphs we are considering. We recall the preorder of Mann--Rafi on end spaces \cite{mann_large-scale_2023}. Recall that a preorder $\preceq$ on a set $X$ is a relation such that for all $x\in X$, $x\preceq x$, and if $x,y,z\in X$ so that $x\preceq y$ and $y\preceq z$, then $x\preceq z$.

\begin{definition}[{\cite[Definition 4.1]{mann_large-scale_2023}}]\label{def:preorder}
    Given two points $x, y \in E(\Gamma)$, we say that $y \preceq x$ if for every neighborhood $U$ of $x$ there is a neighborhood $V$ of $y$ such that $f(V)\subset U$ for some $f\in \text{Map}(\Gamma)$.
\end{definition}

One can equivalently define this using $\text{Map}(M_{\Gamma})$ instead. They give equivalent orders by Theorem \ref{ADMQGraphClassification} and Proposition \ref{RichardsHandlebody}.

\begin{definition}[{\cite[Definition 4.14]{mann_large-scale_2023}}]\label{def:stableends}
    Given an end $e\in E(\Gamma)$, a neighborhood $U$ of $e$ is \textit{stable} if for all neighborhoods $U' \subset U$ of $e$, $U'$ contains a homeomorphic copy of $U$. We say that $e$ is \textit{stable} if it has a stable neighborhood.
\end{definition}
 \par 
By Proposition 4.7 of \cite{mann_large-scale_2023}, the preorder in Definition \ref{def:preorder} has maximal elements. 

\begin{definition}[{\cite[Definition 2.15]{schaffer-cohen_graphs_2024}}]\label{def:tame}
    We say that $E(\Gamma)$ is \textit{tame} if all of its maximal ends and the immediate predecessors of the maximal ends are stable. 
\end{definition}

Tameness will be precisely the extra condition needed to be able to pick a good choice of edge set for the spheres we have been considering (that is, so that we get a connected graph which is often infinite diameter). 
\par 
Before this, we make the following general definition, and then discuss what is done to make the correct choices. Given two disjoint spheres $S$ and $T$ both separating $e_+$ and $e_-$ with $T$ in the complementary component of $S$ containing $e_+$, we let $(S,T)$ denote the doubled handlebody with two boundary components $S$ and $T$ which is co-bounded by these two spheres. 

\begin{definition}[{\cite[Definition 4.1]{schaffer-cohen_graphs_2024}}]\label{def:TranslSphereGraph}
    Fix a collection $\mathcal{M}$  of submanifolds of $M_{\Gamma}$ each of which is homeomorphic to a doubled handlebody with two boundary components. The \textit{translatable sphere graph} $\mathcal{TS}(M_{\Gamma}, \mathcal{M})$ of $M_{\Gamma}$ with respect to $\mathcal{M}$ is the graph whose vertices are spheres separating $e_+$ and $e_-$, with an edge between two spheres $S$ and $T$ if they are disjoint and $(S, T)$ or $(T,S)$ is homeomorphic to some $M\in \mathcal{M}$. 
\end{definition}

One could also consider the translatable sphere complex, though to be consistent with \cite{schaffer-cohen_graphs_2024} we use the graph instead.
\par
To make a good choice for $\mathcal{M}$, tameness is required. See Lemma 4.2 through Lemma 4.5 in \cite{schaffer-cohen_graphs_2024}. For surfaces, this choice is made entirely by considering the end space, including considerations about genus. We can do the same thing here, replacing genus with rank of the submanifold. As we will not need to work with the collection $\mathcal{M}$ directly, we will not define it here. See the discussion after Lemma 4.5 in \cite{schaffer-cohen_graphs_2024} for more details.
\par 
Following the convention in \cite{schaffer-cohen_graphs_2024}, we denote the graph $\mathcal{TS}(M_{\Gamma}, \mathcal{M})$ by $\mathcal{TS}(M_{\Gamma})$, where $\mathcal{M}$ is the collection discussed above. We obtain the following.

\begin{lemma}[{\cite[Lemma 4.7]{schaffer-cohen_graphs_2024}}]\label{lem:TranslSphereConnected}
    The graph $\mathcal{TS}(M_{\Gamma})$ is connected.
\end{lemma}
\begin{proof}
    This essentially follows from the arguments given in the proof of Lemma 4.7 in \cite{schaffer-cohen_graphs_2024}. The only thing we note is that the surgeries as depicted in Figures 6 and 7 in \cite{schaffer-cohen_graphs_2024} which induce the desired separation properties can be done for spheres in $\mathcal{TS}(M_{\Gamma})$. This allows one to prove an identical version of Lemma 4.6 in \cite{schaffer-cohen_graphs_2024}. The argument of Lemma 4.7 of \cite{schaffer-cohen_graphs_2024} happens entirely in the end space, except for when Lemma 4.6 of \cite{schaffer-cohen_graphs_2024} needs to be invoked to produce spheres which realize certain partitions of $E(\Gamma)$, sometimes with rank considerations as well.
\end{proof}

We now obtain the final piece needed to apply Lemma \ref{lem:MilnorSchwartz}.

\begin{lemma}[{\cite[Lemma 4.8]{schaffer-cohen_graphs_2024}}]\label{lem:CoarseStabCB}
    Let $\Gamma$ be a translatable graph with tame end space. Fix a sphere $S$ separating $e_+$ and $e_-$. Then the subgroups
    \begin{equation*}
        H'_{M_{\Gamma}}=\{f\in \mathrm{Map}(M_{\Gamma}) \ | \ d_{\mathcal{TS}(M_{\Gamma})}(f(S),S)\leq 1\}
    \end{equation*}
    \begin{equation*}
        H'_{\Gamma}=\{f\in \mathrm{Map}(\Gamma) \ | \ d_{\mathcal{TS}(M_{\Gamma})}(f(S),S)\leq 1\}
    \end{equation*}
    are coarsely bounded in their respective groups. 
\end{lemma}
\begin{proof}
    Just as in the proof of connectivity, essentially all the work is done in the end space and realizing certain maps defined on the end space as homeomorphisms of $M_{\Gamma}$ (which we can do by Proposition \ref{RichardsHandlebody}). In certain cases one needs be able to shift handles, and we can do this by applying cylinder shifts. As $H_{\Gamma}$ is the $\Psi$ image of $H_{M_{\Gamma}}$, the result follows.
\end{proof}

We thus can conclude the following, which is an analog of Theorem 4.9 of \cite{schaffer-cohen_graphs_2024}.

\begin{theorem}\label{thm:TranslCBGen}
    Suppose $\Gamma$ is a translatable graph that has a tame end space. The groups $\mathrm{Map}(\Gamma)$ and $\mathrm{Map}(M_{\Gamma})$ are CB generated and, equipped with a CB generating set and the associated word metric, are equivariantly quasi-isometric to $\mathcal{TS}(M_{\Gamma})$.
\end{theorem}
\begin{proof}
    By combining the transitivity of the action of these groups on $\mathcal{TS}(M_{\Gamma})$ with Lemmas \ref{lem:TranslSphereConnected} and \ref{lem:CoarseStabCB}, we can apply Lemma \ref{lem:MilnorSchwartz} to obtain the desired result.
\end{proof}

We also note the following consequence of the ideas we have discussed so far. Let $\mathcal{H}(\Sigma)$ denote the \textit{handlebody group} of the surface $\Sigma$ associated to a locally finite graph $\Gamma$. This is the subgroup of $\text{Map}(\Sigma)$, the mapping class group of $\Sigma$, which preserves the meridians of $\Sigma$, i.e. curves which bound a disc in $B$, the handlebody of $\Gamma$. This can be identified with the mapping class group of $B$, where the boundary is not assumed to be fixed.
\par 
Assume that $\Gamma$ is translatable and tame. The translatable curve graph $\mathcal{TC}(\Sigma)$ is defined in the same way as $\mathcal{TS}(M_{\Gamma})$, using submanifolds with the same end space and genus assumptions as the collection $\mathcal{M}$. We can consider the full subgraph of $\mathcal{TC}(\Sigma)$ whose vertices consist of meridians, i.e. curves which bound a disc in the handlebody $B$. We denote this graph by $\mathcal{TD}(B)$. 
\par 
All the work done so far for $\text{Map}(M_{\Gamma})$ and $\mathcal{TS}(M_{\Gamma})$ can be redone verbatim for $\mathcal{H}(\Sigma)$ and $\mathcal{TD}(B)$. Thus, under the same hypotheses as Theorem \ref{thm:TranslCBGen}, $\mathcal{H}(\Sigma)$ is CB generated, and with the associated word metric given by a CB generating set it is quasi-isometric to $\mathcal{TD}(B)$. 
\par 
On the other hand, $\mathcal{TD}(B)$ embeds as a subgraph of both $\mathcal{TC}(\Sigma)$ and $\mathcal{TS}(M_{\Gamma})$ (it embeds into the latter by doubling a disc in $B$ along its boundary to be a sphere in $M_{\Gamma}$). We then ask the following natural question.



\begin{question}
    Are there any examples of translatable graphs $\Gamma$ so that $\mathcal{TD}(B)$ is quasi-isometrically embeds into to $\mathcal{TC}(\Sigma)$ or $\mathcal{TS}(M_{\Gamma})$, and so that all three graphs have infinite diameter? What about examples where the inclusion maps are quasi-isometries?
\end{question}

If such a $\Gamma$ existed so that both inclusion maps are quasi-isometries, then this would give an example of a mapping class group of a graph and a mapping class group of a surface which are quasi-isometric so that neither are quasi-isometric to a point.

\vskip 10pt

\subsection{Further questions on coarse geometry}
Very little is known about the coarse geometry of the full mapping class group $\text{Map}(\Gamma)$ for general $\Gamma$. For finite ended graphs we know that they are all CB generated by Theorem A in \cite{domat_generating_2023}, but beyond that no classification work has been done.
\par 
From Theorem \ref{thm:TranslCBGen} one can extract from the work above many examples where $\text{Map}(\Gamma)$ is CB generated but not CB. For example, as long as at least one submanifold in $\mathcal{M}$ has a unique maximal end (as opposed to a Cantor set of maximal ends), or when $\Gamma$ has infinite rank but every noncompact submanifold in $\mathcal{M}$ is simply connected, it is easy to see that any translation of $\mathrm{Map}(\Gamma)$ has positive translation length in $\mathcal{TS}(M_{\Gamma})$.
\par 
One might hope that results of Mann--Rafi can be replicated for graphs to some degree \cite{mann_large-scale_2023}. It seems likely that many of the results that hold for surfaces also hold for graphs and their doubled handlebodies, but potentially only in one direction. For example, if $\Gamma$ has $0$ or infinite rank and the end space is self-similar in the sense of \cite{mann_large-scale_2023}, one should be able to repeat the proof of Proposition 3.1 in that paper to show that $\mathrm{Map}(M_{\Gamma})$, and thus $\mathrm{Map}(\Gamma)$, are CB. 
\par 
On the other hand, the ideas of nondisplacable subsurfaces seem much more subtle to apply. Indeed, by Theorem 1.1 of \cite{hill_large-scale_2023} it follows that if $\Sigma$ is a Loch Ness monster surface (i.e. a surface with one end accumulated by genus) with a nonzero number of punctures, then its mapping class group is not locally CB. On the other hand, Corollary \ref{CBPMap} shows that the mapping class group of the Loch Ness monster doubled handlebody with a nonzero number of punctures is CB. The failure of this technique seem to be related to the following two issues. First, defining interesting projections of spheres is hard. In Subsection \ref{nonseparating Sphere Finite Rank} we were only able to do it in a very special case using nonseparating spheres projecting to a particular type of submanifold. Second, even if one could define a reasonable notion of projection in a wider variety of contexts, the fact that $\mathcal{S}(M_{n,s})$ has finite diameter for $s\geq 3$ is a roadblock in producing unbounded length functions analogous to those given in Section 2 of \cite{mann_large-scale_2023}. 
\par 
We note one other potential application. The work of Grant--Rafi--Verberne in \cite{grant_asymptotic_2021} shows that the asymptotic dimension of many mapping class groups of surfaces is infinite. This work should translate naturally to arguments in doubled handlebodies. Indeed, the core of their arguments lie in counting ends and genus (to "count genus", they utilize homology with $\Z/2$ coefficients), much of which seems repeatable in doubled handlebodies. This would in particular produce examples of graphs $\Gamma$ where $|E_{\ell}(\Gamma)|\leq 1$ whose mapping class group has infinite asymptotic dimension. This is in contrast to Theorem E of \cite{domat_coarse_2023} which implies that the pure mapping class group of any such graph has asymptotic dimension $0$. 
\par 
On the other hand to all of this discussion, curve graphs are not sufficient to study much of the behavior of mapping class groups of infinite type surfaces. Much work has been done to study graphs of arcs, i.e. graphs whose vertices consist of proper embeddings of $\R$ into the surface. For example, the omnipresent arc graph of Fanoni--Ghaswala--McLeay and the grand arc graph of Bar-Natan--Verberne, which in many cases are known to be infinite diameter and hyperbolic \cite{fanoni_homeomorphic_2021}\cite{bar-natan_grand_2023}. It is natural to ask if there are analogs for mapping class groups of graphs.

\begin{question}
    Is there a geometrically significant graph, analogous to the omnipresent arc graph or the grand arc graph, that $\text{Map}(\Gamma)$ acts on? 
\end{question}

\printbibliography
\end{document}